\documentclass[reqno,12pt]{amsart}
\usepackage{amssymb,amsmath,amsthm,amsxtra,mathrsfs,calc,bm, color}
\usepackage{enumerate}
\usepackage[margin=1in]{geometry}

\usepackage{mathtools}

\usepackage{mleftright}
\mleftright

\usepackage{nccmath}
\usepackage{cases}

\usepackage{calc}
\usepackage{graphicx}

\usepackage{hyperref}

\DeclareRobustCommand{\SkipTocEntry}[5]{}

\let\stdthebibliography\thebibliography
\let\stdendthebibliography\endthebibliography
\renewenvironment*{thebibliography}[1]{%
  \stdthebibliography{OLBC}}
  {\stdendthebibliography}

\def\Z{\mathbb{Z}}

\def\R{\mathbb{R}}
\def\H{\mathbb{H}}

\def\C{\mathbb{C}}

\DeclareMathOperator{\sh}{sh}
\DeclareMathOperator{\ch}{ch}
\let\th\undefined
\DeclareMathOperator{\th}{th}
\DeclareMathOperator{\im}{Im}
\DeclareMathOperator{\re}{Re}

\DeclareMathOperator{\Ai}{Ai}
\DeclareMathOperator{\arccosh}{arccosh}
\def\SL{{\rm SL}}
\def\GL{{\rm GL}}

\newcommand{\pfrac}[2]{\left(\frac{#1}{#2}\right)}
\newcommand{\pmfrac}[2]{\left(\mfrac{#1}{#2}\right)}
\newcommand{\ptfrac}[2]{\left(\tfrac{#1}{#2}\right)}
\newcommand{\pMatrix}[4]{\left(\begin{matrix}#1 & #2 \\ #3 & #4\end{matrix}\right)}
\newcommand{\ppMatrix}[4]{\left(\!\pMatrix{#1}{#2}{#3}{#4}\!\right)}
\renewcommand{\pmatrix}[4]{\left(\begin{smallmatrix}#1 & #2 \\ #3 & #4\end{smallmatrix}\right)}
\renewcommand{\bar}[1]{\overline{#1}}

\newcommand{\leg}[2] {\left(\frac{#1}{#2}\right)}
\newcommand{\tleg}[2] {\left(\tfrac{#1}{#2}\right)}
\renewcommand{\(}{\left(}
\renewcommand{\)}{\right)}

\renewcommand{\a}{\mathfrak{a}}

\renewcommand{\hat}{\widehat}
\renewcommand{\tilde}{\widetilde}

\renewcommand{\sl}{\big|}

\DeclareMathOperator{\sgn}{sgn}
\DeclareMathOperator{\var}{var}

\def\ep{\epsilon}

\newtheorem{theorem}{Theorem}[section]
\newtheorem{lemma}[theorem]{Lemma}
\newtheorem{corollary}[theorem]{Corollary}
\newtheorem{proposition}[theorem]{Proposition}

\theoremstyle{remark}
\newtheorem*{remark}{Remark}

\newtheorem*{notation}{Notation}

\numberwithin{equation}{section}

\mathtoolsset{showonlyrefs}

\title[Kloosterman sums and Maass cusp forms]{Kloosterman sums and Maass cusp forms of half integral weight for the modular group}

\date{\today}

\author{Scott Ahlgren}
\address{Department of Mathematics\\
University of Illinois\\
Urbana, IL 61801} 
\email{sahlgren@illinois.edu} 

\author{Nickolas Andersen}
\address{Department of Mathematics\\
University of Illinois\\
Urbana, IL 61801} 
\email{nandrsn4@illinois.edu}
 
\thanks{The first author was supported by a grant from the Simons Foundation (\#208525 to Scott Ahlgren).}

\begin{document}

\begin{abstract}
We estimate the sums
\[
	\sum_{c\leq x} \frac{S(m,n,c,\chi)}{c},
\] 
where the $S(m,n,c,\chi)$ are Kloosterman sums of half-integral weight on the modular group.
Our estimates are uniform in $m,n,$ and $x$ in analogy with Sarnak and Tsimerman's improvement of Kuznetsov's bound for the ordinary Kloosterman sums.  
Among other things this requires us to develop
mean value estimates for coefficients of Maass cusp forms of weight $1/2$ and  uniform estimates
for $K$-Bessel integral transforms.
As an application, we obtain an improved estimate for the classical problem of estimating the size of the error term in 
 Rademacher's formula for the partition function $p(n)$.
\end{abstract}

\maketitle

\setcounter{tocdepth}{1}
\tableofcontents


\section{Introduction and statement of results}

The Kloosterman sum
\[
	S(m,n,c) = \sum_{\substack{d\bmod c \\ (d,c)=1}} e\pfrac{m \bar d + n d}{c}, \quad\quad e(x):=e^{2\pi i x}
\]
plays a leading part in  analytic number theory.
Indeed, many problems can be reduced to estimates for sums of Kloosterman sums (see, for example, \cite{heath-brown-survey} or \cite{sarnak-additive}).

In this paper we study sums of generalized Kloosterman sums attached to a multiplier system of weight $1/2$ (the sums $S(m,n,c)$ are attached to the trivial integral weight multiplier system).
Kloosterman sums with general multipliers have been studied by Bruggeman \cite{bruggeman-survey}, Goldfeld-Sarnak \cite{goldfeld-sarnak}, and Pribitkin \cite{pribitkin}, among others.
We focus  on the multiplier system $\chi$ for the Dedekind eta function and the associated Kloosterman sums
\begin{equation} \label{eq:kloo-def-1}
	S(m,n,c,\chi) = \sum_{\substack{0\leq a,d<c \\ \pmatrix abcd\in \SL_2(\Z)}} \bar\chi \ppMatrix abcd \, e\pfrac{\tilde m a + \tilde n d}c, \qquad \tilde m := m-\mfrac{23}{24}.
\end{equation}
These sums are intimately connected to the partition function $p(n)$, and we begin by discussing an application of our main theorem to a classical problem.

In the first of countless important applications of the circle method, Hardy and Ramanujan \cite{hardy-ramanujan} proved the asymptotic formula
\[
	p(n) \sim \mfrac 1{4\sqrt{3} \, n} e^{\pi \sqrt{2n/3}}
\]
(and in fact developed an asymptotic series for $p(n)$).
Perfecting their method, Rademacher \cite{rademacher-partition-function, rademacher-partition-series} proved that
\begin{equation} \label{eq:rademacher}
	p(n) = \frac{2\pi}{(24n-1)^{\frac34}} \sum_{c=1}^\infty \frac{A_c(n)}{c} I_{\frac32}\pfrac{\pi\sqrt{24n-1}}{6c},
\end{equation}
where $I_{\frac32}$ is the $I$-Bessel function,
\begin{equation*}
A_c(n):=\sum_{\substack{d \bmod c\\ (d,c)=1}}e^{\pi i s(d,c)}e\(-\frac{dn}c\),
\end{equation*}
and $s(d,c)$ is a Dedekind sum (see \eqref{eq:ded-sum-def} below).
The Kloosterman sum $A_c(n)$ is a special case of \eqref{eq:kloo-def-1}; in particular we have (see \S \ref{sec:gen-kloo-sums} below)
\begin{equation}\label{eq:a_to_s}
	A_c(n) = \sqrt{-i}S(1,1-n,c,\chi).
\end{equation}
The series \eqref{eq:rademacher} converges very rapidly.  
For example, the first four terms give 
\begin{align*}
	p(100) 	&\approx 190\,568\,944.783 + 348.872 - 2.598 + 0.685 \\
			&= 190 \, 569 \, 291.742, 
\end{align*}
while the actual value is $p(100) = 190 \, 569 \, 292$.

A natural  problem is to estimate the error which results from truncating the series \eqref{eq:rademacher} after the $N$th term, or in other words to estimate the quantity $R(n, N)$ defined by 
\begin{equation} \label{raderrorterm}
	p(n) = \frac{2\pi}{(24n-1)^{\frac34}} \sum_{c=1}^N \frac{A_c(n)}{c} I_{\frac32}\pfrac{\pi\sqrt{24n-1}}{6c}+R(n, N).
\end{equation}
Since $I_\nu(x)$ grows exponentially as $x\to\infty$, one must assume that $N\gg\sqrt{n}$
in order to obtain reasonable estimates.
For $\alpha>0$,
Rademacher \cite[(8.1)]{rademacher-partition-function} showed that
\[R(n, \alpha\sqrt n)\ll_\alpha n^{-\frac14}.\]
Lehmer \cite[Theorem 8]{lehmer-series} proved the sharp Weil-type bound 
\begin{equation} \label{eq:super_weil} 
	|A_c(n)|<2^{\omega_o(c)} \sqrt c,
\end{equation}
where $\omega_o(c)$ is the number of distinct odd primes dividing $c$.
Shortly thereafter \cite{lehmer-remainders} he used this bound to prove that 
\begin{equation}\label{eq:lehmer}
	R(n, \alpha\sqrt n)\ll_\alpha n^{-\frac12}\log n.
\end{equation}
In 1938, values of $p(n)$ had been tabulated only for $n\leq 600$.
Using \eqref{eq:super_weil}, Lehmer showed \cite[Theorem 13]{lehmer-series} that 
$p(n)$ is the nearest integer to the series \eqref{eq:rademacher} truncated at $\frac 23\sqrt{n}$ for all $n>600$.

Building on work of Selberg and Whiteman \cite{whiteman}, Rademacher \cite{rademacher-indian}  later simplified Lehmer's treatment of the sums $A_c(n)$.
Rademacher's book \cite[Chapter IV]{rademacher-book} gives a relatively simple derivation of the error bound $n^{-\frac38}$ using these ideas.

Using  equidistribution results for  Heegner points on the modular curve $X_0(6)$, Folsom and Masri \cite{folsom-masri}  improved Lehmer's estimate. 
They proved that if $24n-1$ is squarefree, then
\begin{equation} \label{eq:folsom-masri}
	R\left(n, \sqrt{\mfrac n6} \right)\ll n^{-\frac12-\delta}\ \ \ \text{for some $\delta>0$}.
\end{equation}

Using the estimates for sums of Kloosterman sums below, we obtain an improvement 
in the exponent, and  we (basically) remove the assumption 
that $24n-1$ is squarefree.
For simplicity in stating the results, we will assume that 
\begin{equation} \label{eq:n-5-7-11-intro}
	24n-1 \text{ is not divisible by $5^4$ or $7^4$.}
\end{equation}
As the proof will show (c.f. Section~\ref{sec:avgduke})
 the exponent $4$ in \eqref{eq:n-5-7-11-intro} can be replaced by any positive integer $m$;
such a change would have the effect of changing the implied constants in \eqref{eq:raderror} and \eqref{eq:raderror1}.

\begin{theorem}\label{raderror}
Suppose that  $\alpha>0$. For $n\geq1$ satisfying \eqref{eq:n-5-7-11-intro} we have
\begin{equation}\label{eq:raderror}
	R\big(n, \alpha \, n^{\frac 12}\big) \ll_{\alpha,\epsilon} n^{-\frac12-\frac1{168}+\epsilon}.
\end{equation}
\end{theorem}

In fact, our method  optimizes when $N$ is slightly larger with respect to $n$.
\begin{theorem}\label{raderror1}
Suppose that  $\alpha>0$. For $n\geq1$ satisfying \eqref{eq:n-5-7-11-intro} we have
\begin{equation}\label{eq:raderror1}
	R\big(n, \alpha \, n^{\frac 12+\frac5{252}}\big) \ll_{\alpha,\epsilon} n^{-\frac12-\frac1{28}+\epsilon}.
\end{equation}
\end{theorem}

\begin{remark}  Suppose, for example, that 
$24n-1$ is divisible by $5^4$.
Then $n$ has the form $n = 625 m + 599$,
so that $p(n)\equiv 0\pmod{625}$ by work of Ramanujan  \cite[Section 22]{berndt-ono}.
A sharp estimate for $R(n,N)$ for such $n$ is less important   since 
$p(n)$ can be determined by 
 showing that
$|R(n,N)| < 312.5$.
The situation is  similar for  any power of  $5$ or $7$.
\end{remark}

\begin{remark}
The estimates in Theorems \ref{raderror} and \ref{raderror1} depend on progress toward the Ramanujan--Lindel\"of conjecture for coefficients of certain Maass cusp forms of weight 1/2.
Assuming the conjecture, we can use the present methods to prove
\[
	R(n,\alpha n^{\frac 12}) \ll_{\alpha,\epsilon} n^{-\frac 12-\frac 1{16}+\epsilon}.
\]
The analogue of Linnik and Selberg's conjecture (\eqref{eq:linnik-selberg} below) would give 
\[
	R(n,\alpha n^\frac12) \ll_{\alpha,\epsilon} n^{-\frac 34+\epsilon}.
\]
Computations suggest that this bound would be optimal.
\end{remark}

Theorems \ref{raderror} and \ref{raderror1} follow from estimates for weighted sums of the Kloosterman sums \eqref{eq:kloo-def-1}.
For the ordinary Kloosterman sum $S(m,n,c)$, Linnik \cite{linnik} and Selberg \cite{selberg-estimation} conjectured that there should be considerable cancellation in the sums
\begin{equation} \label{eq:sum-weighted-kloo}
	\sum_{c\leq x} \frac{S(m,n,c)}c.
\end{equation}
Sarnak and Tsimerman \cite{sarnak-tsimerman} studied these weighted sums for varying $m,n$ and put forth the following modification of Linnik and Selberg's conjecture with an ``$\epsilon$-safety valve'' in $m,n$:
\begin{equation} \label{eq:linnik-selberg}
	\sum_{c\leq x} \frac{S(m,n,c)}c \ll_{\epsilon} (|mn|x)^{\epsilon}.
\end{equation}
Using the Weil bound \cite{weil}
\begin{equation}\label{eq:weil-ordinary}
	|S(m,n,c)| \leq \tau(c) (m,n,c)^{\frac 12} \sqrt c,
\end{equation}
where $\tau(c)$ is the number of divisors of $c$, one obtains
\[
	\sum_{c\leq x} \frac{S(m,n,c)}c \ll  \tau((m,n))x^\frac12 \log x.
\]
Thus the conjecture \eqref{eq:linnik-selberg} represents full square-root cancellation.

The best current bound in the $x$-aspect for the sums \eqref{eq:sum-weighted-kloo} was obtained by 
Kuznetsov \cite{kuznetsov}, who proved for $m,n>0$ that
\[
	\sum_{c\leq x}\frac{S(m,n,c)}c\ll_{m,n} x^\frac16(\log x)^\frac13.
\]
Recently, Sarnak and Tsimerman \cite{sarnak-tsimerman} refined Kuznetsov's method, making the dependence on $m$ and $n$ explicit. They proved that for 
$m$, $n>0$ we have 
\[
	\sum_{c\leq x}\frac{S(m,n,c)}c\ll_{\ep} \(x^\frac16+(mn)^\frac16+(m+n)^\frac18(mn)^{\frac \theta2}\)(mnx)^\ep,
\]
where $\theta$ is an admissible exponent toward the Ramanujan--Petersson conjecture for coefficients of weight 0 Maass cusp forms. 
By work of Kim and Sarnak \cite[Appendix 2]{kim-sarnak}, the exponent $\theta=7/64$ is available.
A generalization of their results to sums taken over $c$ which are divisible by a fixed integer $q$  is given by Ganguly and Sengupta   \cite{ganguly-sengupta}.

We study sums of the Kloosterman sums $S(m,n,c,\chi)$ defined in \eqref{eq:kloo-def-1}.
In view of \eqref{eq:a_to_s} we consider the case when $m$ and $n$ have mixed sign
(as will be seen in the proof, this brings up a number of difficulties which are not present in the case when $m$, $n>0$).
\begin{theorem} \label{thm:x-1/6}
For $m>0$, $n<0$ we have
\begin{equation*} 
	\sum_{c\leq x} \frac{S(m,n,c, \chi)}{c} \ll_\ep \left(x^{\frac 16} + |mn|^{\frac 14}\right) |mn|^{\epsilon} \log x.
\end{equation*}
\end{theorem}

While Theorem \ref{thm:x-1/6} matches \cite{sarnak-tsimerman} in the $x$ aspect, it falls short in the $mn$ aspect.
As we discuss in more detail below, this is due to the unsatisfactory Hecke theory in half-integral weight.
This bound is also insufficient to improve the  estimate \eqref{eq:lehmer} for $R(n,N)$; for this we sacrifice the bound in the $x$-aspect for an improvement in the $n$-aspect.  The resulting theorem, which for convenience we state in terms of $A_c(n)$,  leads to  the estimates of Theorems~\ref{raderror} and \ref{raderror1}.
\begin{theorem}\label{sumofks-intro}
Fix $0<\delta<1/2$.  Suppose that $n\geq 1$ satisfies \eqref{eq:n-5-7-11-intro}.
Then
\begin{equation*}
\sum_{c\leq x}\frac{A_c(n)}{c}\ll_{\delta,\ep} n^{\frac 14-\frac1{56}+\epsilon}x^{\frac 34\delta} + \left( n^{\frac 14-\frac1{168}+\epsilon} + x^{\frac 12-\delta} \right) \log x.
\end{equation*}
\end{theorem}

As in \cite{kuznetsov} and \cite{sarnak-tsimerman}, 
the basic tool is a version of Kuznetsov's trace formula.
This relates sums of the Kloosterman sums $S(m,n,c,\chi)$ weighted by a suitable test function to sums involving Fourier coefficients of Maass cusp forms of weight $1/2$ and multiplier $\chi$.
Proskurin \cite{proskurin-new} proved such a formula for general weight and multiplier when $mn>0$; in Section~\ref{sec:KTF}
 we give a proof in the case $mn<0$  (Blomer \cite{blomer} has  recorded this formula for twists of the  theta-multiplier by a Dirichlet character).

The Maass cusp forms of interest   are functions which transform like $\im(\tau)^{\frac 14}\eta(\tau)$, where
\begin{equation} \label{eq:eta-def}
	\eta(\tau) = e\Big(\mfrac{\tau}{24}\Big)\prod_{n=1}^\infty \big(1-e(n\tau)\big), \qquad \im(\tau)>0,
\end{equation}
is the Dedekind eta function, and which are 
eigenfunctions of the weight $1/2$ Laplacian $\Delta_{1/2}$ (see Section \ref{sec:background} for details).
To each Maass cusp form $F$ we attach an
eigenvalue $\lambda$ and a spectral parameter $r$
which are defined via
\[\Delta_\frac12 F+\lambda F=0\]
and 
\[\lambda=\mfrac14+r^2.\]

We denote the space spanned by these Maass cusp forms  by $\mathcal S_\frac12(1,\chi)$.
Denote by $\{u_j(\tau)\}$ an orthonormal  basis of Maass cusp forms for this space, and let 
 $r_j$ denote the spectral parameter attached to each $u_j$.
Then (recalling the notation $\tilde n=n-\frac{23}{24}$), $u_j$ has a Fourier expansion of the form
\begin{equation*}
	u_j(\tau) = \sum_{n\neq 0} \rho_j(n) W_{\frac{\sgn(n)}4,ir_j}(4\pi |\tilde n| y) e(\tilde n x),
\end{equation*}
where $W_{\kappa,\mu}(y)$ is the $W$-Whittaker function (see \S \ref{sec:maass-forms}).
In  Section~\ref{sec:KTF} we will prove that
\begin{equation} \label{eq:ktf-intro}
	\sum_{c>0} \frac{S(m,n,c,\chi)}{c} \phi\pfrac{4\pi\sqrt{\tilde m|\tilde n|}}{c} = 8\sqrt i \sqrt{\tilde m|\tilde n|} \, \sum_{r_j} \frac{\bar{\rho_j(m)}\rho_j(n)}{\ch \pi r_j} \check\phi(r_j),
\end{equation}
where $\phi$ is a suitable test function and $\check\phi$ is the $K$-Bessel transform
\begin{equation} \label{eq:phi-check-intro}
	\check\phi(r):=\ch\pi r\int_0^\infty K_{2ir}(u)\phi(u)\frac{du}u.
\end{equation}

In Section~\ref{sect:shimura} 
we introduce theta lifts (as in Niwa \cite{niwa}, Shintani \cite{shintani}) which give a Shimura-type correspondence 
\begin{equation} \label{eq:shimura-intro}
	\mathcal S_\frac12(N,\psi\chi,r) \to \mathcal S_0(6N,\psi^2,2r).
\end{equation}
Here $\mathcal S_k(N,\nu,r)$ denotes the space of Maass cusp forms of weight $k$ on $\Gamma_0(N)$ with multiplier $\nu$ and spectral parameter $r$.
These lifts commute with the action of the Hecke operators 
on the respective spaces, and in particular   allow us to control the size of the weight $1/2$ Hecke eigenvalues,
which becomes important in Section~\ref{sec:avgduke}.
The existence of these lifts shows that there are no small eigenvalues; this fact is used in Sections~\ref{sec:kbessel}, \ref{sec:easyproof}, and \ref{sec:hardproof}.
In particular, from  \eqref{eq:shimura-intro} it will follow that if $\mathcal S_{1/2}(1,\chi,r)\neq 0$ then either $r=i/4$ or $r>1.9$.

To  estimate the right-hand side of \eqref{eq:ktf-intro}, we first
obtain a mean value estimate for the sums
\[
	|\tilde n| \sum_{0<r_j\leq x} \frac{|\rho_j(n)|^2}{\ch \pi r_j}.
\]
In analogy with Kuznetsov's result \cite[Theorem 6]{kuznetsov} in weight $0$,  we prove  in Section~\ref{sec:MVE}
a general mean value result for weight $\pm 1/2$ Maass cusp forms which has the following as a corollary.
\begin{theorem} \label{thm:mve}
With notation as above, we have
\begin{equation*} 
|\tilde n| \sum_{0< r_j\leq x} \frac{|\rho_j(n)|^2}{\ch \pi r_j} =
	\begin{dcases}
		\frac{x^{\frac 52}}{5\pi^2}  + O_\epsilon\(x^{\frac 32}\log x + |n|^{\frac 12+\epsilon}x^{\frac 12}\) & \text{ if }n<0, \\
		\frac{x^{\frac 32}}{3\pi^2} + O_\epsilon\(x^{\frac 12}\log x + n^{\frac 12+\epsilon}\) & \text{ if }n>0.
	\end{dcases}
\end{equation*}
\end{theorem}

We then require uniform estimates for the Bessel transform \eqref{eq:phi-check-intro},
which are made subtle by the oscillatory nature of $K_{ir}(x)$ for small $x$ and by the
transitional range of the $K$-Bessel function.  
In Section \ref{sec:kbessel} we obtain estimates for $\check\phi(r)$ which,
together with Theorem~\ref{thm:mve}, suffice to prove Theorem~\ref{thm:x-1/6}.

To prove Theorem~\ref{sumofks-intro} we require a second estimate for the Fourier coefficients $\rho_j(n)$. 
In \cite{sarnak-tsimerman}
such an estimate is obtained via the simple relationship
\[
	a(n) = \lambda(n) a(1)
\]
satisfied by the coefficients $a(n)$ of a Hecke eigenform with eigenvalue $\lambda(n)$.
This relationship is not available in half-integral weight (the best which one can do is to relate 
coefficients of index $m^2n$ to those of index $n$). 
As a substitute, 
we employ an average version of a theorem of Duke  \cite{duke-half-integral}.

Duke proved that if $a(n)$ is the $n$-th coefficient of a normalized Maass cusp form in $\mathcal S_{1/2}(N,\pfrac{D}\bullet \nu_\theta,r)$
($\nu_\theta$ is the theta-multiplier defined in \S \ref{sec:mult-sys})
then for squarefree  $n$ we have
\[
	|a(n)|\ll_\ep |n|^{-\frac27+\ep}|r|^{\frac52-\frac{\sgn(n)}4}\ch\ptfrac{\pi r}{2}.
\]
This  is not strong enough in the $r$-aspect for our purposes (Baruch and Mao  \cite{baruch-mao} have obtained a bound which is stronger in the $n$-aspect, but even weaker in the $r$-aspect).

Here we modify Duke's argument to obtain an  average version of his result.
We prove in Theorem~\ref{thm:avgduke} that if the $u_j$ as above are
eigenforms of the Hecke operators $T(p^2)$, $p\nmid 6$, then we have
\begin{equation}\label{eq:avgduke}
 	\sum_{0< r_j\leq x} \frac{|\rho_j(n)|^2}{\ch \pi r_j} \ll_\ep |n|^{-\frac47+\ep}x^{5-\frac{\sgn n}2},
\end{equation}
for any collection of  values of $n$ such that $24n-23$ is not divisible by arbitrarily large powers of $5$ or $7$.
In Section \ref{sec:hardproof} we use Theorem \ref{thm:avgduke} and a modification of the argument of \cite{sarnak-tsimerman} to prove Theorem~\ref{sumofks-intro}.
Section \ref{sec:partitionproof} contains the 
 proofs of Theorems \ref{raderror} and \ref{raderror1}.

\begin{notation}
Throughout, $\epsilon$ denotes an arbitrarily small positive number whose value is allowed to change with each occurence.
Implied constants in any equation which contains $\epsilon$  are allowed to depend on $\epsilon$.
For all other parameters, we will use a subscript (e.g. $\ll_{a,b,c}$) to signify  dependencies in the implied constant.
\end{notation}

\addtocontents{toc}{\SkipTocEntry}
\section*{Acknowledgments}
We thank Andrew Booker for providing the computations which appear in Section \ref{sect:shimura}.


\section{Background} \label{sec:background}
We give some background on Maass forms and Kloosterman sums of general weight $k$ and multiplier $\nu$
(references for this material are \cite{duke-friedlander-iwaniec,sarnak-additive,proskurin-new} along with the original sources \cite{maass1,maass2,roelcke2,selberg-harmonic,selberg-estimation}).
 In the body of the paper we will work primarily  in integral or half-integral weight
and with the multiplier attached to a Dirichlet character or to  the Dedekind eta-function,
but for much of the background it is no more complicated to describe the general case.
For the benefit of the reader we have attempted to provide a relatively self-contained account by providing details at various points.

\subsection{Eigenfunctions of the Laplacian}
Let $k$ be a real number and let $\H$ denote the upper half-plane.
The group $\SL_2(\R)$ acts on $\tau=x+iy\in\H$ via linear fractional transformations
\[
	\gamma \tau := \frac{a\tau+b}{c\tau+d}, \qquad \gamma=\pMatrix abcd.
\]
For $\gamma\in \SL_2(\R)$ we define the weight $k$ slash operator   by
\[
	f\sl_k \gamma := j(\gamma,\tau)^{-k} f(\gamma \tau), \qquad j(\gamma,\tau) := \frac{c\tau+d}{|c\tau+d|} = e^{i\arg(c\tau+d)},
\]
where we always choose the argument in $(-\pi,\pi]$.

The weight $k$ Laplacian
\begin{equation*}
	\Delta_k := y^2 \bigg( \frac{\partial^2}{\partial x^2} + \frac{\partial^2}{\partial y^2} \bigg) - iky \frac{\partial}{\partial x}
\end{equation*}
 commutes with the weight $k$ slash operator for every $\gamma\in \SL_2(\R)$.
The operator $\Delta_k$ can be written as
\begin{align}
	\Delta_k &= -R_{k-2}L_k - \mfrac k2\Big(\!1-\mfrac k2\Big), \label{eq:delta-rk-lk} \\
	\Delta_k &= -L_{k+2}R_k + \mfrac k2\Big(\!1+\mfrac k2\Big), \label{eq:delta-lk-rk}
\end{align}
where $R_k$ is the Maass raising operator
\begin{equation*}
	R_k := \frac k2 + 2iy \frac{\partial}{\partial \tau} = \frac k2 + iy \left(\frac{\partial}{\partial x} - i\frac{\partial}{\partial y}\right),
\end{equation*}
and $L_k$ is the Maass lowering operator
\begin{equation*}
	L_k := \frac k2 + 2iy \frac{\partial}{\partial \bar \tau} = \frac k2 + iy \left(\frac{\partial}{\partial x} + i\frac{\partial}{\partial y}\right).
\end{equation*}
These raise and lower the weight by $2$: we have

\[
	R_k \big( f\sl_k \gamma \big) = (R_k f) \sl_{k+2} \gamma 
	\quad \text{and} \quad
	L_k \big( f\sl_k \gamma \big) = (L_k f) \sl_{k-2} \gamma.
\]
From \eqref{eq:delta-rk-lk} and \eqref{eq:delta-lk-rk} we obtain the relations
\begin{equation} \label{eq:rk-lk-delta-commute}
	R_k \Delta_k = \Delta_{k+2}R_k \quad \text{and} \quad L_k \Delta_k = \Delta_{k-2} L_k.
\end{equation}

A real analytic function $f:\H\to\C$ is an eigenfunction of $\Delta_k$ with eigenvalue $\lambda$ if
\begin{equation} \label{eq:delta-k-f-lambda-f}
	\Delta_k f + \lambda f = 0.
\end{equation}
If $f$ satisfies \eqref{eq:delta-k-f-lambda-f} then for 
 notational convenience we write
\[
	\lambda = \mfrac 14 + r^2,
\]
and we refer to $r$ as the spectral parameter of $f$.
From \eqref{eq:rk-lk-delta-commute} it follows that if $f$ is an eigenfunction of $\Delta_k$ with eigenvalue $\lambda$ then $R_k f$ (resp. $L_k f$) is an eigenfunction of $\Delta_{k+2}$ (resp. $\Delta_{k-2}$) with eigenvalue $\lambda$.

\subsection{Multiplier systems}\label{sec:mult-sys}

Let $\Gamma = \Gamma_0(N)$ for some $N\geq 1$.
We say that $\nu:\Gamma\to \C^\times$ is a multiplier system of weight $k$ if
\begin{enumerate}[\hspace{1.5em}(i)]\setlength\itemsep{.4em} 
	\item $|\nu|=1$,
	\item $\nu(-I)=e^{-\pi i k}$, and
	\item $\nu(\gamma_1 \gamma_2) \, j(\gamma_1\gamma_2,\tau)^k=\nu(\gamma_1)\nu(\gamma_2) \, j(\gamma_2,\tau)^k j(\gamma_1,\gamma_2\tau)^k$ for all $\gamma_1,\gamma_2\in \Gamma$.
\end{enumerate}
If $\nu$ is a multiplier system of weight $k$, then it is also a multiplier system of weight $k'$ for any $k'\equiv k\pmod 2$, and the conjugate $\bar\nu$ is a multiplier system of weight $-k$.
If $\nu_1$ and $\nu_2$ are multiplier systems of weights $k_1$ and $k_2$ for the same group $\Gamma$, then their product $\nu_1\nu_2$ is a multiplier system of weight $k_1+k_2$ for $\Gamma$.

When $k$ is an integer, a multiplier system of weight $k$ for $\Gamma$ is simply a character of $\Gamma$ which satisfies $\nu(-I)=e^{-\pi i k}$.
If $\psi$ is an even (resp. odd) Dirichlet character modulo $N$, we can extend $\psi$ to a multiplier system of even (resp. odd) integral weight for $\Gamma$ by setting $\psi(\gamma):=\psi(d)$ for $\gamma=\pmatrix abcd$.

Given a cusp $\a$ let $\Gamma_\a:=\{\gamma\in\Gamma:\gamma\a=\a\}$ denote its stabilizer in $\Gamma$, and let $\sigma_\a$ denote the unique
(up to translation by $\pm\pmatrix 1*01$ on the right) matrix in $\SL_2(\R)$ satisfying $\sigma_\a \infty = \a$ and $\sigma_\a^{-1}\Gamma_\a\sigma_\a=\Gamma_\infty$.
Define $\alpha_{\nu,\a} \in [0,1)$ by the condition
\[
	\nu\left(\!\sigma_\a\pMatrix 1101 \sigma_\a^{-1}\!\right) = e(-\alpha_{\nu,\a}).
\]
We say that $\a$ is singular with respect to $\nu$ if $\nu$ is trivial on $\Gamma_\a$, that is, if $\alpha_{\nu,\a}=0$.
For convenience we write $\alpha_{\nu}$ for $\alpha_{\nu,\a}$ when $\a=\infty$.
Note that if $\alpha_\nu>0$ then
\begin{equation} \label{eq:alpha-nu-bar}
	\alpha_{\bar\nu} = 1-\alpha_\nu.
\end{equation}

We are primarily interested in the multiplier system $\chi$ of weight $1/2$ on $\SL_2(\Z)$ (and its conjugate $\bar\chi$ of weight $-1/2$) where
\begin{equation} \label{eq:eta-mult-def}
	\eta(\gamma \tau) = \chi(\gamma)\sqrt{c\tau+d}\,\eta(\tau), \qquad \gamma=\pMatrix abcd \in \SL_2(\Z),
\end{equation}
and $\eta$ is the Dedekind eta function defined in \eqref{eq:eta-def}.
From condition (ii) above we have $\chi(-I)=-i$.
From  the definition of $\eta(\tau)$ we have $\chi(\pmatrix 1101)=e(1/24)$;
it follows that
\[
	\alpha_{\chi} = \mfrac{23}{24} \quad \text{and} \quad \alpha_{\bar\chi} = \mfrac{1}{24},
\]
so  the unique cusp $\infty$ of $\SL_2(\Z)$ is singular neither  with respect to $\chi$ nor with respect to  $\bar\chi$.

For $\gamma=\pmatrix abcd$ with $c>0$ there are two useful formulas for $\chi$.
Rademacher (see, for instance, (74.11), (74.12), and (71.21) of \cite{rademacher-book}) showed that 
\begin{equation} \label{eq:chi-dedekind-sum}
	\chi(\gamma) = \sqrt{-i} \, e^{-\pi i s(d,c)} \, e\pfrac{a+d}{24c},
\end{equation}
where $s(d,c)$ is the Dedekind sum
\begin{equation}  \label{eq:ded-sum-def}
	s(d,c) = \sum_{r=1}^{c-1} \mfrac rc \, \left(\mfrac{dr}{c} - \left\lfloor\! \mfrac{dr}{c}\!\right\rfloor - \mfrac 12\right).	
\end{equation}
On the other hand, Petersson (see, for instance, \S 4.1 of \cite{knopp}) showed that 
\begin{equation} \label{eq:chi-kronecker-symbol}
	\chi(\gamma) = 
	\begin{dcases}
		\Big( \mfrac dc \Big) \, e\Big(\mfrac 1{24} \left[(a+d)c-bd(c^2-1)-3c\right]\Big) & \text{ if $c$ is odd}, \\
		\Big(\mfrac cd \Big) \, e\Big(\mfrac 1{24} \left[(a+d)c-bd(c^2-1)+3d-3-3cd\right]\Big) & \text{ if $c$ is even.}
	\end{dcases}
\end{equation}

We will also use the multiplier system $\nu_\theta$ of weight $1/2$ on $\Gamma_0(4N)$ defined by
\[
	\theta(\gamma\tau) = \nu_\theta(\gamma) \sqrt{c\tau+d} \, \theta(\tau), \qquad \gamma\in \Gamma_0(4N),
\]
where
\[
	\theta(\tau) := \sum_{n\in\Z} e(n^2 \tau).
\]
Explicitly, we have
\begin{equation} \label{eq:def-theta-mult}
	\nu_\theta \ppMatrix abcd = \pfrac cd \varepsilon_d^{-1},
\end{equation}
where $\ptfrac\cdot\cdot$ is the extension of the Kronecker symbol given e.g. in \cite{shimura} and
\[
	\varepsilon_d = \pmfrac{-1}d^{\frac 12} =
	\begin{cases}
		1 & \text{ if }d\equiv 1\pmod 4, \\
		i & \text{ if }d\equiv 3\pmod 4.
	\end{cases}
\]

Let $\Gamma_0(N,M)$ denote the subgroup of $\Gamma_0(N)$ consisting of matrices whose upper-right entry is divisible by $M$.
Equation \eqref{eq:chi-kronecker-symbol} shows that
\[
	\chi(\gamma) = \pmfrac cd e\pmfrac{d-1}8 \qquad \text{for } \gamma=\pMatrix abcd\in \Gamma_0(24,24).
\]
This implies that
\begin{equation} \label{eq:576-theta}
	\chi \ppMatrix a{24b}{c/24}d = \pmfrac{12}d \pmfrac cd \varepsilon_d^{-1} \qquad \text{for } \pMatrix abcd\in \Gamma_0(576),
\end{equation}
which allows one to relate automorphic functions with multiplier $\chi$ on $\SL_2(\Z)$ to those with multiplier $\pfrac{12}{\bullet}\nu_\theta$ on $\Gamma_0(576)$.

\subsection{Maass forms} \label{sec:maass-forms}

A function $f:\H\to\C$ is automorphic of weight $k$ and multiplier $\nu$ for $\Gamma=\Gamma_0(N)$ if
\begin{equation} \label{eq:transformation-law}
	f \sl_k \gamma = \nu(\gamma) f \qquad \text{ for all }\gamma\in\Gamma.
\end{equation}
Let $\mathcal{A}_k(N,\nu)$ denote the space of all such functions.
A smooth, automorphic function which is also an eigenfunction of $\Delta_k$ and which satisfies the growth condition
\begin{equation} \label{eq:growth-condition}
	f(\tau) \ll y^\sigma + y^{1-\sigma} \qquad \text{for some $\sigma$, for all $\tau\in \H$} 
\end{equation}
is called a Maass form.
We let $\mathcal{A}_k(N,\nu,r)$ denote the vector space of Maass forms with spectral parameter $r$.
From the preceeding discussion, the raising (resp. lowering) operator maps $\mathcal{A}_k(N,\nu,r)$ into $\mathcal{A}_{k+2}(N,\nu,r)$ (resp. $\mathcal{A}_{k-2}(N,\nu,r)$).
Also, complex conjugation $f\to \bar f$ gives a bijection $\mathcal{A}_k(N,\nu,r) \longleftrightarrow \mathcal{A}_{-k}(N,\bar\nu,r)$.

If $f \in \mathcal{A}_k(n,\nu,r)$, then it satisfies $f(\tau+1)=e(-\alpha_\nu)f(\tau)$.
For $n\in\Z$ define
\begin{equation}\label{eq:n_nu_def}
n_\nu := n-\alpha_\nu.
\end{equation}
Then $f$ has a Fourier expansion of the form
\begin{equation} \label{eq:fourier-exp-0}
	f(\tau) = \sum_{n=-\infty}^{\infty} a(n,y) e(n_\nu x).
	\end{equation}
By imposing  condition \eqref{eq:delta-k-f-lambda-f} on the Fourier expansion  we find that for $n_\nu\neq0$, the function $a(n,y)$ satisfies
\[
	\frac{a''(n,y)}{(4\pi |n_\nu|y)^2} + \left( \frac{1/4+r^2}{(4\pi|n_\nu|y)^2} + \frac{k\sgn(n_\nu)}{8\pi|n_\nu|y} - \frac14 \right) a(n,y) = 0.
\]
The Whittaker functions $M_{\kappa,\mu}(y)$ and $W_{\kappa,\mu}(y)$
are the two linearly independent solutions to 
Whittaker's equation \cite[\S 13.14]{nist}
\[
	W'' + \left(\frac{1/4-\mu^2}{y^2} + \frac{\kappa}{y} - \frac 14\right)W = 0.
\]
As $y\to\infty$, the former solution grows exponentially, while the latter decays exponentially.
Since $f$ satisfies the growth condition \eqref{eq:growth-condition}, we must have 
\begin{equation} \label{eq:a-n-y-rho-n}
	a(n,y) = a(n) W_{\frac k2\sgn n_\nu, ir}(4\pi |n_\nu|y)
\end{equation}
for some constant $a(n)$.
For $n_\nu=0$ we have
\[
	a(n,y) = a(0)y^{\frac 12+ir} + a'(0) y^{\frac 12-ir}.
\]
We call the numbers $a(n)$ (and $a'(0)$) the Fourier coefficients of $f$.
For $\re(\mu-\kappa+1/2)>0$ we have the integral representation \cite[(13.14.3), (13.4.4)]{nist}
\begin{equation} \label{eq:W-int-rep}
	W_{\kappa,\mu}(y) = \frac{e^{-y/2} \, y^{\mu+1/2}}{\Gamma\big(\mu-\kappa+\frac12\big)} \int_0^\infty e^{-yt} \, t^{\mu-\kappa-\frac12} (1+t)^{\mu+\kappa-\frac12} \, dt.
\end{equation}
When $\kappa=0$ we have \cite[(13.18.9)]{nist}
\[
	W_{0,\mu}(y) = \frac {\sqrt y}{\sqrt \pi} \, K_{\mu}(y/2),
\]
where $K_{\mu}$ is the $K$-Bessel function. 
For Maass forms of weight $0$, many authors normalize the Fourier coefficients in \eqref{eq:a-n-y-rho-n} so that $a(n)$ is the coefficient of 
$\sqrt y \, K_{ir}(2\pi|n_\nu|y)$, which
has the effect of multiplying $a(n)$ by $2|n_\nu|^{1/2}$.

By \cite[(13.15.26)]{nist}, we have
\[\begin{aligned}
	y \frac{d}{dy} W_{\kappa,\mu}(y) &= \Big(\mfrac y2-\kappa\Big) W_{\kappa,\mu}(y) - W_{\kappa+1,\mu}(y)\\
	&=\(\kappa-\mfrac y2\)W_{\kappa,\mu}(y)+\(\kappa+\mu-\mfrac12\)\(\kappa-\mu-\mfrac12\)W_{\kappa-1,\mu}(y).
\end{aligned}\]
This, together with \cite[(13.15.11)]{nist} shows that
\begin{multline*}
	L_k \Big(W_{\frac k2\sgn(n_\nu),ir}(4\pi |n_\nu|y)e(n_\nu x)\Big) \\
	= W_{\frac {k-2}2\sgn(n_\nu),ir}(4\pi |n_\nu|y)e(n_\nu x) \times
	\begin{cases}
		-\Big(r^2+\mfrac {(k-1)^2}4\Big) & \text{ if }n_\nu>0, \\
		1 & \text{ if }n_\nu<0,
	\end{cases}
\end{multline*}
from which it follows that if $f_L:=L_k f$ has coefficients $a_L(n)$, then
\begin{equation}\label{eq:lower_coeff}
	a_L(n) = 
	\begin{cases}
		-\(r^2+\mfrac{(k-1)^2}{4}\) a(n) & \text{ if }n_\nu>0, \\
		a(n) & \text{ if }n_\nu<0.
	\end{cases}
\end{equation}
Similarly, for the coefficients $a_R(n)$ of $f_R:=R_k f$, 
 we have
\begin{equation}\label{eq:raise-coeff}
	a_R(n) = 
	\begin{cases}
		-a(n) & \text{ if }n_\nu>0, \\
		\Big(r^2+\mfrac {(k+1)^2}4\Big) a(n) & \text{ if }n_\nu<0.
	\end{cases}
\end{equation}

From the integral representation \eqref{eq:W-int-rep} we find that $\bar{W_{\kappa,\mu}(y)}=W_{\kappa,\bar\mu}(y)$ when $y,\kappa\in\R$,
so $W_{\kappa,\mu}(y)$ is real when $\mu \in \R$ or (by \cite[(13.14.31)]{nist}) when $\mu$ is purely imaginary.
Suppose that $f$ has multiplier $\nu$ and that  $\alpha_\nu\neq 0$.
Then by \eqref{eq:alpha-nu-bar} we have $-n_\nu=(1-n)_{\overline\nu}$,
so 
the coefficients $a_c(n)$ of $f_c:= \bar f$ satisfy
\begin{equation} \label{eq:conj-coeff}
	a_c(n) = \bar {a(1-n)}.
\end{equation}

\subsection{The spectrum of $\Delta_k$}

Let $\mathcal{L}_k(N,\nu)$ 
denote the $L^2$-space of automorphic functions with respect to the Petersson inner product
\begin{equation}\label{eq:inner_prod_def}
	\langle f,g \rangle := \int_{\Gamma\backslash\H} f(\tau) \bar{g(\tau)} \, d\mu, \qquad d\mu := \mfrac{dx\,dy}{y^2}.
\end{equation}
Let $\mathcal{B}_k(N,\nu)$ denote the subspace of $\mathcal{L}_k(N,\nu)$ consisting of smooth functions $f$ such that $f$ and $\Delta_k f$ are bounded on $\H$.
By \eqref{eq:delta-rk-lk}, \eqref{eq:delta-lk-rk}, and Green's formula, we have the relations (see also \S 3 of \cite{roelcke1})
\begin{align}
	\langle f, -\Delta_k g \rangle 
	&= \langle R_k f, R_k g \rangle - \mfrac k2\Big(1+\mfrac k2\Big) \langle f,g \rangle \label{eq:f-delta-k-g-R-k} \\
	&= \langle L_k f, L_k g \rangle + \mfrac k2\Big(1-\mfrac k2\Big) \langle f,g \rangle \label{eq:f-delta-k-g-L-k}
\end{align}
for any $f,g\in \mathcal{B}_k(N,\nu)$.
It follows that $\Delta_k$ is symmetric and that 
\begin{equation} \label{eq:delta-k-bounded-below}
	\langle f, -\Delta_k f \rangle \geq \lambda_0(k) \langle f,f \rangle, \qquad \lambda_0(k) := \mfrac {|k|}2 \Big(1-\mfrac{|k|}2\Big).
\end{equation}
By Friedrichs' extension theorem (see e.g. \cite[Appendix A.1]{iwaniec-spectral}) the operator $\Delta_k$ has a unique self-adjoint extension to  $\mathcal{L}_k(N,\nu)$
(denoted also  by by $\Delta_k$).

From  \eqref{eq:delta-k-bounded-below}  we see that the spectrum of $\Delta_k$ is real and contained in $[\lambda_0(k),\infty)$.
The holomorphic forms correspond to
 the bottom of the spectrum.
To be precise, if $f_0\in \mathcal{L}_k(N,\nu)$ has eigenvalue $\lambda_0(k)$ then the equations above show that
\[
	f_0(\tau) =
	\begin{cases}
		y^{\frac k2} \, F(\tau) & \text{ if }k\geq 0, \\
		y^{-\frac k2} \, \bar F(\tau) & \text{ if }k<0,
	\end{cases}
\]
where $F:\H\to\C$ is a   holomorphic cusp form of weight $k$;  i.e. $F$ vanishes at the cusps and  satisfies
\[
	F(\gamma\tau) = \nu(\gamma)(c\tau+d)^k F(\tau) \qquad \text{for all }\gamma=\pmatrix abcd\in \Gamma.
\]
In particular, $\lambda_0(k)$ is not an eigenvalue when the space of cusp forms is trivial.
Note also that if $a_0(n)$ are the coefficients of $f_0$, then
\begin{align}
	a_0(n) = 0 \quad \text{ when } \quad \sgn(n_\nu) = -\sgn(k).
\end{align}

The spectrum of $\Delta_k$ on $\mathcal{L}_k(N,\nu)$ consists of an absolutely continuous spectrum of multiplicity equal to the number of singular cusps, and a discrete spectrum of finite multiplicity.
The Eisenstein series, of which there is one for each singular cusp $\a$, give rise to the continuous spectrum, which  is bounded below by $1/4$.
When $(k,\nu)=(1/2,\chi)$ or $(-1/2,\bar\chi)$ and $\Gamma=\SL_2(\Z)$ there are no Eisenstein series since the only cusp is not  singular.

Let $\mathcal{S}_k(N,\nu)$ denote the orthogonal complement in $\mathcal{L}_k(N,\nu)$ of the space generated by Eisenstein series.
The spectrum of $\Delta_k$ on $\mathcal{S}_k(N,\nu)$ is countable and of finite multiplicity with no limit points except $\infty$.
The exceptional eigenvalues are those which lie in  $(\lambda_0(k),1/4)$.
Let $\lambda_1(N,\nu,k)$ denote the smallest eigenvalue larger than $\lambda_0(k)$ in the spectrum of $\Delta_k$ on $\mathcal{S}_k(N,\nu)$.
Selberg's eigenvalue conjecture states that
\[
	\lambda_1(N,\bm 1,0)\geq \mfrac14,
\]
i.e., there are no exceptional eigenvalues.
Selberg \cite{selberg-estimation} showed that 
\[
	\lambda_1(N,\bm 1,0) \geq \mfrac{3}{16}.
\]
The best progress toward this conjecture was made by Kim and Sarnak \cite[Appendix~2]{kim-sarnak}, who proved that
\[
	\lambda_1(N,\bm 1,0) \geq \mfrac 14 - \pmfrac 7{64}^2 = \mfrac{975}{4096},
\]
as a  consequence of Langlands functoriality for the symmetric fourth power of an automorphic representation on $\GL_2$.

\subsection{Maass cusp forms}
The subspace $\mathcal{S}_k(N,\nu)$ consists of functions $f$ whose zeroth Fourier coefficient at each singular cusp vanishes.
Eigenfunctions of $\Delta_k$ in $\mathcal{S}_k(N,\nu)$ are called Maass cusp forms.
Let $\{f_j\}$ be an orthonormal 
basis of $\mathcal{S}_k(N,\nu)$, and for each $j$ let
\[
	\lambda_j = \mfrac 14+r_j^2
\]
denote the  Laplace eigenvalue and $\{a_j(n)\}$ the Fourier coefficients.
Then we have the Fourier expansion
\begin{equation} \label{eq:u-j-Fourier-exp}
	f_j(\tau) = \sum_{n_\nu\neq 0} a_j(n) W_{\frac k2\sgn n_\nu,ir_j}(4\pi|n_\nu|y) e(n_\nu x).
\end{equation}
Suppose that $f$ has spectral parameter $r$.
Then by \eqref{eq:f-delta-k-g-L-k} we have
\begin{align}
	\lVert L_k f \rVert^2 &= \Big(r^2 + \mfrac{(k-1)^2}{4}\Big) \lVert f \rVert^2. \label{eq:L-k-norm}
\end{align}

Weyl's law describes the distribution of the spectral parameters $r_j$.
Theorem~2.28 of \cite{hejhal-stf2} shows that
\begin{equation}
	\sum_{0\leq r_j\leq T} 1 - \mfrac{1}{4\pi} \int_{-T}^T \mfrac{\varphi'}{\varphi}\Big(\mfrac 12+it\Big) \, dt 
	= \mfrac{\operatorname{vol}(\Gamma\backslash\H)}{4\pi} \, T^2 - \mfrac{K_0}{\pi} \, T\log T + O(T),
\end{equation}
where $\varphi(s)$ and $K_0$ are the determinant (see \cite[p. 298]{hejhal-stf2}) and dimension (see \cite[p. 281]{hejhal-stf2}), respectively, of the scattering matrix $\Phi(s)$ whose entries are given in terms of constant terms of Eisenstein series.
In particular we have
\begin{equation}\label{eq:weyl_law}
	\sum_{0\leq r_j\leq T} 1 = \mfrac{1}{12} \, T^2 + O(T) \qquad \text{for }(k,\nu)=\Big(\mfrac 12, \chi\Big).
\end{equation}

We will be working mostly in the space $\mathcal S_\frac12(1,\chi)$.  
Throughout the paper we will
denote an orthonormal basis for this space by $\{u_j\}$,
and we will write the Fourier expansions as
\begin{gather} \label{eq:uj-fix}
	u_j(\tau) = \sum_{n\neq 0} \rho_j(n) W_{\frac{\sgn n}4,ir_j}(4\pi |\tilde n| y) e(\tilde n x), \qquad \tilde n:=n-\mfrac{23}{24}.
\end{gather}

\subsection{Hecke operators}
We introduce Hecke operators on the spaces $\mathcal S_0(N,\bm 1)$ and $\mathcal S_\frac12(1,\chi)$.
For each prime $p\nmid N$ the Hecke operator $T_p$ acts on a Maass form $f\in \mathcal{A}_0(N,\bm 1)$ as
\[
	T_p f(\tau) = p^{-\frac 12} \Big( \sum_{j\bmod p} f\pmfrac{\tau+j}{p} + f(p\tau) \Big).
\]
The $T_p$ commute with each other and with the Laplacian $\Delta_0$,
so $T_p$ acts on each $\mathcal{S}_0(N,\bm 1, r)$.
The action of $T_p$ on the Fourier expansion of  $f\in\mathcal{S}_0(N,\bm 1, r)$ is given by
\begin{equation} \label{eq:hecke-wt-0}
	T_p f(\tau) = \sum_{n\neq 0} \Big( p^{\frac 12} a(pn) + p^{-\frac 12} a(n/p) \Big) W_{0,ir}(4\pi|n|y) e(nx).
\end{equation}

Hecke operators on half-integral weight spaces have been defined for the theta multiplier (see e.g. \cite{katok-sarnak}).
We define Hecke operators $T_{p^2}$ for $p\geq 5$ on $\mathcal A_{1/2}(1,\chi)$ directly by
\[
	T_{p^2} f = \frac 1p \left[ \sum_{b\bmod p^2} e\pmfrac {-b}{24} f \sl_{\frac 12} \pMatrix{\frac1p}{\frac bp}{0}{p} + e\pmfrac{p-1}8 \sum_{h=1}^{p-1} e\pmfrac{-hp}{24} \pmfrac hp f \sl_{\frac 12} \pMatrix{1}{\frac hp}{0}{1} + f \sl_{\frac 12} \pMatrix{p}{0}{0}{\frac 1p} \right].
\]
One can show that
\begin{equation}
	(T_{p^2}f) \sl_{\frac 12} \pMatrix 1101 = e\pmfrac 1{24} f \qquad \text{and} \qquad (T_{p^2} f) \sl_{\frac 12} \pMatrix 0{-1}10 = \sqrt{-i} \, f
\end{equation}
and that $T_{p^2}$ commutes with $\Delta_{\frac 12}$,
so $T_{p^2}$ is an endomorphism of $\mathcal{S}_{\frac 12}(1,\chi,r)$.
To describe the action of $T_{p^2}$ on Fourier expansions, it is convenient to write $f\in \mathcal{S}_{\frac 12}(1,\chi,r)$ in the form
\[
	f(\tau) = \sum_{n\neq 0} a(n) W_{\frac{\sgn(n)}4,ir}\pmfrac{\pi |n|y}{6} e\pmfrac{nx}{24}.
\]
Then a standard computation gives
\begin{equation}\label{eq:hecke_action}
	T_{p^2} f(\tau) = \sum_{n\neq 0} \left( p \, a(p^2 n) + p^{-\frac 12} \ptfrac{12n}p a(n) + p^{-1} a(n/p^2) \right) W_{\frac{\sgn n }4,ir}\pmfrac{\pi |n|y}{6} e\pmfrac{nx}{24}.
\end{equation}

\subsection{Further operators}

The reflection operator
\[
	f(x+iy) \mapsto f(-x+iy)
\]
defines an involution on $\mathcal S_0(N,\bm 1)$ which commutes with $\Delta_0$ and the Hecke operators.
We say that a Maass cusp form is even if it is fixed by reflection; such a form has Fourier expansion
\[
	f(\tau) =  \sum_{n\neq 0} a(n) W_{0,ir} (4\pi|n|y) \cos(2\pi n x).
\]
We also have the operator
\[
	f(\tau) \mapsto f(d\tau)
\]
which in general raises the level   and   changes the multiplier.
We will only need this operator in the case $d=24$ on $\mathcal S_{1/2}(1,\chi)$, where \eqref{eq:576-theta} shows that
\begin{equation}
	f(\tau) \mapsto f(24\tau) : \mathcal S_{\frac 12}(1,\chi,r) \to \mathcal S_{\frac 12}\Big(576,\pmfrac{12}\bullet \nu_\theta,r\Big).
\end{equation}

\subsection{Generalized Kloosterman sums} \label{sec:gen-kloo-sums}
Our main objects of study are the generalized Kloosterman sums given by 
\begin{equation}\label{eq:kloos_def}
	S(m,n,c,\nu) := \sum_{\substack{0\leq a,d<c \\ \gamma=\pmatrix abcd\in \Gamma}} \bar\nu(\gamma) e\pfrac{m_\nu a+n_\nu d}{c}
	=\sum_{\substack{\gamma\in \Gamma_\infty \backslash \Gamma/ \Gamma_\infty\\ \gamma=\pmatrix abcd }}\bar\nu(\gamma) e\pfrac{m_\nu a+n_\nu d}{c}
\end{equation}
(to see that these are equal, one checks that each summand is invariant under the substitutions $\gamma\to \pmatrix 1101\gamma$
and $\gamma\to \gamma\pmatrix 1101$ using   axiom (iii)  for multiplier systems and \eqref{eq:n_nu_def}).
When $\Gamma=\SL_2(\Z)$ and $\nu=\bm 1$ we recover the ordinary Kloosterman sums
\[
	S(m,n,c) = S(m,n,c,\bm 1) = \sum_{\substack{d\bmod c \\ (d,c)=1}} e\pfrac{m\bar d+nd}{c},
\]
where $d\bar d\equiv 1\pmod c$.
For the eta-multiplier, \eqref{eq:chi-dedekind-sum} gives
\begin{equation*}
	S(m,n,c,\chi) = \sqrt i \, \sum_{\substack{d\bmod c \\ (d,c)=1}} e^{\pi i s(d,c)} e\pfrac{(m-1)\bar d+(n-1)d}{c},
\end{equation*}
so the sums $A_c(n)$ appearing in Rademacher's formula are given by
\begin{equation} \label{eq:A-c-n-S}
	A_c(n) = \sqrt{-i} \, S(1,1-n,c,\chi).
\end{equation}
We recall the bounds \eqref{eq:super_weil} and \eqref{eq:weil-ordinary}  from the introduction.
A general Weil-type bound for $S(m,n,c,\chi)$ which holds for all $m,n$ follows from  work of the second author, and is given in the following proposition.

\begin{proposition}\label{prop:weil-bound}
Let $m,n\in \Z$ and write $24n-23=N$ and $24m-23=v^2 M$, with $M$ squarefree.
For $c\geq 1$ we have
\begin{equation} \label{eq:weil-bound}
	|S(m,n,c,\chi)| \ll \tau\big((v,c)\big)\tau(c) \(MN,c\)^{\frac 12} \sqrt c.
\end{equation}
\end{proposition}

\begin{proof}
We apply Theorem 3 of \cite{andersen-singular}.
In the notation of that paper, we have $S(m,n,c,\chi)=\sqrt{i} \, K(m-1,n-1;c)$.
Estimating the right-hand side of \cite[(1.20)]{andersen-singular} trivially, we find that
\begin{equation} \label{S-leq-sum-R}
	|S(m,n,c,\chi)| \leq \mfrac{4}{\sqrt 3} \, \tau\big((v,c)\big) c^{\frac 12} R\big(MN,24c\big),
\end{equation}
where
\[
	R(y,\ell) = \# \{x\bmod \ell : x^2\equiv y\bmod \ell\}.
\]
The function $R(y,\ell)$ is multiplicative in $\ell$.
If $p$ is prime and  $p\nmid y$, then
\begin{equation} \label{eq:R-y-N-est-rel-prime}
	R(y,p^e) \leq 
	\begin{cases}
		2 & \text{ if $p$ is odd}, \\
		4 & \text{ if $p=2$}.
	\end{cases}
\end{equation}
If $y=p^dy'$ with $p\nmid y'$ then $R(y,p^e)\leq p^{d/2} R(y',p^{e-d})$.
It follows that
\begin{equation} \label{eq:R-y-N-est}
	R(y,\ell) \leq 2\cdot 2^{\omega(\ell)} (y,\ell)^{\frac 12},
\end{equation}
where $\omega(\ell)$ is the number of primes dividing $\ell$.
By \eqref{S-leq-sum-R}, \eqref{eq:R-y-N-est-rel-prime}, \eqref{eq:R-y-N-est}, and the fact that $2^{\omega(\ell)}\leq \tau(\ell)$, we obtain \eqref{eq:weil-bound}.
\end{proof}

One useful consequence of Proposition~\ref{prop:weil-bound}  is the ``trivial'' estimate
\begin{equation}\label{eq:trivial_est}
\sum_{c\leq X}\frac{|S(m,n,c,\chi)|}c\ll X^{\frac12} \log X \, |mn|^\ep,
\end{equation}
which follows from a standard argument involving the mean value bound for $\tau(c)$.


\section{A mean value estimate for coefficients of Maass cusp forms}\label{sec:MVE}

In this section we prove a general version of Theorem \ref{thm:mve} which applies to the Fourier coefficients of weight $\pm 1/2$ Maass cusp forms with multiplier $\nu$ for $\Gamma_0(N)$, where $\nu$ satisfies the following assumptions:
\begin{enumerate}
	\item There exists $\beta=\beta_\nu \in (1/2,1)$ such that
	\begin{align} \label{eq:mve-weil-assumption}
		\sum_{c>0} \frac{|S(n,n,c,\nu)|}{c^{1+\beta}} \ll_{\nu} n^{\epsilon}.
	\end{align}
	\item None of the cusps of $\Gamma_0(N)$ is singular for $\nu$.
\end{enumerate}
By \eqref{eq:trivial_est} the multipliers $\chi$ and $\bar\chi$ on $\SL_2(\Z)$ satisfy these assumptions with $\beta=1/2+\epsilon$.

Fix an orthonormal basis of cusp forms $\{f_j\}$ for $\mathcal{S}_k(\Gamma_0(N),\nu)$.
For each $j$, let $a_j(n)$ denote the $n$-th Fourier coefficient of $f_j$ and let $r_j$ denote the spectral parameter.

\begin{theorem} \label{thm:mve-general}
	Suppose that $k=\pm 1/2$ and that $\nu$ satisfies conditions (1) and (2) above. 
	Then for all $x\geq 2$ and $n\geq 1$ we have
	\begin{equation} \label{eq:mve}
	n_\nu \sum_{0<r_j\leq x} \frac{|a_j(n)|^2}{\ch \pi r_j} =
	\begin{dcases}
		\frac{x^{\frac 32}}{3\pi^2} + O_{\nu}\(x^{\frac 12}\log x + n^{\beta+\epsilon}\) & \text{ if }k=1/2, \\
		\frac{x^{\frac 52}}{5\pi^2} + O_{\nu}\(x^{\frac 32}\log x + n^{\beta+\epsilon}x^{\frac 12}\) & \text{ if }k=-1/2.
	\end{dcases}
	\end{equation}
\end{theorem}

The $n>0$ case of Theorem \ref{thm:mve} follows from a direct application of Theorem \ref{thm:mve-general} when $(k,\nu)=(1/2,\chi)$.
For $n<0$, we apply Theorem \ref{thm:mve-general} to the case $(k,\nu)=(-1/2,\bar\chi)$ and use the relation \eqref{eq:conj-coeff}.

\subsection{An auxiliary version of the Kuznetsov trace formula}
We begin with an auxiliary version of Kuznetsov's formula  (\cite[\S 5]{kuznetsov}) which is Lemma 3 of \cite{proskurin-new} with $m=n$.
By assumption (2) there are no Eisenstein series for the multiplier system $\nu$.
While Proskurin assumes that $k>0$ throughout his paper, this lemma is still valid for $k<0$ by the same proof.
We  include the term $\Gamma(2\sigma-1)$ which is omitted on the right-hand side of \cite[Lemma~3]{proskurin-new}.

\begin{lemma}\label{lem:pre_kusnetsov}
Suppose that $k=\pm 1/2$ and that $\nu$ satisfies conditions (1) and (2) above. 
For $n>0$, $t\in\R$, and $\sigma>1$ we have
\begin{multline} \label{eq:proskurin-lemma-3}
	-\mfrac1{i^{k+1}}\sum_{c>0} \frac{S(n,n,c,\nu)}{c^{2\sigma}} \int_L K_{it} \left(\mfrac{4\pi n_\nu}{c} \, q\right) \left(q+\mfrac 1q\right)^{2\sigma-2} q^{k-1} \, dq \\
	= \frac{\Gamma(2\sigma-1)}{4(2\pi n_\nu)^{2\sigma-1}} - \frac{2^{1-2\sigma}\pi^{2-2\sigma}n_\nu^{2-2\sigma}}{\Gamma(\sigma-\frac k2+\frac {it}2)\Gamma(\sigma-\frac k2-\frac{it}2)} \sum_{r_j} |a_j(n)|^2 \Lambda\(\sigma+\mfrac{it}2,\sigma-\mfrac{it}2,r_j\),
\end{multline}
where $L$ is the contour $|q|=1$, $\re(q)>0$, from $-i$ to $i$ and
\[
	\Lambda(s_1,s_2,r) = \Gamma\(s_1-\mfrac 12-ir\) \Gamma\(s_1-\mfrac 12+ir\) \Gamma\(s_2-\mfrac 12-ir\) \Gamma\(s_2-\mfrac 12+ir\).
\]
\end{lemma}

We justify substituting $\sigma=1$ in this lemma.  
Suppose that $t\in \R$ and write $a = 4\pi n_\nu/c$ and $q=e^{i\theta}$ with $-\pi/2<\theta<\pi/2$.
By \cite[(10.27.4), (10.25.2), and (5.6.7)]{nist} we have, for $c$ sufficiently large,
\begin{equation*}
	K_{it} \big(a e^{i\theta}\big) \ll e^{-\frac{\pi |t|}{2}} \big(e^{\theta t} + e^{-\theta t}\big).
\end{equation*}
So $K_{it}(a q) \ll 1$ on $L$ and
the integral over $L$ is bounded uniformly in $a$.
By this and the assumption \eqref{eq:mve-weil-assumption}, 
the left-hand side of \eqref{eq:proskurin-lemma-3} is absolutely uniformly convergent for $\sigma \in [1,2]$.
In this range we have 
\[
	\left|\Gamma\(\sigma - \mfrac 12 + iy\)\right| = \frac{|\Gamma(\sigma+\tfrac12+iy)|}{|\sigma-\tfrac 12 + iy|} \leq 2 \, \left|\Gamma\(\sigma+\mfrac 12+iy\)\right|.
\]
Using this together with the fact that $\Lambda\(\sigma+\tfrac{it}2,\sigma-\tfrac{it}2,r_j\)$ is positive, we see that
the convergence of the right hand side of  \eqref{eq:proskurin-lemma-3} for $\sigma>1$ implies the   uniform convergence of the right hand side for $\sigma \in [1,2]$. 
We may therefore take $\sigma=1$.

To evaluate the integral on the left-hand side of \eqref{eq:proskurin-lemma-3} we require the following lemma. 

\begin{lemma} \label{lem:int_L}
Suppose that $a>0$. Then
\begin{multline}
	\int_L K_{it}(aq) q^{k-1} \, dq \\= i\sqrt{2} \times
	\begin{dcases}
		\frac 1{\sqrt a} \int_0^a u^{-\frac 12} \int_0^\infty \cos(t\xi)\big(\! \cos(u\ch \xi)-\sin(u \ch \xi)\big) \, d\xi \, du &\text{ if }k=\mfrac12, \\
		\sqrt a \int_a^\infty u^{-\frac 32} \int_0^\infty \cos(t\xi)\big(\! \cos(u\ch \xi)+\sin(u \ch \xi)\big) \, d\xi \, du &\text{ if }k=-\mfrac12.
	\end{dcases}
\end{multline}
\end{lemma}

\begin{proof}
The case $k=1/2$ follows from equation (44) in \cite{proskurin-new} (correcting a typo in the upper limit of integration) and the integral representation \cite[(10.9.8)]{nist}.

For $k=-1/2$, we write
\[
	\int_L K_{it}(aq) q^{k-1} \, dq = - \left( \int_{-i R}^{-i} + \int_i^{iR} + \int_{L_R}\right) K_{it}(aq) q^{k-1} \, dq,
\]
where the first two integrals are along the imaginary axis, and $L_R$ is the semicircle $|q|=R$, $\re(q)>0$, traversed clockwise.
By \cite[(10.40.10)]{nist}, the integral over $L_R$ approaches zero as $R\to\infty$.
For the first and second integrals we change variables to $q=-iu/a$ and $q=iu/a$, respectively, to obtain
\[
	\int_L K_{it}(aq) q^{k-1} \, dq = -\sqrt a \int_a^\infty \left(\sqrt{-i}K_{it}(iu) - \sqrt{i}K_{it}(-iu)\right) \frac{du}{u^{3/2}}.
\]
The lemma follows from writing $K_{it}(\pm iu)$ as a linear combination of $J_{it}(u)$ and $J_{-it}(u)$ using \cite[(10.27.9), (10.27.3)]{nist} and applying the integral representation \cite[(10.9.8)]{nist}.
\end{proof}

Starting with \eqref{eq:proskurin-lemma-3} with $\sigma=1$ and using the fact that
\begin{equation}\label{eq:lambda_def}
	\Lambda\(1+it,1-it,r\) = \frac{\pi^2}{\ch(\pi r+\pi t)\ch(\pi r-\pi t)},
\end{equation}
 we replace $t$ by $2t$, multiply by
\[
	\frac{4 n_\nu}{\pi^2}\,g(t), \quad \text{ where } g(t) := \left|\Gamma\(1-\mfrac k2+it\)\right|^2 \sh \pi t,
\]
and integrate on $t$ from $0$ to $x$.
Applying Lemma \ref{lem:int_L} to the result, we obtain
\begin{equation} \label{eq:mve-1}
	n_\nu \sum_{r_j} \frac{|a_j(n)|^2}{\ch\pi r_j} h_x(r_j) = \frac{1}{2\pi^3} \int_0^x g(t) \, dt
	+ \frac{(4 n_\nu)^{1-k}\sqrt 2}{i^k \pi^{2+k}} \sum_{c>0} \frac{S(n,n,c,\nu)}{c^{2-k}} I\left(x,\mfrac{4\pi n_\nu}{c}\right),
\end{equation}
where
\begin{equation}\label{eq:hxrdef}
	h_x(r) := \int_0^x \frac{2\sh\pi t\ch \pi r}{\ch(\pi r+\pi t)\ch(\pi r-\pi t)} \, dt
\end{equation}
and
\begin{equation}\label{eq:ixadef}
	I(x,a) := 
	\begin{dcases}
		\int_0^x g(t) \int_0^a u^{-\frac12} \int_0^\infty \cos(2t\xi)\big(\!\cos(u\ch \xi) - \sin(u\ch \xi)\big) \, d\xi \, du \, dt &\text{ if }k=\mfrac 12, \\
		\int_0^x g(t) \int_a^\infty u^{-\frac32} \int_0^\infty \cos(2t\xi)\big(\!\cos(u\ch \xi) + \sin(u\ch \xi)\big) \, d\xi \, du \, dt &\text{ if }k=-\mfrac 12.
	\end{dcases}
\end{equation}

\subsection{The function $g(t)$ and the main term}
A computation involving \cite[(5.11.3), (5.11.11)]{nist} shows that for fixed $x$ and $y\geq 1$ we have
\begin{align*}
	\Gamma(x+iy)\Gamma(x-iy) 
	= 2\pi y^{2x-1} e^{-\pi y} \left( 1 + O\pmfrac 1y \right),
\end{align*}
from which it follows that
\begin{equation} \label{eq:g-t-approx}
	g(t) = \pi t^{1-k} + O(t^{-k})\qquad \text{as $t\to\infty$}.
\end{equation}
We have
\begin{equation} \label{eq:g-t-approx0}
g(t)=\pi   \Gamma \(1-\tfrac{k}{2}\)^2t+O(t^2)\qquad \text{as $t\to0$}.
\end{equation}
This gives the main term
\begin{equation} \label{eq:mve-main-term}
	\mfrac{1}{2\pi^3} \int_0^x g(t) \, dt = 
	\begin{dcases}
		\frac {x^{\frac 32}}{3\pi^2}  + O\big(x^{\frac 12}\big) & \text{ if }k=\mfrac 12, \\
		\frac {x^{\frac 52}}{5\pi^2}  + O\big(x^{\frac 32}\big) & \text{ if }k=-\mfrac 12.
	\end{dcases}
\end{equation}

\subsection{Preparing to estimate the error term}
The error term in \eqref{eq:mve} is obtained by uniformly estimating the integral $I(x,a)$ in the ranges $a\leq 1$ and $a\geq 1$.
The analysis is similar in spirit to \cite[Section 5]{kuznetsov}, but the nature of the function $g(t)$ introduces substantial 
difficulties.
In this section we collect a number of facts which we will require for these estimates.

Given an interval $[a,b]$, let $\var f$ denote the total variation of $f$ on $[a,b]$ which, for $f$ differentiable, is given by
\begin{equation} \label{eq:var-diff}
	\var f = \int_a^b |f'(x)| \, dx.
\end{equation}
We begin by collecting some facts about the functions $g'(t)$, 
$(g(t)/t)'$, and $t(g(t)/t)'$.
\begin{lemma}
Let $k=\pm 1/2$.
If $x\geq 1$ then on the interval $[0,x]$ we have
\begin{align}
	\var\, (g(t)/t)' &\ll 1 + x^{-1-k}, \label{eq:G-1-est} \\
	\var\, t(g(t)/t)'&\ll 1 + x^{-k} \label{eq:G-2-est}.
\end{align}
For $t>0$ we have
\begin{equation}\label{eq:gtsign}
g'(t)>0,\qquad \sgn\, (g(t)/t)'=-\sgn k.
\end{equation}
\end{lemma}

\begin{proof}
Let $\sigma=1-k/2$.
Taking the logarithmic derivative of $g(t)/t$, we find that
\begin{equation}\label{eq:gtprime}
	\pfrac{g(t)}t' = g(t) h(t),
\end{equation}
where
\[
	h(t) := \mfrac 1t \left[ i\big(\psi(\sigma+it)-\psi(\sigma-it)\big)-\mfrac 1t+\pi\coth\pi t \right], \qquad \psi := \mfrac{\Gamma'}{\Gamma}.
\]
We have
\[
	g'(t) = g(t)\left[ i\big(\psi(\sigma+it)-\psi(\sigma-it)\big) + \pi\coth\pi t \right],
\]
and
\[
	h'(t) = -\frac{h(t)}{t} - \frac 1t \left[\psi'(\sigma+it)+\psi'(\sigma-it) - \mfrac 1{t^2} + \pi^2 (\operatorname{csch}\pi t)^2 \right].
\]
At $t=0$ we have the Taylor series
\begin{align}
	g(t) &= \pi \Gamma(\sigma)^2 t + \pi\Gamma(\sigma)^2 \Big( \mfrac{\pi^2}{6} - \psi'(\sigma) \Big)t^3 + \ldots, \\
	h(t) &= \mfrac{\pi^2}{3} - 2\psi'(\sigma) + \Big(\mfrac 13 \psi'''(\sigma) - \mfrac{\pi^4}{45}\Big)t^2 + \ldots.
\end{align}
Thus we have the estimates
\begin{equation} \label{eq:g-h-est-0}
	g(t) \ll t, \quad g'(t) \ll 1, \quad h(t) \ll 1, \quad \text{and} \quad h'(t) \ll t \quad \text{ as }t\to 0.
\end{equation}
For large $t$ we use \cite[(5.11.2) and (5.15.8)]{nist} to obtain the asymptotic expansions
\begin{align}
	i\big(\psi(\sigma+it)-\psi(\sigma-it)\big) &= -2\arctan\pfrac t\sigma - \frac{t}{\sigma^2+t^2} + O\Big(\frac{1}{t^2}\Big), \label{i-psi-asymp} \\
	\psi'(\sigma+it)+\psi'(\sigma-it) &= \frac{2\sigma}{\sigma^2+t^2} + \frac{\sigma^2-t^2}{(\sigma^2+t^2)^2} + O\Big(\frac{1}{t^3}\Big). \label{psi-prime-asymp}
\end{align}
By \eqref{eq:g-t-approx}, \eqref{i-psi-asymp}, and \eqref{psi-prime-asymp} we have the estimates
\begin{equation} \label{g-h-infty}
	g(t) \ll t^{1-k}, \quad g'(t) \ll t^{-k}, \quad h(t) \ll \frac{1}{t^2}, \quad \text{and} \quad h'(t) \ll \frac{1}{t^3} \quad \text{as } t\to\infty.
\end{equation}
From \eqref{g-h-infty} and \eqref{eq:var-diff} it follows that
\[
	\var\, (g(t)/t)' = \int_0^x \left|g(t)h'(t)+g'(t)h(t)\right| \, dt \ll 1+x^{-1-k}
\]
and
\[
	\var\, t(g(t)/t)'= \int_0^x \left|t\,g(t)h'(t)+t\,g'(t)h(t)+g(t)h(t)\right| \, dt \ll 1+x^{-k}. \]

Using \cite[(4.36.3), (5.7.6)]{nist} we have
\[h(t)=2\sum_{n=1}^\infty\(\frac1{t^2+n^2}-\frac1{t^2+(n-\frac k2)^2}\).\]
The second claim in  \eqref{eq:gtsign} follows from \eqref{eq:gtprime}.
The first is simpler.
\end{proof}

Suppose that $f\geq 0$ and $\phi$ are of bounded variation on $[a,b]$ with no common points of discontinuity.
By Corollary 2 of \cite{knowles}, there exists $[\alpha,\beta]\subseteq[a,b]$ such that
\begin{equation} \label{eq:knowles}
	\int_a^b f \, d\phi = (\inf f + \var f) \int_\alpha^\beta d\phi,
\end{equation}
where the integrals are of Riemann-Stieltjes type.
We obtain the following  by taking $\phi(x)=\int_0^x G(t)\, dt$.

\begin{lemma} \label{lem:ganelius}
Suppose that $F\geq 0$ and $G$ are of bounded variation on $[a,b]$ and that $G$ is continuous.
Then
\begin{equation} \label{eq:ganelius}
	\left|\int_a^b F(x) G(x) \, dx \right| \leq (\inf F + \var F) \sup_{[\alpha,\beta]\subseteq [a,b]} \left| \int_\alpha^\beta G(x) \, dx \right|.
\end{equation}
\end{lemma}

We will frequently make use of the following well-known estimates for oscillatory integrals (see, for instance, \cite[Chapter IV]{titchmarsh}).

\begin{lemma} [First derivative estimate]
Suppose that $F$ and $G$ are real-valued functions on $[a,b]$ with $F$ differentiable, such that $G(x)/F'(x)$ is monotonic. 
If $|F'(x)/G(x)|\geq m>0$ then
\[
	\int_a^b G(x) e(F(x)) \, dx \ll \frac{1}{m}.
\]
\end{lemma}

\begin{lemma} [Second derivative estimate]
Suppose that $F$ and $G$ are real-valued functions on $[a,b]$ with $F$ twice differentiable, such that $G(x)/F'(x)$ is monotonic. 
If $|G(x)|\leq M$ and $|F''(x)|\geq m>0$ then
\[
	\int_a^b G(x) e(F(x)) \, dx \ll \frac{M}{\sqrt m}.
\]
\end{lemma}

\subsection{The first error estimate}
We estimate the error arising from 
$I(x,a)$ (recall \eqref{eq:ixadef}) when $k=1/2$.
\begin{proposition} \label{prop:mve-int-est-pos}
Suppose that $k=1/2$. 
For $a>0$ we have
\begin{equation}
	I(x,a) \ll
	\begin{cases}
		\sqrt{a} \, \big(\!\log(1/a)+1\big) & \text{ if }a\leq 1, \\
		1 & \text{ if }a\geq 1.
	\end{cases}
\end{equation}
\end{proposition}

\begin{proof}
We write $I(x,a)$ as a difference of integrals, one involving $\cos(u\ch\xi)$, the other involving $\sin(u\ch \xi)$.
We will estimate the first integral, the second being similar.
We have
\[
	\int_0^\infty \cos(2t\xi)\cos(u\ch \xi) \, d\xi = \int_0^T \cos(2t\xi)\cos(u\ch \xi) \, d\xi + R_T(u,t),
\]
where, by the second derivatve estimate,
\[
	R_T(u,t) \ll u^{-\frac 12} e^{-\frac T2}.
\]
For $0<\varepsilon < a$ we have, using \eqref{eq:g-t-approx},
\[
	\int_0^x g(t) \int_{\varepsilon}^a u^{-\frac 12} R_T(u,t) \, du\, dt \ll e^{-\frac T2} x^{\frac 32} \big( |\log a|+|\log \varepsilon|\big) \to 0 \quad \text{ as }T\to\infty.
\]
So our goal is to estimate
\[
	\lim_{\varepsilon\to0^+} \lim_{T\to\infty} I(x,\varepsilon, a, T),
\]
where
\[
	I(x,\varepsilon, a, T) = \int_0^x g(t) \int_\varepsilon^a u^{-\frac 12} \int_0^T \cos(2t\xi) \cos(u\ch \xi) \, d\xi \, du \, dt.
\]
Integrating the innermost integral by parts, we find that
\begin{multline} \label{eq:I-x-ep-a-T-parts}
	I(x,\varepsilon, a, T) = \mfrac 12 \int_0^x \frac{g(t)}{t} \sin(2tT) \, dt \int_\varepsilon^a u^{-\frac 12} \cos(u\ch T) \, du \\ + \mfrac 12 \int_0^T \sh\xi \int_0^x \frac{g(t)}{t} \sin(2t\xi) \, dt \int_\varepsilon^a \sqrt u \, \sin(u\ch \xi) \, du \, d\xi.
\end{multline}
The first derivative estimate (recalling \eqref{eq:g-t-approx0} and \eqref{eq:gtsign}) gives
\begin{equation} \label{eq:I-x-e-a-t-est}
	\int_0^x \frac{g(t)}{t} \sin(2tT) \, dt \ll \frac{1}T \quad \text{ and } \quad 
	\int_\varepsilon^a u^{-\frac 12} \cos(u\ch T) \, du \ll \frac{1}{\sqrt \varepsilon \, \ch T},
\end{equation}
so the first term on the right-hand side of \eqref{eq:I-x-ep-a-T-parts} approaches zero as $T\to\infty$.

Turning to the second term, set
\[
	G_x(\xi) := \int_0^x \frac{g(t)}{t} \sin(2t\xi) \, dt.\]
Integrating by parts gives
\begin{equation} \label{eq:G-x-xi-parts}
	G_x(\xi) = \frac{\pi \Gamma(\tfrac34)^2}{2\xi} - \frac{g(x)}{x}\frac{\cos(2x\xi)}{2\xi} + \frac{1}{2\xi} \int_0^x \pfrac{g(t)}{t}' \cos(2t\xi) \, dt.
\end{equation}
Applying Lemma \ref{lem:ganelius} and \eqref{eq:G-1-est} to the integral in \eqref{eq:G-x-xi-parts} we find that
\[
	\frac{1}{\xi} \int_0^x \pfrac{g(t)}{t}' \cos(2t\xi) \, dt \ll \frac{1}{\xi} \sup_{0\leq\alpha<\beta\leq x} \left| \int_\alpha^\beta \cos(2t\xi) \, dt \right| \ll \frac{1}{\xi^2}.
\]
On the other hand, \eqref{eq:gtsign} gives
\begin{equation} \label{eq:G-x-xi-leq-1}
	\left|\frac{1}{\xi}\int_0^x \pfrac{g(t)}{t}' \cos(2t\xi) \, dt\right| \leq \frac{1}{\xi}\int_0^x -\left(\frac{g(t)}{t}\right)' \, dt \ll \frac{1}{\xi}.
\end{equation}
It follows that
\begin{equation} \label{eq:G-x-xi-O}
	G_x(\xi) = \frac{\pi \Gamma(\tfrac34)^2}{2\xi} - \frac{g(x)}{x}\frac{\cos(2x\xi)}{2\xi} + O\(\min(\xi^{-1},\xi^{-2})\).
\end{equation}

The integral
\[
	\int_\varepsilon^a \sqrt{u} \, \sin(u\ch \xi) \, du
\]
evaluates to $H(a,\xi) - H(\varepsilon, \xi)$, where
\begin{equation} \label{eq:H-a-xi-def}
	H(u,\xi) := - \frac{\sqrt u\, \cos(u\ch \xi)}{\ch \xi} + \frac{\sqrt{\pi/2}}{(\ch \xi)^{\frac32}} C\(\!\sqrt{\frac{2u\ch \xi}{\pi}}\,\),
\end{equation}
and $C(x)$ denotes the Fresnel integral
\[
	C(x) := \int_0^x \cos\(\mfrac \pi2 t^2\) \, dt.
\]
For $u>0$, write
\begin{equation} \label{eq:J-x-u-def}
	J(x,u) := \int_0^\infty \sh \xi \, G_x(\xi) H(u,\xi)\, d\xi.
\end{equation}
To show that $J(x, u)$ converges, suppose that $B\geq A\geq T$. 
By \eqref{eq:G-x-xi-O}, \eqref{eq:H-a-xi-def}, and the bound $C(x)\leq 1$ we have 
\begin{align*}
	&\int_A^B \sh \xi \, G_x(\xi) H(u,\xi)\, d\xi \\
	&\phantom{\int} = -\sqrt{u} \int_A^B \th \xi \cos(u\ch \xi) \left( \frac{\pi\Gamma(\frac34)^2}{2\xi} - \frac{g(x)}{x}\frac{\cos(2x\xi)}{2\xi} + O\Big(\mfrac 1{\xi^2}\Big) \! \right) \, d\xi + O\left(\int_A^B \frac{\sh\xi}{\xi(\ch\xi)^{\frac 32}} \, d\xi \right)\\
	&\phantom{\int} \ll u^{-\frac 12}e^{-A/2} + u^{\frac 12} A^{-1} + e^{-A/2},
\end{align*}
where we have used the first derivative estimate in the last inequality.
We have 
\begin{equation}\label{eq:ixjx}
	\lim_{\varepsilon\to0^+} \lim_{T\to\infty} I(x,\varepsilon, a, T)
	=\mfrac12J(x,a)-\mfrac12\lim_{\varepsilon\to 0^+}J(x,\varepsilon),
\end{equation}
so it remains only to  estimate $J(x,u)$
in the ranges $u\leq 1$ and $u\geq 1$.

First, suppose that $u\leq 1$.
We estimate each of the six terms obtained by multiplying \eqref{eq:G-x-xi-O} and \eqref{eq:H-a-xi-def} in \eqref{eq:J-x-u-def}.
Starting with the terms involving $\cos(u\ch\xi)$,
we have
\begin{align*}
	\sqrt{u} \int_0^\infty \frac{\th \xi}{\xi} \cos(u\ch\xi) \, d\xi
	&= \sqrt{u} \(\int_0^{\log{\frac 1u} } + \int_{\log {\frac 1u}}^\infty\) \frac{\th \xi}{\xi} \cos(u\ch\xi) \, d\xi \\
	&\ll \sqrt{u} \, \big(\log(1/u) + 1\big)
\end{align*}
by applying the second derivative estimate to the second integral.
By the same method,
\[
	\sqrt{u} \, \frac{g(x)}{x} \int_0^\infty \frac{\th \xi}{\xi} \cos(2x\xi) \cos(u\ch\xi) \, d\xi \ll  \sqrt{u} \, \big(\log(1/u) + 1\big),
\]
since $g(x)/x\ll 1$.
For the third term we have
\[
	\sqrt{u} \int_0^\infty \th \xi \cos(u\ch \xi) \min(\xi^{-1},\xi^{-2}) \, d\xi
	\ll \sqrt{u} \int_0^1 \frac{\th \xi}{\xi} \, d\xi + \sqrt u\int_1^\infty \frac{\th \xi}{\xi^2} \, d\xi 
	\ll \sqrt{u}.
\]
For the terms involving the Fresnel integral, we use the trivial estimate $C(x)\leq \min(x,1)$.
For the first of these terms we have
\[
	\int_0^\infty \frac{\sh \xi}{\xi (\ch \xi)^{\frac 32}} C\(\!\sqrt{\mfrac{2u\ch \xi}{\pi}}\,\) \, d\xi \ll \sqrt u \int_0^{\log {\frac 1u}} \frac{\th \xi}{\xi} \, d\xi + \int_{\log \frac 1u}^\infty \frac{d\xi}{\sqrt{\ch \xi}} \ll \sqrt u\,\big(\log(1/u)+1\big).
\]
By the same method, the remaining two terms are $\ll \sqrt{u} \,(\log(1/u)+1)$.
It follows that
\begin{equation} \label{eq:J-x-u-leq-1}
	J(x,u) \ll \sqrt{u} \, \big(\log(1/u)+1\big) \qquad \text{ for }u\leq 1.
\end{equation}

For $u\geq 1$ we show that $J(x, u)\ll 1$.
Write
\[
	J(x,u) = J_1(x,u) + J_2(x,u) := \left(\int_0^{1/\sqrt{u}} + \int_{1/\sqrt{u}}^\infty \right) \sh \xi \, G_x(\xi) H(u,\xi) \, d\xi.
\]
Since $G_x(\xi)\ll \xi^{-1}$ and $H(u,\xi)\ll \sqrt{u}/\ch\xi$, we have
\[
	J_1(x,u) \ll \sqrt{u} \int_0^{1/\sqrt{u}} \frac{\th \xi}{\xi} \, d\xi \ll 1.
\]
We break $J_2(x,u)$ into six terms, using 
\eqref{eq:G-x-xi-parts} in place of \eqref{eq:G-x-xi-O}.
We start with the terms involving $\pi\Gamma(3/4)^2/2\xi$.
By the first derivative estimate we have
\[
	\sqrt{u} \int_{1/\sqrt u}^\infty \frac{\th \xi}{\xi} \cos(u\ch \xi) \, d\xi \ll 1.
\]
Estimating the Fresnel integral trivially, we have
\[
	\int_{1/\sqrt u}^\infty \frac{\sh \xi}{\xi(\ch \xi)^{\frac 32}} C\(\!\sqrt{\mfrac{2u\ch \xi}{\pi}}\,\) d\xi \ll \int_0^\infty \frac{d\xi}{\sqrt{\ch\xi}} \ll 1.
\]
The two terms involving $\cos(2x\xi)/\xi$ are treated similarly, and are $\ll 1$.
The term involving the integral in \eqref{eq:G-x-xi-parts} and the Fresnel integral can be estimated trivially using \eqref{eq:G-x-xi-leq-1}; it is also $\ll 1$.

The final term,
\[
	\sqrt{u}\int_{1/\sqrt{u}}^\infty \frac{\sh \xi \cos(u\ch \xi)}{\xi \ch\xi} \int_0^x \pfrac{g(t)}{t}' \cos(2t\xi) \, dt\, d\xi,
\]
is more delicate.
We write the integral as
\[
	\sqrt{u} \int_{1/\sqrt{u}}^\infty \left\{\frac{1}{\xi\ch\xi} \int_0^x \pfrac{g(t)}{t}' \cos(2t\xi) \, dt \right\} \cdot \bigg\{\!\sh\xi\cos(u\ch \xi) d\xi\bigg\} = \sqrt{u} \int_{1/\sqrt{u}}^\infty U \, dV
\]
and integrate by parts.
By \eqref{eq:G-x-xi-leq-1}
we have
\[
	\sqrt{u} \, UV \Big|_{1/\sqrt u}^\infty \ll \sqrt u \cdot \frac{\sqrt u}{\ch(1/\sqrt u)} \cdot \frac{1}{u} \ll 1.
\]
Write $U'=R+S$, where
\[
	R = \pfrac{1}{\xi \ch\xi}' \int_0^x \pfrac{g(t)}{t}' \cos(2t\xi) \, dt, \qquad  S=\frac{2}{\xi\ch\xi} \int_0^x t\pfrac{g(t)}{t}' \sin(2t\xi) \, dt.
\]
By \eqref{eq:G-x-xi-leq-1} we have
\[
	R \ll \frac{1}{\xi \ch\xi} \left(\frac 1\xi + \th\xi\right) \ll \begin{cases}
		(\xi^2\ch\xi)^{-1} & \text{ if }\xi\leq 1, \\
		(\xi \ch\xi)^{-1} & \text{ if }\xi \geq 1.
	\end{cases}
\]
Applying Lemma \ref{lem:ganelius} and the estimate \eqref{eq:G-2-est} we find that
\[
	S \ll \frac{1}{\xi^2 \ch \xi}.
\]
Thus we have
\[
	\sqrt{u} \int_{1/\sqrt{u}}^\infty V\, dU 
	\ll \frac{1}{\sqrt{u}} \int_{1/\sqrt{u}}^1 \frac{d\xi}{\xi^2} + \frac{1}{\sqrt{u}} \int_1^\infty \frac{d\xi}{\xi\ch\xi} 
	\ll 1.
\]
We conclude that
\begin{equation} \label{eq:J-x-u-geq-1}
	J(x,u) \ll 1  \qquad \text{ for }u\geq 1.
\end{equation}
The proposition follows from 
\eqref{eq:ixjx}, \eqref{eq:J-x-u-leq-1}, and  \eqref{eq:J-x-u-geq-1}.
\end{proof}

\subsection{The second error estimate}
In the case when  $k=-1/2$ we must keep track of the dependence on $x$.
\begin{proposition}\label{prop:mve-int-est-neg}
Suppose that $k=-1/2$.
For $a>0$ we have
\begin{equation}
	I(x,a) \ll \sqrt{x} \times
	\begin{dcases}
		a^{-\frac 12}\big(\!\log(1/a)+1\big) & \text{ if }a\leq 1, \\
		a^{-1} & \text{ if }a\geq 1.
	\end{dcases}
\end{equation}
\end{proposition}

\begin{proof}
As before, we estimate the term involving $\cos(u\ch\xi)$ and we write
\begin{multline} \label{eq:int-est-neg-parts}
	\int_0^\infty \cos(2t\xi) \cos(u\ch\xi) \, d\xi  \\
	= \frac{1}{2t} \sin(2tT)\cos(u\ch T) + \frac{u}{2t}\int_0^T \sh \xi\sin(2t\xi) \sin(u\ch\xi) \, d\xi + R_T(u,t),
\end{multline}
with $R_T(u,t)\ll u^{-1/2}e^{-T/2}$.
Recalling \eqref{eq:g-t-approx} and \eqref{eq:g-t-approx0},
we have 
\[
	\int_0^x g(t) \int_a^\infty u^{-\frac 32} R_T(u,t) \, du \, dt \ll a^{-1} x^{\frac 52} e^{-T/2} \to 0 \quad \text{ as } T\to \infty
\]
and, by the first derivative estimate,
\[
	\int_0^x \frac{g(t)}{t} \sin(2tT) \, dt \int_a^\infty u^{-\frac 32}\cos(u\ch T) \, du \ll a^{-\frac 32}x^{\frac 32} e^{-T} \to 0 \quad \text{ as }T\to\infty.
\]

Our goal is to estimate
\begin{equation}
	J(x,a) := \int_0^{\infty} \sh\xi \, G_x(\xi) H(a,\xi) \, d\xi,
\end{equation}
where
\begin{equation} \label{eq:G-x-xi-def}
	G_x(\xi) := \int_0^x \frac{g(t)}{t} \sin(2t\xi) \, dt
\end{equation}
and
\begin{equation}\label{eq:haxi}
	H(a,\xi) := \int_a^\infty u^{-\frac 12} \sin(u\ch\xi) \, du.
\end{equation}
The integral defining $J(x,a)$ converges by an argument as in the proof of Proposition~\ref{prop:mve-int-est-pos}.

Recalling \eqref{eq:g-t-approx} and \eqref{eq:gtsign},   the first derivative estimate gives
\begin{equation} \label{eq:G-x-xi-est-1}
	G_x(\xi) \ll \frac{\sqrt{x}}{\xi}.
\end{equation}
For a better estimate, we integrate by parts in \eqref{eq:G-x-xi-def} to get
\begin{equation} \label{eq:G-x-xi-parts-2}
	G_x(\xi) = \frac{\pi \Gamma(\tfrac 54)^2}{2\xi} - \frac{g(x)}{x} \frac{\cos(2x\xi)}{2\xi} + \frac{1}{2\xi}\int_0^x \pfrac{g(t)}{t}' \cos(2t\xi) \, dt.
\end{equation}
We estimate the third term  in two ways.
Taking absolute values and using \eqref{eq:gtsign}, it is $\ll \sqrt x/\xi$; by \eqref{eq:G-1-est} and Lemma~\ref{lem:ganelius} it is $\ll 1/\xi^2$. 
Thus
\begin{equation} \label{eq:G-x-xi-int-est}
	\frac{1}{\xi}\int_0^x \pfrac{g(t)}{t}' \cos(2t\xi) \, dt \ll \min\(\frac{\sqrt x}\xi,\frac{1}{\xi^2}\).
\end{equation}
We also need two estimates for $H(a,\xi)$.
The first derivative estimate applied to \eqref{eq:haxi} gives
\begin{equation} \label{eq:H-a-xi-est-1}
	H(a,\xi) \ll \frac{1}{\sqrt a\, \ch \xi},
\end{equation}
while integrating \eqref{eq:haxi} by parts and applying the first derivative estimate gives
\begin{equation} \label{eq:H-a-xi-est-2}
	H(a,\xi) = \frac{\cos(a\ch\xi)}{\sqrt a \,\ch\xi} + O\(\frac{1}{a^{\frac 32}(\ch\xi)^2}\).
\end{equation}

First suppose that $a\leq 1$. Write
\[
	J(x,a) = J_1(x,a) + J_2(x,a) = \left(\int_0^{1/a} + \int_{1/a}^\infty\right) \sh\xi \, G_x(\xi) H(a,\xi) \, d\xi.
\]
By \eqref{eq:G-x-xi-est-1} and \eqref{eq:H-a-xi-est-1} we have
\begin{align*}
	J_1(x,a) \ll \sqrt{\frac xa} \, \left(\int_0^1 + \int_1^{1/a}\right) \frac{\th \xi}{\xi} \, d\xi \ll \sqrt \frac xa \, \big(1+\log(1/a) \big).
\end{align*}
By \eqref{eq:H-a-xi-est-2}, \eqref{eq:G-x-xi-est-1}, \eqref{eq:G-x-xi-parts-2}, and the second bound in \eqref{eq:G-x-xi-int-est} we have
\begin{align*}
	J_2(x,a) &=\frac{1}{\sqrt a}\int_{1/a}^\infty \th \xi \, G_x(\xi) \cos(a\ch\xi) \, d\xi 
	+ O\left(\frac{\sqrt x}{a^{\frac 32}} \int_{1/a}^\infty \frac{\sh\xi}{\xi(\ch\xi)^2} \, d\xi\right) \\
	&\ll \frac{1}{\sqrt a}\int_{1/a}^\infty \frac{\th \xi}{\xi} \cos(a\ch\xi) \, d\xi + \frac{g(x)}{x\sqrt a} \int_{1/a}^\infty \frac{\th \xi}{\xi} \cos(2x\xi)\cos(a\ch\xi) \, d\xi + \sqrt {x/a}.
\end{align*}
Applying the first derivative estimate to each integral above, we find that $J_2(x,a)\ll \sqrt{x/a}$, from which it follows that
\begin{equation} \label{eq:J-x-a-est-a-leq-1}
	J(x,a) \ll \sqrt{\frac xa} \, \big(\!\log(1/a)+1\big) \qquad \text{ for }a\leq 1.
\end{equation}

Now suppose that $a\geq 1$.
Using \eqref{eq:G-x-xi-est-1}, the contribution to $J(x,a)$ from the second term in \eqref{eq:H-a-xi-est-2} is
\[
	a^{-\frac 32}\int_0^\infty \frac{\sh \xi}{(\ch \xi)^2} G_x(\xi) \, d\xi \ll a^{-\frac 32} \sqrt x.
\]
We break the integral involving $\cos(a\ch \xi)$ at $1/\sqrt a$.
By \eqref{eq:G-x-xi-est-1} the contribution from the  initial segment is  $\ll a^{-1}\sqrt x$.
It remains to estimate
\[
	\frac{1}{\sqrt a}\int_{1/\sqrt a}^\infty \th\xi \, G_x(\xi) \cos(a\ch\xi) \, d\xi,
\]
which we break into three terms using \eqref{eq:G-x-xi-parts-2}.
For the first two terms  the first derivative estimate gives
\begin{gather*}
	\frac{1}{\sqrt a}\int_{1/\sqrt a}^\infty \frac{\th \xi}{\xi} \cos(a\ch\xi) \, d\xi \ll \frac 1a, \\
	\frac{g(x)}{x\sqrt a} \int_{1/\sqrt a}^\infty \frac{\th \xi}{\xi} \cos(2x\xi) \cos(a\ch \xi) \, d\xi \ll \frac{\sqrt x}{a}.
\end{gather*}
The remaining term,
\[
	\frac{1}{\sqrt a} \int_{1/\sqrt a}^\infty \frac{\th \xi}{\xi} \cos(a\ch\xi) \int_0^x \pfrac{g(t)}{t}' \sin(2t\xi) \, dt \, d\xi,
\]
requires more care, but we follow the estimate for the corresponding term in the proof of Proposition \ref{prop:mve-int-est-pos} using Lemma \ref{lem:ganelius} and \eqref{eq:G-2-est}. 
The details are  similar, and the contribution is 
$ \ll \frac{\sqrt x}{a}$,
which, together with \eqref{eq:J-x-a-est-a-leq-1}, completes the proof.
\end{proof}

\subsection{Proof of Theorem \ref{thm:mve-general}}
We give the proof for the case $k=1/2$, as the other case is analogous.
By Proposition \ref{prop:mve-int-est-pos} we have 
\[
I(x,a)\ll\min(1,a^{1/2}\log(1/a))\ll_\delta a^\delta\qquad \text{for any $0\leq \delta<1/2$}.
\]
So by \eqref{eq:mve-1} and \eqref{eq:mve-main-term} we have
\begin{equation*}
	n_\nu \sum_{r_j} \frac{|a_j(n)|^2}{\ch \pi r_j} h_x(r_j) = \frac{x^{\frac 32}}{3\pi^2} + O_\delta\left( x^\frac 12 + n^{\frac 12+\delta} \sum_{c>0} \frac{|S(n,n,c,\nu)|}{c^{3/2+\delta}} \right).
\end{equation*}
Setting $\delta=\beta-1/2$, with $\beta$ as in \eqref{eq:mve-weil-assumption}, this becomes
\begin{equation} \label{eq:mve-weighted}
	n_\nu \sum_{r_j} \frac{|a_j(n)|^2}{\ch \pi r_j} h_x(r_j) = \frac{x^{\frac 32}}{3\pi^2} + O\left( x^\frac 12 + n^{\beta+\epsilon}\right).
\end{equation}

The function $h_x(r)$ (recall \eqref{eq:hxrdef})  is a smooth approximation to the characteristic function of $[0,x]$.  We recall some properties from \cite[(5.4)--(5.7)]{kuznetsov} for $x\geq 1$:
\begin{align}
 \label{eq:h-x-r-R}
	&h_x(r) = \mfrac 2\pi \arctan(e^{\pi x-\pi r}) +O(e^{-\pi r}) &&\text{for $r\geq 1$}, \\
\label{eq:h-x-r-leq-x} 
	&h_x(r) = 1 + O(x^{-3} + e^{-\pi r})  &&\text{for }1\leq r\leq x-\log x, \\
\label{eq:h-x-r-geq-x}
	&h_x(r) \ll e^{-\pi (r-x)}  &&\text{for }r\geq x+\log x,  \\
\label{eq:h-x-r-geq-x2}
	&0<h_x(r) < 1  &&\text{for } r\geq 0.
\end{align}
For $x\geq 1$ and $0\leq r\leq x$ we bound $h_x(r)$ uniformly from below as follows.
If $r\geq 1$ we have
\[
	h_x(r) = 4\ch \pi r \int_0^x\frac{\sh \pi t}{\ch 2\pi r+\ch 2\pi t}\, dt
	\geq \frac{2\ch \pi r}{\ch 2\pi r}\int_0^r \sh \pi t\, dt > \mfrac 14,
\]
while if $r\leq 1$ we have
\[
	h_x(r) \geq 4\ch \pi r \int_0^1 \frac{\sh\pi t}{\ch 2\pi r+\ch2\pi t} \, dt \geq 4\int_0^1 \frac{\sh \pi t}{\ch2\pi +\ch2\pi t} \, dt > \mfrac{3}{100}.
\]
Similarly, we have $h_x(r)>2/5$ for $r\in i[0,1/4]$.

Set
\begin{equation} \label{eq:mve-A-x-def}
	A(x) := n_\nu \sum_{\substack{r_j\leq x \\ \text{or\ }r_j\in i\R}} \frac{ |a_j(n)|^2}{\ch\pi r_j}.
\end{equation}
Then
\[A(x)\ll  n_\nu \sum_{\substack{r_j\leq x \\ \text{or\ } r_j\in i\R}} \frac{|a_j(n)|^2}{\ch \pi r_j} h_x(r_j) \ll
	n_\nu \sum_{r_j} \frac{|a_j(n)|^2}{\ch \pi r_j} h_x(r_j) 
\]
which, together with \eqref{eq:mve-weighted}, gives
\begin{equation} \label{eq:A-x-trivial}
	A(x) \ll x^{\frac 32} + n^{\beta+\epsilon}.
\end{equation}
Let $A^*(x)$ denote the sum \eqref{eq:mve-A-x-def} restricted to   $1\leq r_j\leq x$.
By \eqref{eq:A-x-trivial}  with $x=1$ we have $A(x)=A^*(x)+O(n^{\beta+\epsilon})$, so in what follows we  work with $A^*(x)$.

Assume that $x\geq 2$ and set $X=x+2\log x$ so that $X-\log X\geq x$.
Then by \eqref{eq:h-x-r-leq-x} there exist $c_1,c_2>0$ such that 
\begin{equation*}
	n_\nu \sum_{1\leq r_j} \frac{|a_j(n)|^2}{\ch \pi r_j} h_{X}(r_j) \geq A^*(x)\left(1-\frac{c_1}{x^3}\right) - c_2 \, n_\nu \sum_{1\leq r_j\leq x} \frac{|a_j(n)|^2}{\ch\pi r_j} e^{-\pi r_j}.
\end{equation*}
On the other hand, \eqref{eq:mve-weighted} gives
\[
	n_\nu \sum_{1\leq r_j} \frac{|a_j(n)|^2}{\ch \pi r_j} h_{X}(r_j)
	= \frac{x^{\frac 32}}{3\pi^2} + O\(x^{\frac 12}\log x + n^{\beta+\epsilon}\).
\]
By \eqref{eq:A-x-trivial} we have
\begin{equation} \label{eq:a-star-error}
	n_\nu \sum_{1\leq r_j\leq x} \frac{|a_j(n)|^2}{\ch\pi r_j} e^{-\pi r_j}
	= n_\nu \sum_{\ell=1}^{\lfloor x\rfloor} \sum_{\ell\leq r_j\leq \ell+1} \frac{|a_j(n)|^2}{\ch \pi r_j} e^{-\pi r_j}
	\ll \sum_{\ell=1}^\infty e^{-\pi \ell} \big(\ell^{\frac 32} + n^{\beta+\epsilon}\big)
	 \ll n^{\beta+\epsilon},
\end{equation}
so that
\begin{equation} \label{eq:A-x-1}
	A^*(x) \leq \frac{x^{\frac 32}}{3\pi^2} + O\(x^{\frac 12}\log x + n^{\beta+\epsilon}\).
\end{equation}
Similarly, replacing $x$ by $x-\log x$ in \eqref{eq:mve-weighted} and using \eqref{eq:h-x-r-geq-x},  \eqref{eq:h-x-r-geq-x2} gives
\[
	A^*(x) \geq \frac{x^{\frac 32}}{3\pi^2} + O\(x^{\frac 12}\log x + n^{\beta+\epsilon}\) + O\Big( n_\nu \sum_{r_j>x} \frac{|a_j(n)|^2}{\ch\pi r_j}e^{-\pi(r_j-x+\log x)} \Big).
\]
Arguing as in \eqref{eq:a-star-error}, the error term is $\ll x^{-\frac 32} + n^{\beta+\epsilon} x^{-3}$, 
from which we obtain
\begin{equation} \label{eq:A-x-2}
	A^*(x) \geq \frac{x^{\frac 32}}{3\pi^2} + O\(x^{\frac 12}\log x + n^{\beta+\epsilon}\).
\end{equation}
Equations \eqref{eq:A-x-1} and \eqref{eq:A-x-2} together give Theorem \ref{thm:mve-general}.
\qed


\section{The Kuznetsov trace formula} \label{sec:KTF}
We require a variant of Kuznetsov's formula \cite[Theorem 1]{kuznetsov}, which relates weighted sums of the Kloosterman sums $S(m,n,c,\chi)$ on one side to spectral data on the other side (in this case, weighted sums of coefficients of Maass cusp forms).
In \cite{proskurin-new}, Proskurin generalized Kuznetsov's theorem to any weight $k$ and multiplier system $\nu$, but only for the case $m,n>0$.
Here we treat the mixed sign case when $k=1/2$ and $\nu=\chi$; the proof follows the same lines.
Blomer \cite{blomer} has  recorded this formula for twists of the  theta-multiplier by a Dirichlet character;
we provide a sketch of the proof in the present case since there are some details which are unique to this situation.

Suppose that $\phi: [0, \infty)\rightarrow \C$ is four times continuously differentiable and satisfies
\begin{equation}\label{phiproperties}
\phi(0)=\phi'(0)=0,\quad \phi^{(j)}(x)\ll_\epsilon x^{-2-\ep} \quad (j=0,\dots, 4) \quad \text{as $x\rightarrow\infty$}
\end{equation}
for some $\ep>0$.
Define
\begin{equation}\label{phicheckdef}
\check\phi(r):=\ch \pi r\int_0^\infty K_{2ir}(u)\phi(u)\frac{du}{u}.
\end{equation}

We state the analogue of the main result of \cite{proskurin-new} in this case.
\begin{theorem}\label{KTFmixedsign}
Suppose that $\phi$ satisfies the conditions \eqref{phiproperties}. 
As in \eqref{eq:uj-fix}
let $\rho_j(n)$ denote the coefficients of an orthornomal basis $\{u_j\}$ for $\mathcal{S}_{1/2}(1,\chi)$ with spectral parameters $r_j$.
If $m>0$ and $n<0$ then
\[
 \sum_{c>0}\frac{S(m, n,c,\chi)}c\phi\bigg(\frac{4\pi\sqrt{\tilde m|\tilde n|}}c\bigg)
=8\sqrt i\sqrt{\tilde m|\tilde n|}\sum_{r_j}\frac{\overline{\rho_j(m)}\rho_j(n)}{\ch \pi r_j}\check\phi(r_j).
\]
\end{theorem}

The proof involves evaluating the inner product of   Poincar\'e series in two ways.
Let $\tau=x+iy \in \H$ and $s=\sigma+it\in \C$.
Let $k\in \R$ and let $\nu$ be a multiplier system for $\Gamma$ in weight $k$.
For $m>0$ the Poincar\'e series $\mathcal U_m(\tau,s,k,\nu)$ is defined by
\begin{equation}\label{eq:U-m-def}
	\mathcal U_m(\tau,s,k,\nu) := \sum_{\gamma\in \Gamma_\infty \backslash\Gamma} \bar{\nu(\gamma)} \, j(\gamma,\tau)^{-k} \im(\gamma\tau)^s e(m_\nu \gamma \tau), \qquad \sigma>1.
\end{equation}
This function satisfies the transformation law
\[
	\mathcal U_m(\cdot,s,k,\nu) \sl_k \gamma = \nu(\gamma) \, \mathcal U_m(\cdot,s,k,\nu) \qquad \text{for all }\gamma\in \Gamma 
\]
and has Fourier expansion (see equation (15) of \cite{proskurin-new})
\begin{equation}\label{eq:U-m-fourier}
	\mathcal U_m(\tau, s, k, \nu)=y^s e( m_\nu \tau)+y^s\sum_{\ell\in \Z}\sum_{c>0}\frac{S(m, \ell, c,\nu )}{c^{2s}}B(c, m_\nu, \ell_\nu, y, s,k)e(\ell_\nu x),
\end{equation}
where 
\begin{equation}\label{bigBdef}
\begin{aligned}B(c, m_\nu, \ell_\nu, y, s,k)
=y\int_{-\infty}^\infty e\leg{-m_\nu}{c^2y(u+i)}\leg{u+i}{|u+i|}^{-k}e(-\ell_\nu yu)\frac{du}{y^{2s}(u^2+1)^s}.
\end{aligned}
\end{equation}

In the next two lemmas we will obtain expressions for the inner product
\[
	\left\langle \mathcal U_m\(\cdot, s_1,\tfrac12, \chi\),\  \overline{\mathcal U_{1-n}\(\cdot, s_2, -\tfrac12, \overline\chi\)}\right\rangle.
\]
The  first lemma 
is stated in such a way that
symmetry may be exploited in its proof.
\begin{lemma}\label{innerprodlemma1}  
Suppose that $(k, \nu)$ is one of the pairs $(1/2, \chi)$ or $(-1/2, \bar\chi)$.
Suppose     that $m>0$, $n<0$.
Then for  $\re s_1>1$, $\re s_2>1$ we have
\begin{multline} \label{eq:innerprodlemma1}
\left\langle \mathcal U_m\(\cdot, s_1, k, \nu\),\  \overline{\mathcal U_{1-n}\(\cdot, s_2, -k, \overline\nu\)}\right\rangle\\
=i^{-k}2^{3-s_1-s_2}\pi \pmfrac{m_\nu}{|n_\nu|}^\frac{s_2-s_1}2 \!\!\! \frac{\Gamma(s_1+s_2-1)}{\Gamma(s_1-k/2)\Gamma(s_2+k/2)}
\sum_{c>0}\frac{S(m, n,c,\nu)}{c^{s_1+s_2}}K_{s_1-s_2}\bigg(\frac{4\pi\sqrt{ m_\nu |n_\nu|}}c\bigg).
\end{multline}
\end{lemma}

\begin{proof}[Proof of Lemma~\ref{innerprodlemma1}]
We will prove \eqref{eq:innerprodlemma1} under the assumption $\re s_2>\re s_1$. 
Recall the definitions \eqref{eq:kloos_def}, 
\eqref{eq:alpha-nu-bar}, and \eqref{eq:n_nu_def},
and recall that 
\begin{equation}\label{eq:one_minus_n}
(1-n)_{\bar\nu}=-n_\nu.
\end{equation}
Replacing $\gamma$ by $-\gamma^{-1}$ in \eqref{eq:kloos_def} and
using $\nu(-I)=e^{-\pi i k}$, we find that
\[
	S(m,n,c,\nu)= e^{\pi i k} \, S(1-n,1-m,c,\overline\nu).
\]
The case $\re s_1>\re s_2$  then follows 
since both  sides of the putative identity are invariant  under the changes $s_1\leftrightarrow s_2$, $k\leftrightarrow -k$, $\nu\leftrightarrow \overline\nu$, $m\leftrightarrow 1-n$ (see \cite[(10.27.3)]{nist}).
The case $\re s_1=\re s_2$  follows by continuity.

Let $I_{m,n}(s_1,s_2)$ denote the inner product on the left-hand side of \eqref{eq:innerprodlemma1}.
Unfolding using  \eqref{eq:U-m-fourier} and \eqref{eq:U-m-def}, we 
 find that
 \begin{equation}\label{unfold}
 \begin{aligned}
I_{m,n}(s_1,s_2) &= 
\int_{\Gamma_\infty\backslash \H} \mathcal U_m(\tau, s_1, k, \nu)(\im \tau)^{s_2}e(-n_\nu)\, d\mu\\
&=\int_0^\infty y^{s_1+s_2-2}
\sum_{c>0}\frac{S(m,n,c,\nu)}{c^{2s_1}}B\(c,m_\nu,n_\nu,y,s_1,k\)e^{2\pi n_\nu y}\, dy.
\end{aligned}
\end{equation}
By \cite[(17)]{proskurin-new} 
we have
\[B(c,m_\nu,n_\nu,y,s_1,k)\ll_{s_1}y^{1-2\re s_1}e^{-\pi|n_\nu|y}.\]
Therefore the integral in \eqref{unfold} is majorized by the convergent (since $\re s_2>\re s_1$) integral
\[\int_0^\infty y^{\re s_2- \re s_1-1}e^{3\pi n_\nu y}\, dy.\]
Interchanging the integral and sum in \eqref{unfold} and using \eqref{bigBdef},
we find that 
\begin{multline*}
I_{m,n}(s_1,s_2)=
\sum_{c>0}\frac{S(m,n,c,\nu)}{c^{2s_1}}\int_{-\infty}^\infty \pmfrac{u+i}{|u+i|}^{-k}(u^2+1)^{-s_1}\\
\times\(\int_0^\infty y^{s_2-s_1-1}e\(\mfrac{-  m_\nu}{c^2 y(u+i)}-  n_\nu y u\)e^{2\pi  n_\nu y}\, dy\)\, du.
\end{multline*}
Setting $w=2\pi n_\nu(iu-1)y$ and $a=4\pi\sqrt{ m_\nu| n_\nu|}/c$,  the inner integral  becomes
\[ 
(2\pi   n_\nu(iu-1))^{s_1-s_2}\int_0^{(1-iu)\infty} \exp\(-w-\frac{a^2}{4w}\)\frac{dw}{w^{s_1-s_2+1}}.
\]
We may shift the path of integration to the positive real axis since
\[\int_{T}^{(1-iu)T}\exp\(-w-\frac{a^2}{4w}\)\frac{dw}{w^{s_1+s_2+1}}\ll T^{-\re (s_1-s_2)}e^{-T} \qquad \text{as $T\to\infty$}.\]
Using the integral representation \cite[(10.32.10)]{nist} for the $K$-Bessel function,
we find that the inner integral is equal to 
\[2\pmfrac{ m_\nu}{  |n_\nu|}^{\frac{s_2-s_1}2}(1-iu)^{s_1-s_2}c^{s_1-s_2}K_{s_1-s_2}\pmfrac{4\pi\sqrt{ m_\nu|n_\nu|}}c.\]
It follows that 

\begin{multline*}
	I_{m,n}(s_1,s_2) 
	= 2\pmfrac{m_\nu}{|n_\nu|}^\frac{s_2-s_1}2 \sum_{c>0}\frac{S(m,n,c,\nu)}{c^{s_1+s_2}}
	K_{s_1-s_2}\pmfrac{4\pi\sqrt{ m_\nu| n_\nu|}}c\\
	\times \int_{-\infty}^\infty
\pmfrac{u+i}{|u+i|}^{-k}(u^2+1)^{-s_1}(1-iu)^{s_1-s_2}\, du.
\end{multline*}
Lemma~\ref{innerprodlemma1} follows after using \cite[(5.12.8)]{nist} to evaluate the integral.
\end{proof}

Recall the definition
\begin{equation}\label{Lambdadef}
\Lambda(s_1, s_2, r) = \Gamma(s_1-\tfrac12-ir)\Gamma(s_1-\tfrac12+ir)\Gamma(s_2-\tfrac12-ir)\Gamma(s_2-\tfrac12+ir).
\end{equation}

\begin{lemma}\label{innerprodlemma2}Suppose that $\re s_1>1, \re s_2>1$.  
As in \eqref{eq:uj-fix}
let $\rho_j(n)$ and $r_j$ denote the coefficients and spectral parameters of an orthornomal basis $\{u_j\}$ for $\mathcal{S}_{1/2}(1,\chi)$.
If $m>0$ and $n<0$ then 
\begin{multline}
\left\langle \mathcal U_m\(\cdot, s_1, \tfrac12, \chi\),\  \overline{\mathcal U_{1-n}\(\cdot, s_2, -\tfrac12, \overline\chi\)}\right\rangle
=\frac{(4\pi  m_\chi)^{1-s_1}(4\pi|  n_\chi |)^{1-s_2}}{\Gamma(s_1-1/4)\Gamma(s_2+1/4)}\sum_{r_j}\overline{\rho_j(m)}\rho_j(n)\Lambda(s_1, s_2, r_j).
\end{multline}
\end{lemma}

\begin{proof}[Proof of Lemma~\ref{innerprodlemma2}]
In this situation there are no Eisenstein series.
So for any $f_1,f_2 \in \mathcal{L}_{\frac 12}(1,\chi)$ we have the Parseval identity \cite[(27)]{proskurin-new}
\begin{equation} \label{eq:parseval}
	\langle f_1, f_2\rangle=\sum_{r_j}\langle f_1, u_j\rangle\overline{\langle f_2, u_j\rangle}.
\end{equation}
By \cite[(32)]{proskurin-new} we have
\begin{equation} \label{eq:U-m-u-j}
	\left\langle\mathcal U_m\(z, s_1, \tfrac12, \chi\), u_j\right\rangle=\overline{\rho_j(m)}(4\pi  m_\chi)^{1-s_1}\frac{\Gamma(s_1-\frac12-ir_j)\Gamma(s_1-\frac12+ir_j)}
{\Gamma(s_1-\frac 14)}.
\end{equation}
Unfolding as in the last lemma using 
 the definition \eqref{eq:U-m-def} for $\mathcal U_{1-n}$ and the Fourier expansion \eqref{eq:uj-fix} for $u_j$ gives
\begin{align*}
	\overline{\left\langle\overline{\mathcal U_{1-n}\(\cdot, s_2, -\tfrac12, \overline\chi\)}, u_j\right\rangle}
	&=\left\langle u_j, \bar{\mathcal U_{1-n}\(\cdot, s_2, -\tfrac12, \overline\chi\)}\right\rangle\\
	&=\rho_j(n) \int_0^\infty y^{s_2-2}W_{-\frac 14, ir_j}(4\pi| n_\chi|y)e^{2\pi n_\chi y}\, dy.
\end{align*}
Using \cite[(13.23.4) and (16.2.5)]{nist} the latter expression becomes
\begin{equation*} 
	\rho_j(n)(4\pi | n_\chi|)^{1-s_2}\frac{\Gamma(s_2-\frac12-i r_j)\Gamma(s_2-\frac12+ir_j)}{\Gamma(s_2+\frac 14)}.
\end{equation*}
The lemma follows from this together with \eqref{eq:parseval} and \eqref{eq:U-m-u-j}.
\end{proof}

We are now ready to prove Theorem \ref{KTFmixedsign}.
Recall the notation $\tilde n=n_\chi$.

\begin{proof} [Proof of Theorem \ref{KTFmixedsign}]
Equating the right-hand sides of Lemmas~\ref{innerprodlemma1} (for $k=1/2$) and \ref{innerprodlemma2} and setting 
\[s_1=\sigma+ \mfrac {it}2,\qquad s_2=\sigma- \mfrac {it}2 \]
we obtain (for $\sigma>1$)
\begin{multline}\label{innerprodcombine}
i^{-\frac12}2^{3-2\sigma}\pi\Gamma(2\sigma-1)\sum_{c>0}\frac{S(m, n,c, \chi)}{c^{2\sigma}}K_{it}\leg{4\pi\sqrt{\tilde m|\tilde n|}}c\\
=(4\pi\tilde m)^{1-\sigma}(4\pi|\tilde n|)^{1-\sigma}\sum_{r_j}\overline{\rho_j(m)}\rho_j(n)
\Lambda(\sigma+\tfrac{it}2, \sigma-\tfrac{it}2, r_j).
\end{multline}

We justify the  substitution of  $\sigma=1$ in \eqref{innerprodcombine}.
By \cite[(10.45.7)]{nist} we have
\[
	K_{it}(x)\ll(t\sh \pi t)^{-1/2}  \quad\text{as $x\to 0$.}
\]
Using \eqref{eq:trivial_est}
we see that the left side of \eqref{innerprodcombine} converges absolutely uniformly 
for $\sigma\in[1, 2]$.
For the right hand side we use the inequality $|\rho_j(m)\rho_j(n)|\leq |\rho_j(m)|^2+|\rho_j(n)|^2$.
The argument which follows Lemma~\ref{lem:pre_kusnetsov} shows that the right side 
converges  absolutely uniformly for $\sigma\in[1, 2]$.

Using \eqref{eq:lambda_def}, 
we find that 
\begin{equation}\label{pluginone}
i^{-\frac12}\sum_{c>0}\frac{S(m, n,c, \chi)}{c^{2}}K_{it}\leg{4\pi\sqrt{\tilde m|\tilde n|}}c
=\frac\pi2\sum_{r_j}\frac{\overline{\rho_j(m)}\rho_j(n)}{\ch \pi(\frac t2-r_j)\ch \pi(\frac t2+r_j)}.
\end{equation}
Letting $\phi$ be a function satisfying the conditions \eqref{phiproperties},
 multiply both sides of \eqref{pluginone} by
\[\mfrac2{\pi^2} \, t\sh \pi t\int_0^\infty K_{it}(u)\phi(u)\frac{du}{u^2}\]
and integrate from $0$ to $\infty$. 
We apply the Kontorovich-Lebedev transform (\cite[(35)]{proskurin-old} or \cite[(10.43.30-31)]{nist})
\begin{equation*} 
	\frac2{\pi^2}\int_0^\infty K_{it}(x)t\sh \pi t\,\(\int_0^\infty K_{it}(u)\phi(u)\frac{du}{u^2}\)\, dt=\frac{\phi(x)}{x}
\end{equation*}
to the left-hand side of \eqref{pluginone} and the transform \cite[(39)]{proskurin-old}
\begin{equation*} 
\int_0^\infty\frac{t\sh\pi t}{\ch \pi(\frac t2+r)\ch \pi(\frac t2-r)}\(\int_0^\infty K_{it}(u)\phi(u)\frac{du}{u^2}\)\, dt
=\frac2{\ch\pi r}\check\phi(r)
\end{equation*}
(recalling the definition \eqref{phicheckdef})
to the right-hand side.
Then   \eqref{pluginone} becomes
\[
	\frac{i^{-1/2}}{4\pi\sqrt{\tilde m|\tilde n|}} \sum_{c>0}\frac{S(m, n,c, \chi)}c\phi\leg{4\pi\sqrt{\tilde m|\tilde n|}}c
= \frac2\pi\sum_{r_j}\frac{\overline{\rho_j(m)}\rho_j(n)}{\ch \pi r_j}\check\phi(r_j),
\]
and Theorem~\ref{KTFmixedsign} follows. 
\end{proof}


\section{A theta lift for Maass cusp forms}\label{sect:shimura}

In this section we construct a version of the Shimura correspondence as in \cite[\S 3]{sarnak-additive} and \cite[\S 4]{katok-sarnak} for Maass forms of weight $1/2$ on $\Gamma_0(N)$ with the eta multiplier twisted by a Dirichlet character.

Throughout this section, $N\equiv 1\bmod {24}$ is a positive integer and $\psi$ is an even Dirichlet character modulo $N$.
When working with the Shimura correspondence and the Hecke operators, it is convenient to write the Fourier expansion of $G\in \mathcal{S}_{\frac 12}(N,\psi\chi,r)$ as
\begin{equation} \label{eq:shim-G-exp}
	G(\tau) = \sum_{n\equiv 1(24)} a(n) W_{\frac{\sgn(n)}4,ir}\pmfrac{\pi|n|y}{6} e\pmfrac{nx}{24}.
\end{equation}
When $G=u_j$ as in \eqref{eq:uj-fix} we have the relation
\begin{equation}
	\rho_j\pmfrac{n+23}{24} = a(n).
\end{equation}

For each $t\equiv 1\bmod {24}$, the following theorem gives a map 
\[
	S_t:\mathcal{S}_{\frac 12}(N,\psi\chi,r)\to \mathcal{S}_{0}(6N,\psi^2,2r).
\]
We will only apply this theorem in the case $N=1$.
The proof in the case $N=t=1$ is simpler than the general proof given below; we must work in this generality since the  proof in the case $t>1$ requires the map $S_1$ on $\Gamma_0(t)$ forms.

\begin{theorem} \label{thm:shimura}
Let $N\equiv 1\bmod{24}$ be squarefree.
Suppose that $G\in \mathcal{S}_{\frac 12}(N,\psi\chi,r)$ with $r\neq i/4$ has Fourier expansion \eqref{eq:shim-G-exp}.
Let $t\equiv 1\bmod{24}$ be a squarefree positive integer and define coefficients $b_t(n)$ by the relation
\begin{equation} \label{eq:shim-coeff-relation}
	\sum_{n=1}^\infty \frac{b_t(n)}{n^s} = L\left(s+1,\psi\ptfrac t\bullet\right)  \sum_{n=1}^\infty \pmfrac {12}n \frac{a(tn^2)}{n^{s-\frac 12}},
\end{equation}
where $L(s,\psi)$ is the usual Dirichlet $L$-function.
Then the function $S_t(G)$ defined by
\[
	(S_tG)(\tau) := \sum_{n=1}^\infty b_t(n) W_{0,2ir}(4\pi ny) \cos(2\pi n x)
\]
is an even Maass cusp form in $\mathcal{S}_0(6N,\psi^2,2r)$.
\end{theorem}

An important property of the Shimura correspondence is Hecke equivariance.
Since we only need this fact for $N=1$, we state it only in that case.

\begin{corollary}\label{cor:shim_hecke_commute}
Suppose that $G\in \mathcal{S}_{\frac 12}\left(1,\chi,r\right)$ with $r\neq i/4$ and that $t\equiv 1\bmod {24}$ is a squarefree positive integer.
Then for any prime $p\geq 5$ we have
\[
	T_p \, S_t(G)=\pmfrac{12}p S_t(T_{p^2}G).
\]
\end{corollary}
\begin{proof}
We prove this in the case $p\nmid t$; the other case is similar, and easier.
Let $G$ have coefficients $a(n)$
as in \eqref{eq:shim-G-exp}.
The coefficients $b_t(n)$ of $S_t(G)$ are given by
\[
	b_t(n) = \sum_{jk=n} \pmfrac tj \pmfrac{12}k \mfrac{\sqrt k}j a(tk^2),
\]
and the coefficients $A(n)$ of $T_{p^2} G$ are given by
\[
	A(n) = p \, a(p^2 n) + p^{-\frac 12} \pmfrac{12n}{p} a(n) + p^{-1} a(n/p^2).
\]
We must show that
\begin{equation} \label{eq:hecke-shim-1}
	p^{\frac 12} b_t(pn) + p^{-\frac 12} b_t(n/p) = \pmfrac{12}{p} \sum_{jk=n} \pmfrac tj \pmfrac{12}k \mfrac{\sqrt k}{j} A(tk^2).
\end{equation}
Writing $n=p^\alpha n'$ with $p\nmid n'$, the left-hand side of \eqref{eq:hecke-shim-1} equals
\[
	p^{-\frac 12} \sum_{\ell=0}^{\alpha+1} S_\ell + p^{\frac 12} \sum_{\ell=0}^{\alpha-1} S_\ell, 
	\quad \text{where} \quad 
	S_\ell = p^{-\alpha}\pmfrac tp^{\alpha+1-\ell} \pmfrac{12}p^{\ell} p^{\frac{3\ell}2} \sum_{jk=n'} \pmfrac tj \pmfrac{12}k \mfrac{\sqrt k}j a(tp^{2\ell}k^2).
\]
A computation shows that the right-hand side of \eqref{eq:hecke-shim-1} equals
\[
	p^{-\frac 12} \sum_{\ell=0}^\alpha S_{\ell+1} + p^{-\frac 12} S_0 + p^{\frac 12} \sum_{\ell=1}^\alpha S_{\ell-1},
\]
and the corollary follows.
\end{proof}

Using Theorem~\ref{thm:shimura} we can rule out the existence of exceptional eigenvalues in $\mathcal S_{1/2}(1,\chi)$
 and obtain a lower bound on the second smallest eigenvalue $\lambda_1=\frac14+r_1^2$
 (we note that Bruggeman \cite[Theorem 2.15]{BruggemanIII} obtained $\lambda_1>\frac14$ using different methods).
Theorem~\ref{thm:shimura} shows that $2r_1$ is bounded below by the smallest spectral parameter 
for $\mathcal S_0(6,\bm 1)$.
Huxley \cite{huxley_kloostermania,huxley} studied the problem of exceptional eigenvalues in weight $0$ for subgroups of $\SL_2(\Z)$ whose fundamental domains are ``hedgehog'' shaped.
On page 247 of \cite{huxley_kloostermania} we find a lower bound which, for $N=6$, gives
$2r_1>0.4$.
Computations of Str\"omberg
\cite{lmfdb}
suggest (since $2r_1$ corresponds to an even Maass cusp form) that
$2r_1 \approx 3.84467$.
Using work of Booker and Str\"ombergsson \cite{booker-strombergsson-stf}  on Selberg's eigenvalue conjecture we can prove the following lower bound.
\begin{corollary}\label{cor:shimura}
Let $\lambda_1=\frac 14+r_1^2$ be as above. Then $r_1>1.9$.
\end{corollary}

\begin{proof}
Let $\{\tilde r_j\}$ denote the set of spectral parameters corresponding to even forms in $\mathcal S_0(6,\bm 1)$.
We show that each $\tilde r_j>3.8$
using the method described in \cite[Section 4]{booker-strombergsson-stf} in the case of level $6$ and trivial character.
We are thankful to the authors of that paper for providing the numerical details of this computation.
Given a suitable test function $h$ such that $h(t)\geq 0$ on $\R$ and $h(t)\geq 1$ on $[-3.8,3.8]$, we compute via an explicit version of the Selberg trace formula \cite[Theorem 4]{booker-strombergsson-stf} that
\begin{equation} \label{eq:sum-h}
	\sum_{\tilde r_j} h(\tilde r_j) < 1,
\end{equation}
so we cannot have $\tilde r_j \leq 3.8$ for any $j$.
Let
\[
	f(t) = \pfrac{\sin t/6}{t/6}^2 \, \sum_{n=0}^{2} x_n \cos \pmfrac{n t}3,
\]
where 
\[
	x_0 = 1.167099688946872, \qquad x_1 = 1.735437017086616, \qquad x_2 = 0.660025094420283.
\]
With $h=f^2$,
we compute that the sum in \eqref{eq:sum-h} is approximately $0.976$. 
\end{proof}

The proof of Theorem \ref{thm:shimura} occupies the remainder of this section. 
We first modify the theta functions introduced by Shintani \cite{shintani} and Niwa \cite{niwa}.
Next, we construct the Shimura lift for $t=1$ and derive the relation \eqref{eq:shim-coeff-relation} in that case.
The function $S_t(G)$ is obtained by applying the lift $S_1$ to the form $G(t\tau)$.
We conclude the section by showing that $S_t(G)$ has the desired level.

\subsection{The theta functions of Shintani and Niwa}

In this subsection we adopt the notation of \cite{niwa} for easier comparison with that paper.
Let $Q=\mfrac{1}{12N} \left(\begin{smallmatrix} &&-2 \\ &1& \\ -2&& \end{smallmatrix}\right)$, and for $x,y\in \R^3$ define
\[
	\langle x,y \rangle = x^T Q y = \mfrac{1}{12N}(x_2y_2-2x_1y_3-2x_3y_1).
\]
The signature of $Q$ is $(2,1)$.
Let $L\subset L'\subset L^*$ denote the lattices
\begin{alignat*}{5}
	L &=& N\Z& \,\oplus\,& 12N\Z& \,\oplus\, &6N\Z&, \\
	L' &=& \Z& \,\oplus\,& N\Z& \,\oplus\, &6N\Z&, \\
	L^* &=& \Z& \,\oplus\,& \Z& \,\oplus\, &6\Z&.
\end{alignat*}
Then $L^*$ is the dual lattice of $L$, and for $x,y\in L$ we have $\langle x,y \rangle \in \Z$ and $\langle x,x \rangle \in 2\Z$.
Let $f:\R^3\to \C$ be a Schwarz function satisfying the conditions of \cite[Corollary 0]{niwa} for $\kappa=1$.
If $\psi$ is a character mod $N$ and $h=(h_1,Nh_2,0)\in L'/L$, define 
\[
	\psi_1(h):= \psi(h_1)\pmfrac{12}{h_2}.
\]
For $z=u+iv$ we define
\[\theta(z,f,h) :=v^{-\frac 14} \sum_{x\in L} \left[r_0(\sigma_z)f\right](x+h)\]
and
\begin{equation}\label{thetadef}
	\theta(z,f) := \sum_{h\in L'/L} \bar\psi_1(h) \theta(z,f,h),
\end{equation}
where $\sigma\mapsto r_0(\sigma)$ is the Weil representation (see \cite[p. 149]{niwa}) and 
\[
	\sigma_z := \pMatrix{v^{\frac 12}}{u v^{-\frac 12}}{0}{v^{-\frac 12}}, \quad z=u+iv \in \H.
\]
The following lemma gives the transformation law for $\theta(z,f)$.
\begin{lemma}\label{thetatrans}
Suppose that $N\equiv 1 \pmod {24}$ and that $\sigma=\pmatrix abcd\in \Gamma_0(N)$. With $\theta(z, f)$ defined as in \eqref{thetadef} we have
\begin{equation}
	\theta(\sigma z,f) = \bar\psi(d) \pmfrac Nd \chi(\sigma) (cz+d)^{\frac 12} \theta(z,f),
\end{equation}
where $\chi$ is the eta multiplier \eqref{eq:eta-mult-def}.
\end{lemma}

In the proof of Lemma \ref{thetatrans} we will encounter the quadratic Gauss sum
\[
	G(a,b,c) = \sum_{x\bmod c} e\pfrac{ax^2+bx}{c}.
\]
With $g=(a,c)$, we have
\begin{equation} \label{eq:gauss-gcd}
	G(a,b,c) = 
	\begin{cases}
		g \, G(a/g,b/g,c/g) & \text{ if } g\mid b, \\
		0 & \text{ otherwise}.
	\end{cases}
\end{equation}
We require the following identity relating the Dedekind sums $s(d,c)$ to these Gauss sums.

\begin{lemma}\label{lem:dedekind-gauss}
Suppose that $(c,d)=1$ and that $k\in\Z$.
Let $\bar d$ satisfy $d\bar d\equiv 1\pmod c$. Then
\begin{equation} \label{eq:dedekind-gauss}
	\sqrt{12c} \, \pmfrac{12}k e\bigg(\frac{\bar d(k^2-1)}{24c}\bigg) e^{\pi i s(d,c)}
	= \sum_{h\bmod 12} \pmfrac{12}{h} e\bigg(\frac{-d(\frac{h^2-1}{24})}{c} - \frac{hk}{12c}\bigg) G(-6d,dh+k,c).
\end{equation}
\end{lemma}

\begin{proof}
First suppose that $(k,6)=1$.
If $c$ is even and $d\bar d\equiv 1\pmod {2c}$, or if $c$ is odd, then
by \cite[Lemma 5 and (4.1)]{andersen-singular} we have 
\begin{multline} \label{eq:dg-1}
\sqrt{12c} \, \pmfrac{12}k e\bigg(\frac{\bar d(k^2-1)}{24c}\bigg) e^{\pi i s(d,c)} \\
	= e\pmfrac{k}{12c} \sum_{j\bmod 2c} e\Big(\mfrac{-3dj^2-dj+j(c+k)}{2c}\Big) + e\pmfrac{-k}{12c} \sum_{j\bmod 2c} e\Big(\mfrac{-3dj^2-dj+j(c-k)}{2c}\Big).
\end{multline}
If $c$ is even and $d\bar d\not\equiv 1\pmod{2c}$ then, applying \eqref{eq:dg-1} with $\bar d$ replaced by $\bar d+c$, we see that \eqref{eq:dg-1} is true in this case as well.
Splitting each of the sums in \eqref{eq:dg-1} into two sums by writing $j=2x$ or $j=2x+1$ shows that \eqref{eq:dg-1} equals the right-hand side of \eqref{eq:dedekind-gauss}.

It remains to show that the right-hand side of $\eqref{eq:dedekind-gauss}$ is zero when $\delta:=(k,6)>1$.
If $(\delta, c)>1$ then \eqref{eq:gauss-gcd} shows that 
$\pfrac{12}{h}G(-6d,dh+k,c)=0$.
If $(\delta,c)=1$,
write
\[
	A(h) = e\bigg(\frac{-d(\frac{h^2-1}{24})}{c} - \frac{hk}{12c}\bigg) G(-6d,dh+k,c).
\]
If $\delta=2$ or $6$ then a computation shows that, replacing $x$ by $x-\bar 2$ in $G(-6d,dh+k,c)$, we have $A(h)=A(h-6)$ for $(h,6)=1$.
Similarly, if $\delta=3$ and $c$ is odd then, replacing $x$ by $x+\bar 6 h$, we find that $A(h)=-A(-h)$.
Finally, if $\delta=3$ and $c$ is even then, replacing $x$ by $x - \bar 3$, we find that $A(1)=A(5)$ and $A(7)=A(11)$.
In each case we find that
\[
	\sum_{h\bmod 12} \pmfrac{12}h A(h) = 0. \qedhere
\]
\end{proof}

\begin{proof}[Proof of Lemma \ref{thetatrans}]
Since $\sigma_{z+1}=\pmatrix11 01\sigma_z$,
Proposition 0 of \cite{niwa} gives
\[
	\theta(z+1,f) = \sum_{h\in L'/L} e\Big(\mfrac{Nh_2^2}{24}\Big) \bar\psi_1(h) \theta(z,f,h) = e\pmfrac{1}{24} \theta(z,f).
\]
For $\sigma=\pmatrix abcd \in \Gamma_0(N)$ with $c>0$ we have, by Corollary 0 of \cite{niwa},
\[
	\theta(\sigma z, f) = (cz+d)^{\frac 12} \sum_{k\in L^*/L} \theta(z,f,k) \cdot \sqrt{-i} \sum_{h\in L'/L} \bar\psi_1(h) c(h,k)_\sigma,
\]
where
\begin{equation} \label{eq:c-h-k-1}
	c(h,k)_\sigma = \frac{1}{(Nc)^{\frac32}\sqrt{12}} \, e\left(\frac{d\langle k,k \rangle}{2c}\right) \sum_{r\in L/cL} e\left(\frac{a\langle h+r, h+r\rangle}{2c} - \frac{\langle k,h+r \rangle}{c}\right).
\end{equation}
For $h\in L'/L$, $k\in L^*/L$, and $r\in L/cL$, we write $h=(h_1,Nh_2,0)$, $k=(k_1,k_2,6k_3)$, and $r=(Nr_1,12Nr_2,6Nr_3)$.
Then the sum on $r$ in \eqref{eq:c-h-k-1} equals
\begin{equation} \label{eq:r-sum}
	e\pfrac{aNh_2^2-2h_2k_2}{24c} \, G(6aN,aNh_2-k_2,c) \, S(c) ,
\end{equation}
where (recalling that $N \mid c$), 
\begin{align}
	S(c) &= e\pfrac{h_1k_3}{Nc} \sum_{r_1,r_3(c)} e\pfrac{-aNr_1r_3-ah_1r_3+r_1k_3+r_3k_1}{c} \notag \\
	&=
	\begin{dcases}
		Nc \, e\pfrac{dk_1k_3'}{c} & \text{ if } k_3=Nk_3' \text{ and } k_1\equiv ah_1\pmod{N}, \\
		0 & \text{ otherwise.} \label{eq:S-c-eval}
	\end{dcases}
\end{align}
So by \eqref{eq:r-sum}, \eqref{eq:S-c-eval}, and \eqref{eq:gauss-gcd} we have $c(h,k)_\sigma=0$ unless $k=(k_1,Nk_2',6Nk_3') \in L'$, in which case $k_1\equiv ah_1\pmod N$.
For such $k$ we have
\begin{gather*}
	\sum_{h\in L'/L} \bar\psi_1(h) c(h,k)_\sigma 
	= \frac{\bar\psi(dk_1)}{\sqrt{12c'}} e\pfrac{d(k_2')^2}{24c'} \sum_{h_2\bmod 12} \pmfrac{12}{h_2} e\pfrac{ah_2^2-2h_2k_2'}{24c'} 
	G(6a,ah_2-k_2',c')
	.
\end{gather*}
Applying Lemma \ref{lem:dedekind-gauss} we have
\begin{align*}
	\sqrt{-i}\sum_{h\in L'/L} \bar\psi_1(h) c(h,k)_\sigma 
	&= \bar\psi(d) \bar\psi_1(k) \sqrt{-i} \, e\pfrac{a+d}{24c'} e^{\pi i s(-a,c')}.
\end{align*}
By (68.4) and (68.5) of \cite{rademacher-book} we have $s(-a,c')=-s(a,c')=-s(d,c')$.
Therefore
\[
\theta(\sigma z,f) = \bar\psi(d) \chi \ppMatrix a{Nb}{c'}d (cz+d)^{\frac 12} \theta(z,f).
\]
By \eqref{eq:chi-kronecker-symbol} and the assumption $N\equiv 1\pmod{24}$, we have $\chi \big(\pmatrix a{Nb}{c'}d\big) = \pfrac Nd \chi(\pmatrix abcd)$, from which the lemma follows.
\end{proof}

For $w=\xi+i\eta\in \H$ and $0\leq \theta<2\pi$, let
\[
	g=g(w,\theta) = \pMatrix 1\xi01 \pMatrix {\eta^{\frac 12}}00{\eta^{-\frac 12}} \pMatrix {\cos\theta}{\sin\theta}{-\sin\theta}{\cos\theta} \in \SL_2(\R).
\]
The action of $\SL_2(\R)$ on $\R^3$ is given by $g.x=x'$, where
\begin{equation} \label{eq:sl2r-action}
	\pMatrix{x_1'}{x_2'/2}{x_2'/2}{x_3'} = g\pMatrix{x_1}{x_2/2}{x_2/2}{x_3}g^T.
\end{equation}
Since $\langle x,x \rangle=-\frac 1{3N} \det\pmatrix{x_1}{x_2/2}{x_2/2}{x_3}$ we have $\langle gx,gx \rangle = \langle x,x \rangle$.
The action of $\SL_2(\R)$ on functions $f:\R^3\to \C$ is given by $gf(x):=f(g^{-1}x)$.

We specialize to $f=f_3$, with
\[
	f_3(x) = \exp\left(-\mfrac{\pi}{12N}\left(2x_1^2+x_2^2+2x_3^2\right)\right)
\]  
 as in \cite[Example 3]{niwa}, and we consider the function $\theta(z,gf_3)$.
Let $k(\theta)=\pmatrix{\cos \theta}{\sin\theta}{-\sin\theta}{\cos\theta}$.
Since $f_3(k(\theta)x)=f_3(x)$, the function $\theta(z,gf_3)$ is independent of the variable $\theta$.
Therefore it makes sense to define
\begin{equation} \label{eq:theta-def}
	\vartheta(z,w) := \theta(z,g(w,0)f_3) = v^{\frac 12} \sum_{x\in L'} \bar\psi_1(x) e\left(\mfrac{u}{2}\langle x,x \rangle\right) f_3\left(\sqrt{v} \, \sigma_w^{-1}x\right),
\end{equation}
where the second equality follows from
\begin{equation*}
	[r_0(\sigma_z)f](x)=v^{\frac34}e\(\tfrac u2\langle x,x\rangle\)f(\sqrt{v}x).
\end{equation*}

To determine the transformation of $\vartheta(z,w)$ in the variable $w$, we use the relation
\[
	\sigma_{\gamma w}=\gamma \, \sigma_w \, k\big(\!\arg(cw+d)\big), \qquad \gamma=\pmatrix abcd \in \SL_2(\R).
\]
Suppose that $\gamma=\pmatrix abcd\in \Gamma_0(6N)$. If $x'=\gamma x$ then by \eqref{eq:sl2r-action} we have
\begin{align}
	x_1' &= a^2 \, x_1 + ab \, x_2 + b^2 \, x_3, \label{eq:x1'} \\
	x_2' &= 2ac \, x_1 + (1+2bc) \, x_2 + 2bd \, x_3, \label{eq:x2'} \\
	x_3' &= c^2 \, x_1 + cd \, x_2 + d^2 \, x_3. \label{eq:x3'}
\end{align}
It follows that $\gamma L'=L'$ and that $\psi_1(x')=\psi^2(a)\psi_1(x)$. Thus
\begin{equation} \label{eq:theta-w-transform}
	\vartheta(z,\gamma w) = \psi^2(d) \vartheta(z,w) \qquad \text{for all }\gamma\in \Gamma_0(6N).
\end{equation}
A computation shows that replacing $w=\xi+i\eta$ by $w'=-\xi+i\eta$ in $f_3(\sqrt{v}\sigma_w^{-1}x)$ has the same effect as replacing $x_2$ by $-x_2$, so
\begin{equation} \label{eq:theta-even}
	\vartheta(z,-\xi+i\eta) = \vartheta(z,\xi+i\eta).
\end{equation}

\subsection{Proof of Theorem \ref{thm:shimura} in the case $t=1$}

Suppose that $F\in \mathcal{S}_{\frac 12}(N,\bar\psi\ptfrac N\bullet \chi, r)$ with $r\neq i/4$.
For $z=u+iv$ and $w=\xi+i\eta$ we define
\begin{equation} \label{eq:Psi-F-def}
	\Psi_F(w) := \int_{\mathcal D} v^{\frac 14} \bar{\vartheta(z,w)} F(z) d\mu,
\end{equation}
where $\mathcal D=\Gamma_0(N)\backslash\H$ and $d\mu=\frac{du\,dv}{v^2}$.
Note that the integral is well-defined by Lemma~\ref{thetatrans}.
There is no issue with convergence since $\vartheta$ and $F$ are both rapidly decreasing at the cusps.

Let $D^{(w)}$ denote the Casimir operator for $\SL_2(\R)$ defined by 
\[
	D^{(w)} = \eta^2 \left( \frac{\partial^2}{\partial \xi^2} + \frac{\partial^2}{\partial \eta^2} \right) - \eta \frac{\partial^2}{\partial\xi\partial\theta}.
\]
To show that $\Psi_F$ is an eigenform of $\Delta_0^{(w)}$, we use 
 Lemmas 1.4 and 1.5 of \cite{shintani}, which give
\[
	v^{-\frac 12} r_0(\sigma_z) D^{(w)}f = \Delta_{\frac 12}^{(z)}\left[v^{-\frac 12} r_0(\sigma_z) f\right].
\]
Since $\Psi_F$ is constant with respect to $\theta$, we have $\Delta_0^{(w)}\Psi_F = D^{(w)}\Psi_F$.
By the lemma on p. 304 of \cite{sarnak-additive}, it follows that
\[
	\Delta_0^{(w)}\Psi_F + \left(\mfrac 14+(2r)^2\right)\Psi_F = 0.
\]
This, together with the transformation law \eqref{eq:theta-w-transform} shows that $\Psi_F$ is a Maass form.
The following lemma shows that $\Psi_F$ is a cusp form.
This is the only point  where we use the assumption that  $N$ is squarefree.  It would be possible to 
remove this assumption with added complications by arguing as in \cite{cipra}.
For the remainder of this section, to avoid confusion with the eta function, we write $w=\xi+iy$.

\begin{lemma} \label{lem:cusps}
Suppose that $N\equiv 1\pmod{24}$ is squarefree and that $F\in \mathcal{S}_{\frac 12}(N,\bar\psi\ptfrac N\bullet \chi, r)$.
Let $\eta_N(w):=y^{\frac 14}\eta(Nw)$. 
For each cusp $\mathfrak a=\gamma_{\mathfrak a}\infty$ there exists $c_{\mathfrak a}\in \C$ such that as $y\to\infty$ we have
\begin{equation} \label{eq:cusps}
	(\Psi_F \sl_0 \gamma_{\mathfrak a}) (iy) =
	\begin{cases}
		c_{\mathfrak a} \, \left\langle \eta_N, F \right\rangle \, y + O(1) & \text{ if $\psi$ is principal}, \\
		O(1) & \text{ otherwise}.
	\end{cases}
\end{equation}
\end{lemma}

\begin{proof}
For each $d\mid 6N$, let $f\mapsto f\sl_0 W_d$ be the  Atkin-Lehner involution \cite[\S 2]{atkin-lehner} given by any matrix $W_d \in \SL_2(\R)$ of the form
\[
	W_d = \pMatrix{\sqrt d \, \alpha}{\beta/\sqrt{d}}{6N\gamma/\sqrt d}{\sqrt d \, \delta}, \qquad \alpha,\beta,\gamma,\delta \in \Z.
\]
Since $N\equiv 1\pmod{24}$ is squarefree, every cusp of $\Gamma_0(6N)$ is of the form $W_d \infty$ for some $d$.
Thus it suffices to establish \eqref{eq:cusps} for $\gamma_{\mathfrak a}=W_d$.

By \eqref{eq:x1'}--\eqref{eq:x3'} we have $W_d L' = L'$,
so
\[
	\vartheta_d(z,w) := \vartheta(z,W_d w) = v^{\frac 12} \sum_{x\in L'} \bar{\psi}_1(W_d x) e\left(\mfrac{u}{2}\langle x,x \rangle\right) f_3\left(\sqrt{v} \, \sigma_w^{-1}x\right).
\]
To determine the asymptotic behavior of $\vartheta_d(z,iy)$ as $y\to\infty$, we follow the method of Cipra \cite[\S 4.3]{cipra}.
We write
\begin{equation} \label{eq:theta-d-z-iy}
	\vartheta_d(z,iy) = v^{\frac 12} \sum_{\substack{x\in L' \\ x_3=0}} \bar{\psi}_1(W_d x) e\Big(\mfrac{z}{24N}x_2^2\Big) \exp\Big(-\mfrac{\pi v}{6Ny^2}x_1^2\Big) + \varepsilon(z,y),
\end{equation}
where
\begin{equation*}
	|\varepsilon(z,y)| \leq v^{\frac 12} \sum_{x_1\in \Z} \exp\Big( \mfrac{-\pi v}{6N y^2} x_1^2 \Big) \sum_{x_2\in \Z} \exp\Big( \mfrac{-\pi v}{12N} x_2^2 \Big) \sum_{x_3\neq 0} \exp\Big( \mfrac{-\pi v y^2}{6N} x_3^2 \Big) \ll \Big(v+\mfrac 1v\Big) y \, e^{-c_1vy^2}
\end{equation*}
for some $c_1>0$ (see \cite[Appendix B]{cipra}).
As in \cite[Proposition~B.3]{cipra}, the contribution of $\varepsilon(z,y)$ to $(\Psi_F\sl_0 W_d)(iy)$ is $o(1)$ as $y\to\infty$.
For $x_3 = 0$ and $W_d x=(x_1',x_2',x_3')$ we have 
\begin{align*}
	x_1' &\equiv d \alpha^2 x_1 \pmod {N}, \\
	x_2' &\equiv \Big(1 + \mfrac{12N}{d} \gamma \beta \Big)x_2 \pmod{12}
\end{align*}
(in particular, $\bar\psi_1(W_d x)=0$ unless $d\in \{1,2,3,6\}$).
Thus for some $c_2$ the main term of \eqref{eq:theta-d-z-iy} equals
\begin{multline*}
	c_2 \, v^{\frac 12} \sum_{x_1\in \Z} \bar\psi(x_1) \exp\Big( -\mfrac{\pi v}{6N y^2} x_1^2 \Big) \sum_{x_2\in N\Z} \pmfrac{12}{x_2} e \Big( \mfrac{z}{24N} x_2^2 \Big) \\ 
	= 2c_2  \, \eta(Nz) \, v^{\frac 12} \sum_{h\bmod N} \bar\psi(h) \theta_1\Big(\mfrac{iv}{12Ny^2},h,N\Big),
\end{multline*}
where $\theta_1(\cdot,h,N)$ is as in \cite[Theorem 1.10(i)]{cipra}.
By the first two assertions of that theorem we have
\[
	v^{\frac 12}\sum_{h\bmod N} \bar\psi(h) \theta_1\Big(\mfrac{iv}{12Ny^2},h,N\Big) = \sqrt{\mfrac 6N} \, y \sum_{h\bmod N} \bar\psi(h) + O_N(\sqrt v).
\]
The latter sum is zero unless $\psi$ is principal.
In that case we have
\[
	(\Psi_F \sl_0 W_d)(iy) = c_3 \, y \int_{\mathcal D} v^{\frac 14} \bar{\eta(Nz)} F(z) \, d\mu + O(1)
\]
for some $c_3\in \C$; otherwise $\Psi_F$ is $O(1)$ at the cusps.
\end{proof}

If the spectral parameter of $F$ is not $i/4$ then $F$ is orthogonal to $\eta_N$.
Since a Maass form that is bounded at the cusps is a cusp form, we obtain the following proposition (recall \eqref{eq:theta-even}).

\begin{proposition} \label{prop:phi-F-maass}
Suppose that $F\in \mathcal{S}_{\frac 12}(N,\bar\psi\ptfrac N\bullet \chi, r)$, where $N\equiv 1\pmod{24}$ is squarefree and $r \neq i/4$. 
Then $\Psi_F$ is an even Maass cusp form in $\mathcal{S}_0(6N,\bar\psi^2,2r)$. 
\end{proposition}

With $G$ as in Theorem \ref{thm:shimura} let $F:=G\sl_{\frac 12}w_N$, where
\[
	w_N = \pMatrix{0}{-1/\sqrt N}{\sqrt N}{0}
\]
is the Fricke involution.
To obtain the $t=1$ case of Theorem \ref{thm:shimura} we will compute the Fourier expansion of 
\begin{equation}\label{phidef}
\Phi_F(w):=\Psi_F\(-1/6Nw\).
\end{equation}
Once we show that the coefficients of $\Phi_F$ satisfy the relation \eqref{eq:shim-coeff-relation}, Proposition~\ref{prop:phi-F-maass} completes the proof.
We first show that $F$ defined in this way satisfies the hypotheses of Proposition~\ref{prop:phi-F-maass}.

\begin{lemma} 
Suppose that $N\equiv 1\pmod{24}$. 
If $G\in \mathcal{S}_{\frac 12}(N,\psi\chi,r)$ and $F=G\sl_{\frac 12}w_N$, then
$F\in \mathcal{S}_{\frac 12}\big(N,\bar\psi \ptfrac N\bullet \chi, r\big)$.
\end{lemma}
\begin{proof}
Since the slash operator commutes with $\Delta_k$, it suffices to show that $F$ transforms with mutliplier $\bar\psi(\tfrac N\bullet)\chi$.
Writing $\gamma=\pmatrix abcd\in \Gamma_0(N)$ and $\gamma'=w_N\gamma w_N^{-1}=\pmatrix d{-c/N}{-bN}a$ we have
\[
	F|_\frac12\gamma=\bar\psi(d) \chi(\gamma') \, F.
\]
Thus it suffices to show that $\chi(\gamma')=\pmfrac Nd \chi(\gamma)$.
For this we argue as in \cite[Proposition~1.4]{shimura}.

Write $w_N=SV$, where
\[
	S=\pMatrix 0{-1}10, \quad V = \pMatrix {\sqrt N}00{1/\sqrt N}.
\]
Let $\tilde \eta(w):= y^{1/4}\eta(w)$ and let $g:=\tilde\eta\sl_{\frac 12}V$.
Then by \eqref{eq:chi-kronecker-symbol}, $g$ transforms on $\Gamma_0(N)$ with multiplier $\pfrac N\bullet \chi$.
We compute
\[
	\chi(w_N\gamma w_N^{-1}) \, \tilde\eta = \tilde\eta \sl_{\frac 12} w_N\gamma w_N^{-1} = \sqrt{-i} \, g \sl_{\frac 12} \gamma w_N^{-1} = \sqrt{-i} \, \pmfrac Nd \chi(\gamma) g \sl_{\frac 12} w_N^{-1} = \pmfrac Nd \chi(\gamma) \, \tilde\eta.
	\qedhere
\]
\end{proof}

We will determine the Fourier coefficients of $\Phi_F$ by computing its Mellin transform $\Omega(s)$ in two ways.
We have
\begin{equation} \label{eq:Omega-def}
	\Omega(s):=\int_0^\infty \Phi_F(iy)y^s\frac{dy}y=(6N)^{-s}\int_0^\infty \Psi_F\(i/y\)y^s\frac{dy}y.
\end{equation}
For $w$ purely imaginary, the sum in \eqref{eq:theta-def} simplifies to
\[
	\vartheta(z,iy) = \vartheta_1(z)\vartheta_2(z,y),
\] 
where $\vartheta_1(z)=2\eta(Nz)$ and
\begin{align*}
	\vartheta_2(z,y) 
	&= v^{\frac 12} \sum_{m,n\in \Z} \bar\psi(m) e(-umn) \exp\left(-\pi v\left(\frac{m^2}{6Ny^2} + 6Ny^2 n^2\right)\right) \\
	&= \frac{1}{y\sqrt{6N}} \sum_{m,n\in \Z} \bar\psi(m) \exp\left(-\frac{\pi}{6Ny^2v}|mz+n|^2\right).
\end{align*}
The latter equality follows from Poisson summation on $n$.

Setting $A:=\frac\pi{6Nv}|mz+n|^2$ for the moment, we find that 
\[\Omega(s)= 2(6N)^{-s-\frac 12} \sum_{m,n} \psi(m) \int_{\mathcal D} v^{\frac 14} \bar{\eta(Nz)} F(z) \int_0^\infty y^s e^{-A y^2}dy \, d\mu.
\]
For $\re(s)>1$ the  inner integral  evaluates to $\frac 12 A^{\frac{-s-1}{2}} \Gamma\pfrac{s+1}{2}$, so we have
\[\Omega(s)= \frac{\Gamma\pfrac{s+1}{2} }{(6N)^\frac s2\pi^\frac{s+1}2}\sum_{m,n} \psi(m) \int_{\mathcal D} v^{\frac 14} \bar{\eta(Nz)} F(z)
\(\frac{v}{|mz+n|^2}\)^\frac{s+1}2 \, d\mu.\]
Replacing $z$ by $w_N z$,  recalling that  $G=F|_\frac12w_N$ and using the fact that $\psi$ is even, we obtain
\[
	\Omega(s) = \frac{\sqrt{i}N^\frac14}{6^\frac s2\pi^\frac{s+1}2}\Gamma\pmfrac{s+1}2
	\sum_{m,n} \psi(m) \int_{\mathcal D} v^{\frac 14} \bar{\eta(z)} G(z) \pfrac{v}{|Nnz+m|^2}^{\frac{s+1}{2}}\, d\mu.
\]
For $(m,N)=1$, write $m=gd$ and $Nn=gc$ with $g=(m,n)=(m,Nn)$ so that
\[
\sum_{m,n} \psi(m) \pfrac{v}{|Nnz+m|^2}^\frac{s+1}2
	= L(s+1,\psi) \sum_{\gamma\in \Gamma_\infty \backslash \Gamma_0(N)} \psi(\gamma) \im(\gamma z)^\frac{s+1}2,
\]
where $\psi(\gamma)=\psi(d)$ for $\gamma=\pmatrix abcd$.
Thus
\begin{align*}
\Omega(s) 
	=  c(s) \sum_{\gamma\in \Gamma_\infty \backslash \Gamma_0(N)}
	\psi(\gamma) \int_{\mathcal D} v^{\frac 14}\bar{\eta(z)} G( z) \im (\gamma z)^{\frac{s+1}2} \, d\mu,
\end{align*}
where
\[
	c(s)=\frac{\sqrt{i}N^\frac14}{6^\frac s2\pi^\frac{s+1}2}\Gamma\pmfrac{s+1}2 L(s+1,\psi).
\]
Since $\psi(\gamma)v^{1/4}\bar{\eta(z)}G(z) = \im(\gamma z)^{1/4}\bar{\eta(\gamma z)}G(\gamma z)$ for $\gamma\in \Gamma$, we have
\begin{align*}
	\Omega(s) = c(s) \int_{\Gamma_\infty\backslash\H} v^{\frac{s+1}2+\frac14}\bar{\eta(z)} G(z) \, d\mu.
\end{align*}

Recall that $G$ has Fourier expansion
\[
	G(z) = \sum_{n\equiv 1(24)} a(n) W_{\frac {\sgn(n)}4, ir} \pmfrac{\pi |n|v}6 e\pmfrac{nu}{24}.
\]
Thus we have 
\[
	\Omega(s) = c(s)  \sum_{n=1}^\infty \pmfrac{12}{n} a(n^2) \int_0^\infty v^{\frac s2-\frac 54} W_{\frac 14, ir}\pmfrac{\pi n^2 v}{6} e^{-\frac{\pi n^2 }{12}v}\, dv.
\]
By \cite[(13.23.4) and (16.2.5)]{nist}, the integral  evaluates to 
\[
	\pfrac{6}{\pi n^2}^{\frac s2-\frac 14} \, \frac{\Gamma\(\frac s2+\frac 14+ir\)\Gamma\(\frac s2+\frac14-ir\)}{\Gamma\(\frac{s+1}2\)},
\]
so we conclude that
\[
	\Omega(s) = \sqrt{i}\pmfrac {\pi N}6^{\frac 14} \pi^{-s-\frac 12} \, \Gamma\(\tfrac s2+\tfrac 14+ir\)\Gamma\(\tfrac s2+\tfrac14-ir\) L(s+1,\psi) \sum_{n=1}^\infty \pmfrac {12}n \frac{a(n^2)}{n^{s-\frac 12}}.
\]

On the other hand, since  $\Psi_F$ is an even Maass cusp form, 
it  follows that $\Phi_F$ is also even.
Since $\Phi_F$ has eigenvalue $\frac 14+(2r)^2$, it has a Fourier expansion of the form
\[
	\Phi_F(w) = 2\sum_{n=1}^\infty b(n) W_{0,2ir}(4\pi n y) \cos(2\pi n x).
\]
By the definition \eqref{eq:Omega-def} of $\Omega(s)$ and \cite[(13.18.9) and (10.43.19)]{nist} we have
\begin{align*}
	\Omega(s) &= 4\sum_{n=1}^\infty \sqrt{n} \, b(n) \int_0^\infty K_{2ir}(2\pi n y) y^{s-\frac 12} \, dy \\
	&= \pi^{-s-\frac 12} \Gamma\(\tfrac s2+\tfrac 14+ir\) \Gamma\(\tfrac s2+\tfrac 14-ir\) \sum_{n=1}^\infty \frac{b(n)}{n^s}.
\end{align*}
So the coefficients $b(n)$ are given by the relation
\begin{equation} \label{eq:phi-F-coeffs}
	\sum_{n=1}^\infty \frac{b(n)}{n^s} = \sqrt{i}\pmfrac{\pi N}{6}^{\frac 14} \, L(s+1,\psi) \sum_{n=1}^\infty \pmfrac {12}n  \frac{a(n^2)}{n^{s-\frac 12}}.
\end{equation}
This proves Theorem \ref{thm:shimura} in the case $t=1$ since $b(n)$ is a constant multiple of the function $b_1(n)$ in \eqref{eq:shim-coeff-relation}.

\subsection{The case $t>1$}

For squarefree $t\equiv 1\pmod{24}$ with $t>1$ we argue as in Section 3 of \cite{niwa}.  
For a function $f$ on $\H$ define $f_t(\tau):=f(t\tau)$.
We apply Theorem \ref{thm:shimura} to 
\[
	G_t(\tau) = \sum_{n\equiv 1(24)} a(n/t) W_{\frac{\sgn(n)}4,ir}\pmfrac{\pi |n| y}{6} e\pmfrac{nx}{24} \in \mathcal{S}_{\frac 12}\left(Nt,\psi\ptfrac{t}{\bullet}\chi,r\right).
\]
The coefficients $c(n)$ of $S_1(G_t) \in \mathcal{S}_0(6Nt,\psi^2,2r)$ are given by
\[
	\sum_{n=1}^\infty \frac{c(n)}{n^s} = L\left(s+1,\psi\ptfrac t\bullet\right) \sum_{n=1}^\infty \pmfrac{12}n \frac{a(n^2/t)}{n^{s-\frac 12}} = L\left(s+1,\psi\ptfrac t\bullet\right)  t^{-s+\frac 12} \sum_{n=1}^\infty \pmfrac{12}n \frac{a(tn^2)}{n^{s-\frac 12}}.
\]
Thus $c(n)=0$ unless $t\mid n$, in which case $c(n)=\sqrt t \, b_t(n/t)$, where $b_t(n)$ are the coefficients of $S_t(G)$.
We conclude that $S_1(G_t) = \sqrt t \, [S_t(G)]_t$.
By a standard argument (see e.g. \cite[Section 3]{niwa}) we have 
\[
	S_t(G)\in \mathcal{S}_0(6N,\psi^2,2r).
\]
This completes the proof of Theorem \ref{thm:shimura}.


\section{Estimates for a \texorpdfstring{$K$}{K}-Bessel transform}\label{sec:kbessel}
This section contains uniform estimates for the $K$-Bessel transform $\check\phi(r)$ which are required 
in Sections \ref{sec:easyproof} and \ref{sec:hardproof}.
Recall that
\[
	\check\phi(r) = \ch \pi r \int_0^\infty K_{2ir} (u) \phi(u) \, \frac{du}u,
\]
where $\phi$ is a suitable test function (see \eqref{phiproperties}).
Given $a,x,T>0$ with
\[
	T \leq \mfrac x3 \quad \text{ and } \quad T \asymp x^{1-\delta}, \qquad 0<\delta<\mfrac12,
\]
we choose $\phi=\phi_{a,x,T}:[0,\infty)\to[0,1]$ to be a smooth function satisfying
\begin{enumerate}[\hspace{1.5em}(i)]\setlength\itemsep{.4em}
	\item $\phi(t)=1$ for $\mfrac{a}{2x} \leq t \leq \mfrac ax$,
	\item $\phi(t)=0$ for $t\leq \mfrac{a}{2x+2T}$ and $t\geq \mfrac{a}{x-T}$,
	\item $\phi'(t) \ll \left(\mfrac{a}{x-T} - \mfrac ax\right)^{-1} \ll \mfrac{x^2}{aT}$, and
	\item $\phi$ and $\phi'$ are piecewise monotone on a fixed number of intervals (whose number is independent of $a$, $x$, and $T$).
\end{enumerate}

In Theorem \ref{KTFmixedsign}, the function $\check\phi(r)$ is evaluated at the spectral parameters $r_j$ corresponding to 
the eigenfunctions $u_j$ as in \eqref{eq:uj-fix}.
In view of Corollary \ref{cor:shimura}, we  require estimates only for $r \geq 1$.
We will prove the following theorem.

\begin{theorem} \label{thm:phicheck}
Suppose that $a,x,T$, and $\phi=\phi_{a,x,T}$ are as above.
Then
\begin{equation} \label{eq:phi-est}
	\check\phi(r) \ll
	\begin{dcases}
		r^{-\frac 32} e^{-r/2} & \text{ if } \, 1\leq r\leq \mfrac{a}{8x}, \\
		r^{-1} & \text{ if } \, \max\big(\mfrac a{8x},1\big) \leq r\leq \mfrac ax, \\
		\min\Big( r^{-\frac 32}, r^{-\frac 52} \mfrac xT \Big) & \text{ if } \, r\geq \max\big(\mfrac ax, 1\big).
	\end{dcases}	
\end{equation}
\end{theorem}

To prove Theorem \ref{thm:phicheck} we require estimates for $K_{iv}(vz)$ which are uniform for $z\in (0,\infty)$ and $v\in [1,\infty)$.

\subsection{Uniform estimates for the $K$-Bessel function}
We estimate $K_{iv}(vz)$ in the following ranges as $v\to \infty$:
\begin{enumerate}[\hspace{1.5em}(A)]\setlength\itemsep{.4em}
 	\item the oscillatory range $0<z \leq 1-O(v^{-\frac23})$,
 	\item the transitional range $1-O(v^{-\frac 23}) \leq z \leq 1+O(v^{-\frac23})$, and
 	\item the decaying range $z \geq 1+O(v^{-\frac 23})$.
\end{enumerate}

Suppose that $c$ is a positive constant.
In the transitional range there is a significant ``bump'' in the $K$-Bessel function.
By \cite[(14) and (21)]{booker-strombergsson-then} we have
\begin{equation} \label{eq:kbessel-est-transition}
	e^{\frac{\pi v}2}K_{iv}(vz) \ll_c v^{-\frac 13} \qquad \text{ for }  z \geq 1-cv^{-\frac 23}.
\end{equation}
In the decaying range the $K$-Bessel function is positive and decreasing.
By \cite[(14)]{booker-strombergsson-then} we have
\begin{equation} \label{eq:kbessel-est-decay}
	e^{\frac{\pi v}2}K_{iv}(vz) \ll_c \frac{e^{-v \mu(z)}}{v^{\frac 12}(z^2-1)^{\frac 14}} \qquad \text{ for } z \geq 1+cv^{-\frac 23},
\end{equation}
where 
\[
	\mu(z):=\sqrt{z^2-1}-\arccos\Big(\mfrac1z\Big).
\]

The oscillatory range is much more delicate. 
Balogh \cite{balogh} gives a uniform asymptotic expansion for $K_{iv}(vz)$ in terms of the Airy function $\Ai$ and its derivative $\Ai'$. 
For $z\in (0,1)$ define
\[
	w(z) := \arccosh\pmfrac 1z - \sqrt{1-z^2}
\]
and define $\zeta$ and $\xi$ by
\[
	\mfrac 23 \zeta^{\frac 32} = -i \, w(z), \qquad \xi = v^{\frac 23} \zeta.
\]
Taking $m=1$ in equation (2) of \cite{balogh} we have
\begin{equation} \label{eq:balogh}
	e^{\frac{\pi v}{2}} K_{iv}(vz) = \frac{\pi\sqrt 2}{v^{\frac 13}} \pfrac{-\zeta}{1-z^2}^{\frac 14} \left\{ \Ai(\xi)\left[1+\frac{A_1(\zeta)}{v^2}\right] + \Ai'(\xi)\frac{B_0(\zeta)}{v^{\frac 43}} + \frac{e^{ivw(z)}}{1+|\xi|^{\frac 14}}O\left(v^{-3}\right) \right\},
\end{equation}
uniformly for $v\in [1,\infty)$, where
\begin{align}
	A_1(\zeta) &:= \frac{455}{10368\,w(z)^2} - \frac{7(3z^2+2)}{1728(1-z^2)^{\frac 32}w(z)} - \frac{81z^4+300z^2+4}{1152(1-z^2)^3}, \\
	B_0(\zeta) &:= \pmfrac 23^{\frac 13} \frac{e^{2\pi i/3}}{w(z)^{\frac 13}} \left( \frac{3z^2+2}{24(1-z^2)^{\frac 32}} - \frac{5}{72w(z)} \right).\label{B0def}
\end{align}
A  computation shows that $A_1(\zeta)$ and $B_0(\zeta)$ have finite limits as $z\to 0^+$ and as $z\to 1^-$, so both functions are $O(1)$ for $z\in(0,1)$.

Note that $\arg \xi=-\frac {\pi}3$.
In order to work on the real line, we apply \cite[(9.6.2-3)]{nist} to obtain
\begin{align*}
 \Ai(\xi) &= \tleg23^{1/6}\mfrac{i}{2^{3/2}} (vw(z))^{\frac 13} H_{\frac 13}^{(1)} (vw(z)), \\
  \Ai'(\xi) &= -\tleg 32^{ 1/6}\mfrac{i}{2^{3/2}} (vw(z))^{\frac 23} H_{\frac 23}^{(1)} (vw(z)),
\end{align*}
where $H_{\frac 13}^{(1)}$ and $H_{\frac 23}^{(1)}$ are Hankel functions of the first kind.
So we have
\begin{multline} \label{eq:balogh-hankel}
	e^{\frac{\pi v}{2}} K_{iv}(vz) = \mfrac{\pi}2 e^{2\pi i/3} \frac{w(z)^{\frac 12}}{(1-z^2)^{\frac 14}} H_{\frac 13}^{(1)}(vw(z))\left[1+\frac{A_1(\zeta)}{v^2}\right] \\
	+ \tleg32^{1/3}\mfrac{\pi}2 e^{-\pi i/3} \frac{w(z)^{\frac 56}}{(1-z^2)^{\frac 14}} H_{\frac 23}^{(1)}(vw(z)) \frac{ B_0(\zeta)}{v} 
	+O\left( \frac{v^{-\frac 72}}{(1-z^2)^{\frac 14}} \right).
\end{multline}

Since $w(z)\to\infty$ as $z\to0$ and $(1-z^2)^{-\frac14}\to\infty$ as $z\to 1$, we derive more convenient expressions for \eqref{eq:balogh-hankel} for $z$ in the intervals $(0,3/4]$ and $[3/16,1-cv^{-\frac 23})$.

\begin{proposition} \label{prop:kbessel-asy-small}
Suppose that $z\in (0, 3/4]$ and that $v\geq 1$. Then
\begin{equation} \label{eq:kbessel-asy-small}
	e^{\frac{\pi v}{2}} K_{iv}(vz) = e^{\pi i/4} \pmfrac{\pi}{2}^{\frac 12} \frac{e^{ivw(z)}}{v^{\frac 12}(1-z^2)^{\frac 14}} \left[ 1-i\frac{3z^2+2}{24v(1-z^2)^{\frac 32}} \right] + O\big( v^{-\frac 52} \big).
\end{equation}
\end{proposition}

\begin{proof}
Since $(1-z^2) \gg 1$ for $z\leq 3/4$,
the error term in \eqref{eq:balogh-hankel} is $\ll v^{-\frac 72}$.
By \cite[(10.17.5) and \S 10.17(iii)]{nist} we have
\begin{align*}
	H_{\frac 13}^{(1)}(vw(z)) &= \pmfrac{2}{\pi}^{\frac 12} e^{-5\pi i/12} \frac{e^{ivw(z)}}{(vw(z))^{\frac 12}} \left(1 - \frac{5i}{72vw(z)}\right) + O\left( (vw(z))^{-\frac 52} \right), \\
	H_{\frac 23}^{(1)}(vw(z)) &= \pmfrac{2}{\pi}^{\frac 12} e^{-7\pi i/12} \frac{e^{ivw(z)}}{(vw(z))^{\frac 12}} + O\left( (vw(z))^{-\frac 32} \right).
\end{align*}
In particular, this implies that
\[
	\frac{w(z)^{\frac 12}}{(1-z^2)^{\frac 14}} H_{\frac 13}^{(1)}(vw(z)) \frac{A_1(\zeta)}{v^2} \ll v^{-\frac 52},
\]
so by \eqref{eq:balogh-hankel} we have
\[
	e^{\frac{\pi v}{2}} K_{iv}(vz) = e^{\pi i/4} \pmfrac{\pi}{2}^{\frac 12} \frac{e^{ivw(z)}}{v^{\frac 12}(1-z^2)^{\frac 14}} \left[ 1 - \frac{5i}{72vw(z)} + e^{-7\pi i/6} \pmfrac 32^{\frac 13} w(z)^{\frac 13} \frac{B_0(\zeta)}{v} \right] + O\big( v^{-\frac 52} \big).
\]
Using \eqref{B0def}, we obtain \eqref{eq:kbessel-asy-small}.
\end{proof}

We require some notation for the next proposition.
Let $J_\nu(x)$ and $Y_\nu(x)$ denote the $J$ and $Y$-Bessel functions, and define 
\[
	M_{\nu}(x) = \sqrt{J_\nu^2(x)+Y_\nu^2(x)}.
\]
\begin{proposition} \label{prop:kbessel-asy-transition}
Suppose that $c>0$.  
Suppose that $v\geq 1$ and that $\frac 3{16}\leq z\leq 1-cv^{-\frac 23}$. Then
\begin{equation} \label{eq:kbessel-asy-transition}
	e^{\frac{\pi v}{2}} K_{iv}(vz) = 
	\mfrac{\pi}{2} e^{2\pi i/3} \frac{w(z)^{\frac 12}}{(1-z^2)^{\frac 14}} M_{\frac 13}(vw(z)) e^{i\theta_{\frac 13}(vw(z))} + O_c(v^{-4/3}),
\end{equation}
where
$\theta_{\frac 13}(x)$ is a real-valued continuous function satisfying
\begin{equation} \label{eq:theta-1/3-prime}
	\theta_{\frac 13}'(x) = \frac{2}{\pi x M_{\frac 13}^2(x)}.
\end{equation}
\end{proposition}

\begin{proof}
Since $(1-z^2) \gg_c v^{-\frac 23}$ for $z \leq 1-cv^{-\frac 23}$, the error term in \eqref{eq:balogh-hankel} is $\ll_c v^{-\frac {10}3}$.
The modulus and phase of $H_\alpha^{(1)}(x)$ are given by \cite[\S 10.18]{nist}
\[
	H_\alpha^{(1)}(x) = M_\alpha(x) e^{i\theta_\alpha(x)},
\]
where 
$M_\alpha^2(x)\theta_\alpha'(x)=2/\pi x$.
A straightforward computation shows that $w(z)\gg(1-z)^{\frac 32}$.
It follows that $vw(z) \gg_c 1$ for $z\leq 1-cv^{-\frac 23}$, so by
\cite[(10.18.17)]{nist} we obtain
\[
	M_\alpha(vw(z)) \ll_{\alpha,c} \frac{1}{(vw(z))^{\frac 12}}.
\]
This, together with \eqref{eq:balogh-hankel} and the fact that $A_1(\zeta)$ and $B_0(\zeta)$ are $O(1)$,
gives \eqref{eq:kbessel-asy-transition}.
\end{proof}

\subsection{Estimates for $\check\phi(r)$}
We treat each of the three ranges considered in Theorem~\ref{thm:phicheck}
separately in the following propositions.
We will make frequent use of an integral estimate which is an immediate corollary of the second mean value theorem for integrals.

\begin{lemma} \label{lem:mvt}
	Suppose that $f$ and $g$ are continuous functions on $[a,b]$ and that $g$ is piecewise monotonic on $M$ intervals. Then
	\begin{equation}
		\left|\int_a^b f(x) g(x) \, dx\right| \leq 2M \sup_{x\in [a,b]} \left|g(x)\right| \sup_{[\alpha,\beta]\subseteq[a,b]} \left| \int_\alpha^\beta f(x) \, dx \right|.
	\end{equation}
\end{lemma}

\begin{proposition} \label{prop:phi-check-small-r}
With the notation of Theorem~\ref{thm:phicheck},
suppose that $1\leq r\leq a/8x$. Then
	\begin{equation}
		\check\phi(r) \ll r^{-\frac 32} e^{-r/2}.
	\end{equation}
\end{proposition}

\begin{proof}
Since $r\leq a/8x$ and $T\leq x/3$, we have $\frac{a}{4(x+T)r} \geq \frac32$.
Taking $c=2^{-1/3}$ in \eqref{eq:kbessel-est-decay} gives
\begin{align*}
	\check\phi(r) = \ch\pi r \int_{\frac{a}{4(x+T)r}}^{\frac{a}{2(x-T)r}} K_{2ir}(2ry) \phi(2ry) \, \frac{dy}y
	\ll r^{-\frac 12} \int_{\frac 32}^{\infty} e^{-2r\mu(y)} \frac{dy}{y(y^2-1)^{\frac 14}}.
\end{align*}
Since $\mu'(y) = \sqrt{y^2-1}/y$ we have
\[
	\check\phi(r) \ll r^{-\frac 32} \int_{\frac 32}^{\infty} \left(-e^{-2r\mu(y)}\right)' \frac{dy}{(y^2-1)^{\frac 34}} \ll r^{-\frac 32} e^{-2r\mu(3/2)},
\]
which, together with $\mu(3/2)\approx .277$ proves the proposition.
\end{proof}

\begin{proposition} \label{prop:phi-check-transition}
With the notation of Theorem~\ref{thm:phicheck},
suppose that $\max(a/8x,1)\leq r\leq a/x$. Then
	\begin{equation}
		\check\phi(r) \ll r^{-1}.
	\end{equation}
\end{proposition}

Before proving Proposition \ref{prop:phi-check-transition}, we require a lemma describing the behavior of the function $M_{\frac 13}(x)$ in Proposition \ref{prop:kbessel-asy-transition}.

\begin{lemma} \label{lem:tilde-M-inc}
	The function $\tilde M(x) := x M_{\frac 13}^2(x)$ is increasing on $[0,\infty)$ with $\lim\limits_{x\to\infty} \tilde M(x) = \frac 2\pi$.
\end{lemma}

\begin{proof}
We will prove that $w(x):=\sqrt{\smash[b]{\tilde M(x)}}$ is increasing.
From \cite[(10.7.3-4), (10.18.17)]{nist} we have
\[
	\tilde M(0) = 0, \qquad 0\leq \tilde M(x) \leq \mfrac 2\pi, \quad \text{ and } \quad \lim_{x\to\infty} \tilde M(x) = \mfrac{2}{\pi}.
\]
It is therefore enough to show that  $w''(x)<0$ for all $x>0$.
In view of  the second order differential equation \cite[(10.18.14)]{nist} satisfied by $w$ it will suffice to prove that
\begin{equation} \label{eq:M''-ineq}
	\tilde M(x) > \frac{12x}{\pi\sqrt{36x^2+5}}.
\end{equation}

The inequality \eqref{eq:M''-ineq} can be proved numerically, using the expansions of $\tilde M(x)$ at $0$ and $\infty$.
By \cite[\S 10.18(iii)]{nist} we have
\begin{equation} \label{eq:tilde-M-geq}
	\tilde M(x) \geq \frac{2}{\pi} \left( 1 + \sum_{k=1}^n \frac{1\cdot 3\cdots (2k-1)}{2\cdot 4\cdots(2k)} \frac{(\frac49-1)(\frac49-9)\cdots(\frac49-(2k-1)^2)}{(2x)^{2k}} \right)
\end{equation}
for odd $n\geq 1$.
Taking $n=3$ in \eqref{eq:tilde-M-geq}, we verify numerically that \eqref{eq:M''-ineq} holds for $x>2.34$.

From 
\cite[(10.2.3)]{nist} we have
\[
	\tilde M(x) = \frac{4x}{3} \left( J_{\frac 13}^2(x) - J_{\frac 13}(x)J_{-\frac 13}(x) + J_{-\frac 13}^2(x) \right).
\] 
So by \cite[(10.8.3)]{nist}
it follows that the series for $\tilde M(x)$ at $0$ is alternating,
and that $\tilde M(x)$ is larger than the truncation of this series after the term with exponent $47/3$.
Moreover, this truncation is larger than $\frac{12x}{\pi\sqrt{36x^2+5}}$ for $x<2.45$.
The claim \eqref{eq:M''-ineq} follows.
\end{proof}

\begin{proof}[Proof of Proposition \ref{prop:phi-check-transition}]
Fix $c=\frac 12$ and write
\begin{align*}
	\check\phi(r) 
	&= \check\phi_1(r) + \check\phi_2(r) + \check\phi_3(r)  \\
	&= \left( \int_{\frac{a}{4(x+T)r}}^{1-cr^{-\frac23}} \!\! + \int_{1-cr^{-\frac 23}}^{1+cr^{-\frac 23}} + \int_{1+cr^{-\frac 23}}^\infty \right) \ch \pi r K_{2ir}(2ry) \phi(2ry) \frac{dy}{y}.
\end{align*}
We will show that $\check\phi_i(r)\ll r^{-1}$ for $i=1,2,3$.
Note that $\frac{a}{4(x+T)r}\geq \frac 3{16}$.

By Proposition \ref{prop:kbessel-asy-transition} we have 
\begin{equation} \label{eq:phi-1-check}
	\check\phi_1(r) \ll \Bigg| \int_{\frac{a}{4(x+T)r}}^{1-cr^{-\frac 23}} \!\! e^{i\theta_{\frac 13}(2rw(y))} \frac{M_{\frac 13}(2rw(y))w(y)^{\frac 12}\phi(2ry)}{y(1-y^2)^{\frac 14}} \, dy \Bigg| + 
	r^{-\frac 43}  \int_{\frac 3{16}}^{1} \phi(2ry) \frac{dy}{y}.
\end{equation}
The error term is $O(r^{-{4/3}})$.
Using \eqref{eq:theta-1/3-prime} and the fact that $w'(y)=-\sqrt{1-y^2}/y$, the first term in \eqref{eq:phi-1-check} equals
\begin{align*}
	\frac{\pi}{2(2r)^{\frac 32}} \Bigg| \int_{\frac{a}{4(x+T)r}}^{1-cr^{-\frac 23}} \left( e^{i\theta_{\frac 13}(2rw(y))} \right)' \ \frac{\left(2rw(y)M_{\frac 13}^2(2rw(y))\right)^{\frac 32}}{(1-y^2)^{\frac 34}} \phi(2ry)\, dy \Bigg|.
\end{align*}
Note that $w(y)$ is decreasing, so by Lemma \ref{lem:tilde-M-inc} the function $2rw(y)M_{\frac 13}^2(2rw(y))$ is decreasing.
We apply Lemma \ref{lem:mvt} three times; first with the decreasing function 
\[
	g(y)=\left(2rw(y)M_{\frac 13}(2rw(y))\right)^{\frac 32} \ll 1,
\]
next with the increasing function 
\[
	g(y)=(1-y^2)^{-\frac 34} \ll r^{\frac 12},
\]
and then with $g(y)=\phi(2ry)$.
We conclude that $\check\phi_1(r)\ll r^{-1}$.

For $\check\phi_2(r)$ we apply \eqref{eq:kbessel-est-transition} to obtain
\[
	\check\phi_2(r) \ll r^{-\frac 13} \int_{1-cr^{-\frac 23}}^{1+cr^{-\frac 23}} dy \ll r^{-1}.
\]
For $\check\phi_3(r)$ we argue as in the proof of Proposition \ref{prop:phi-check-small-r} to obtain
\[
	\check\phi_3(r) \ll r^{-\frac 32} \int_{1+cr^{-\frac 23}}^\infty \left(-e^{-2r\mu(y)}\right)' \frac{dy}{(y^2-1)^{\frac 34}}.
\]
For $y\geq 1+cr^{-\frac 23}$ we have $(y^2-1)^{\frac 34}\gg r^{-\frac 12}$, so $\check\phi_3(r)\ll r^{-1}$.
\end{proof}

\begin{proposition}\label{prop:big_r}
	Suppose that $r\geq \max(a/x,1)$. Then
	\begin{equation}
		\check\phi(r) \ll \min\left( r^{-\frac 32}, r^{-\frac 52}\frac{x}{T} \right).
	\end{equation}
\end{proposition}

\begin{proof}
Since $r\geq a/x$ and $T\leq x/3$, the interval on which $\phi(2ry)\neq 0$ is contained in $(0, 3/4]$.
So by Proposition \ref{prop:kbessel-asy-small} we have
\begin{equation} \label{eq:phicheck-I0-I1}
	\check\phi(r) \ll r^{-\frac 12} |I_0| + r^{-\frac 32} |I_1| + r^{-\frac 52} \int_{\frac{a}{4(x+T)r}}^{\frac{a}{2(x-T)r}} \frac{dy}{y},
\end{equation}
where
\[
	I_\alpha = \int_{\frac{a}{4(x+T)r}}^{\frac{a}{2(x-T)r}} e^{2irw(y)} \frac{\phi(2ry)(3y^2+2)^\alpha}{y(1-y^2)^{\frac 14+\frac 32\alpha}} \, dy.
\]
The third term in \eqref{eq:phicheck-I0-I1} equals
\[
	r^{-\frac 52}\log \mfrac{2(x+T)}{x-T} 
	\ll r^{-\frac 52}.
\]
For the other terms we use $w'(y)=-\sqrt{1-y^2}/y$ to write
\begin{equation} \label{eq:I-alpha-est}
	I_\alpha \ll r^{-1} \left| \int_{\frac{a}{4(x+T)r}}^{\frac{a}{2(x-T)r}} \left(e^{2irw(y)}\right)' \frac{\phi(2ry)(3y^2+2)^\alpha}{(1-y^2)^{\frac 34+\frac 32\alpha}} \, dy \right|.
\end{equation}
Since $y\leq 3/4$ we have 
\[
	\frac{(3y^2+2)^\alpha}{(1-y^2)^{\frac 34+\frac 32\alpha}} \ll 1,
\]
so by Lemma \ref{lem:mvt} we obtain $I_\alpha \ll r^{-1}$ and therefore $\check\phi(r)\ll r^{-\frac 32}$.

For large $r$, we obtain a better estimate by integrating by parts in \eqref{eq:I-alpha-est} for $\alpha=0$.
Since $\phi(2ry)=0$ at the limits of integration, we find that
\begin{equation}
	I_0 \ll r^{-1} \left| \int_{\frac{a}{4(x+T)r}}^{\frac{a}{2(x-T)r}} e^{2irw(y)} \frac{y\phi(2ry)}{(1-y^2)^{\frac 74}} \, dy \right| + r^{-1} \left| \int_{\frac{a}{4(x+T)r}}^{\frac{a}{2(x-T)r}} e^{2irw(y)} \frac{r\phi'(2ry)}{(1-y^2)^{\frac 34}} \, dy \right| =:J_1+J_2.
\end{equation}
As above, Lemma~\ref{lem:mvt} gives
\[
	J_1 \ll r^{-2} \left|\int_{\frac{a}{4(x+T)r}}^{\frac{a}{2(x-T)r}} \left(e^{2irw(y)}\right)' \frac{y^2\phi(2ry)}{(1-y^2)^{\frac 94}} \, dy \right| \ll r^{-2}.
\]
For $J_2$, we apply Lemma \ref{lem:mvt} with the estimates $\phi'(2ry)\ll x^2/aT$ and $y \ll a/rx$ to obtain
\[
	J_2 = r^{-1}\left| \int_{\frac{a}{4(x+T)r}}^{\frac{a}{2(x-T)r}} \left(e^{2irw(y)}\right)' \frac{y\phi'(2ry)}{(1-y^2)^{\frac 54}} \, dy\right| \ll r^{-2} \frac xT.
\]
So $I_0 \ll r^{-2}x/T$.
This, together with \eqref{eq:phicheck-I0-I1} and our estimate for $I_1$ above, gives \eqref{prop:big_r}.
\end{proof}


\section{Application to sums of Kloosterman sums}\label{sec:easyproof}
We are in a position to prove Theorem~\ref{thm:x-1/6},
which asserts that 
for $m>0$, $n<0$ we have
\begin{equation} \label{eq:x-1/6}
	\sum_{c\leq X} \frac{S(m,n,c, \chi)}{c} \ll \left(X^{\frac 16} + |mn|^{\frac 14}\right) |mn|^{\epsilon} \log X.
\end{equation}
This will follow from an estimate for dyadic sums.
\begin{proposition}\label{dyadicpropeasy}
If  $m>0$, $n<0$ and $x>4\pi\sqrt{\tilde m|\tilde n|}$ then
\begin{equation*}
	\sum_{x\leq c\leq 2x} \frac{S(m,n,c,\chi)}c\ll   \(x^{\frac 16}\log x + |mn|^{\frac 14}\)|mn|^\ep
\end{equation*}
\end{proposition}

To obtain Theorem~\ref{thm:x-1/6} from the proposition, we estimate using \eqref{eq:trivial_est}
to  see that the initial segment
$c\leq 4\pi\sqrt{\tilde m|\tilde n|}$ of  \eqref{eq:x-1/6} contributes $O(|mn|^{\frac 14+\epsilon})$.
We break the rest of the sum into 
dyadic pieces $x\leq c\leq 2x$ with $4\pi\sqrt{\tilde m|\tilde n|} < x\leq X/2$.
Estimating each of these with Proposition~\ref{dyadicpropeasy} and summing their contributions gives \eqref{eq:x-1/6}.

 \begin{proof}[Proof of Proposition \ref{dyadicpropeasy}]
Let 
\[a:=4\pi\sqrt{\tilde m|\tilde n|}, \qquad T:=x^\frac23,\]
and suppose that $x>a$.
We apply Theorem~\ref{KTFmixedsign} 
using a test function $\phi$ which satisfies conditions (i)--(iv) of Section~\ref{sec:kbessel}.
 Proposition~\ref{prop:weil-bound} and the mean value bound for the divisor function give
\begin{multline} \label{dyadic0easy}
\left|\sum_{c=1}^\infty \frac{S(m,n,c, \chi)}c\phi\leg ac-\sum_{x\leq c\leq 2x} \frac{S(m,n,c, \chi)}c\right|\\
\leq\sum_{\substack{x-T\leq c\leq x\\ 2x\leq c\leq 2x+2T}}\frac{\left|S(m,n,c, \chi)\right|}c  
 \ll \frac{T\log x}{\sqrt x}|mn|^{\epsilon} 
 \ll |mn|^{\epsilon}x^{\frac 16}\log x.
\end{multline}
Using Theorem   \ref{KTFmixedsign}, it will therefore suffice to obtain the stated estimate for 
 the quantity
\begin{equation}\label{eq:mn-coeff-sum-easy}
\sum_{c=1}^\infty \frac{S(m,n,c,\chi)}c\phi\leg ac= 8\sqrt{i}\sqrt{\tilde m |\tilde n|}\sum_{1<r_j} \frac{\overline{\rho_j(m)}\rho_j(n)}{\ch \pi r_j}\check\phi(r_j),
\end{equation}
where we have used Corollary~\ref{cor:shimura} and the fact that the eigenvalue $r_0=i/4$ makes no contribution since $n$ is negative.

We break the sum \eqref{eq:mn-coeff-sum-easy} into dyadic intervals $A\leq r_j\leq 2A$.
Using the Cauchy-Schwarz inequality together with Theorems~\ref{thm:mve} and 
\ref{thm:phicheck}
(recall that  $a/x<1 < r_j$)
we obtain
\begin{equation*}
\begin{aligned}
	\sqrt{\tilde m|\tilde n|}\sum_{A\leq r_j\leq 2A} &\left| \frac{\bar{\rho_j(m)}\rho_j(n)}{\ch \pi r_j} \check\phi(r_j) \right| \\
	&\ll\min\left(A^{-\frac 32},A^{-\frac 52}x^{\frac 13}\right)  \(\tilde m \sum_{A\leq r_j\leq 2A} \frac{|\rho_j(m)|^2}{\ch \pi r_j}\)^\frac12
	\(|\tilde n|\sum_{A\leq r_j\leq 2A} \frac{|\rho_j(n)|^2}{\ch \pi r_j}\)^\frac12
	\\
	&\ll \min \left(A^{-\frac 32},A^{-\frac 52}x^{\frac 13}\right)    \(A^\frac32+m^{\frac12+\ep}\)^\frac12 \(A^\frac52+|n|^{\frac12+\ep}A^\frac12\)^\frac12 \\
	&\ll \min\left(A^\frac12,A^{-\frac 12}x^{\frac 13}\right)
	 \left(1 + m^{\frac 14+\epsilon}A^{-\frac 34} + |n|^{\frac 14+\epsilon}A^{-1} + |mn|^{\frac 14+\epsilon}A^{-\frac 74}\right).
\end{aligned}
\end{equation*}
Summing the contribution from the dyadic intervals gives
\[
	\sqrt{\tilde m|\tilde n|}\sum_{1<r_j} \frac{\bar{\rho_j(m)}\rho_j(n)}{\ch \pi r_j} \check\phi(r_j) \ll   x^{\frac 16} + |mn|^{\frac 14+\epsilon},
\]
and Proposition~\ref{dyadicpropeasy} follows.
\end{proof}


\section{A second estimate for coefficients of Maass cusp forms}\label{sec:avgduke}

In the case $m=1$ we can improve the estimate of Proposition \ref{dyadicpropeasy} by using a second estimate for the sum of the Fourier coefficients of Maass cusp forms in $\mathcal{S}_\frac12(1,\chi)$.
The next theorem is an improvement on Theorem~\ref{thm:mve} only when $n$ is much larger than $x$.

\begin{theorem}\label{thm:avgduke}
Suppose that $\{u_j\}$ is an orthonormal basis of $\mathcal{S}_\frac12(1,\chi)$ with spectral parameters $r_j$ and coefficients $\rho_j(n)$ as in
\eqref{eq:uj-fix}.
Suppose that the $u_j$ are eigenforms of the Hecke operators $T_{p^2}$ for $p\nmid 6$.
If $24n-23$ is not divisible by $5^4$ or $7^4$ then
\begin{equation} \label{eq:est-fourier-coeffs}
	\sum_{0< r_j\leq x} \frac{\left|\rho_j(n)\right|^2}{\ch\pi r_j} \ll
	 	|n|^{-\frac 47+\epsilon} \, x^{5-\frac{\sgn n}2}.
\end{equation}
\end{theorem}

\begin{remark}
As the proof will show, the exponent $4$ in the assumption on $n$ in Theorem \ref{thm:avgduke} can be replaced by any positive integer $m$.
Increasing $m$ has the effect of increasing the implied constant in \eqref{eq:est-fourier-coeffs}.
\end{remark}

We require an average version of an estimate of Duke \cite[Theorem~5]{duke-half-integral}.
Let $\nu_\theta$ denote the weight $1/2$ multiplier for $\Gamma_0(4)$ defined in \eqref{eq:def-theta-mult}.

\begin{proposition} \label{prop:avg-duke}
Let $D$ be an even fundamental discriminant, let $N$ be a positive integer with $D\mid N$, and let 
\[(k,\nu)=\(\mfrac12,\,\nu_\theta\(\tfrac {|D|} \bullet\)\)\qquad\text{or}\qquad\(\mfrac32,\,\bar\nu_\theta\(\tfrac {|D|}\bullet\)\).\]
Suppose that $\{v_j\}$ is an orthonormal basis of $\mathcal{S}_k(N,\nu)$ with spectral parameters $r_j$ and coefficients $b_j(n)$.
If $n>0$ is squarefree then
\begin{equation} \label{eq:avg-duke}
	 \sum_{0\leq r_j\leq x} \frac{|b_j(n)|^2}{\ch\pi r_j} 
	\ll_{N,D} n^{-\frac 47+\ep} x^{5-k}.
\end{equation}
\end{proposition}

\begin{proof}[Sketch of proof]
Let $\phi,\hat\phi$ be as in the proof of Theorem 5 of \cite{duke-half-integral}.
Then $\hat\phi(r)>0$ and
\begin{equation} \label{eq:duke-phi-hat-asymp}
	\hat\phi(r) \sim \mfrac{1}{2\pi^2} |r|^{k-5} \qquad \text{as } |r|\to\infty.
\end{equation}
Set 
\[V_1(n,n):=n \sum_{r_j} \frac{|b_j(n)|^2}{\ch\pi r_j} \hat\phi(r_j).\]
By Theorem 2 of \cite{duke-half-integral} 
we have $|V_1(n,n) +V_2(n,n)|\ll |S_N| + |V_3(n,n)|$,
where $S_N$ is defined in the proof of Theorem 5 and $V_2,V_3$ are defined in Section~3 of that paper.
Since $V_1(n,n)$ and $V_2(n,n)$ are visibly non-negative, we have
\begin{equation}\label{eq:duke_ktf}
	V_1(n,n) \ll |S_N| + |V_3(n,n)|.
\end{equation}

The terms $S_N$ and $V_3(n,n)$ are estimated by averaging over the level. 
For $P>(4\log 2n)^2$ let
\[
	\bar Q = \left\{ pN : p \text{ prime, } P<p\leq 2P, \, p\nmid 2n \right\}.
\]
Summing \eqref{eq:duke_ktf} gives
\begin{equation} \label{eq:V-1-avg-level}
	\sum_{Q\in \bar Q} V_1^{(Q)}(n,n) \ll \sum_{Q\in \bar Q} |S_Q| + \sum_{Q\in \bar Q} |V_3^{(Q)}(n,n)|,
\end{equation}
where $V_1^{(Q)}$, $S_Q$, and $V_3^{(Q)}$ are the analogues of $V_1$, $S_N$, and $V_3$ for $\Gamma_0(Q)$.

For each $Q\in \bar Q$,   the functions $\{[\Gamma_0(N):\Gamma_0(Q)]^{-1/2}u_j(\tau)\}$ form an orthonormal subset of $\mathcal{S}_k(Q,\nu)$.
It follows that
\[
	V_1^{(Q)}(n,n) \geq \frac{n}{[\Gamma_0(N):\Gamma_0(Q)]} \sum_{r_j} \frac{|b_j(n)|^2}{\ch\pi r_j} \hat\phi(r_j) = \frac{V_1(n,n)}{[\Gamma_0(N):\Gamma_0(Q)]}.
\]
Since $[\Gamma_0(N):\Gamma_0(Q)]\leq p+1\ll P$ we find that
\[
	V_1(n,n) \ll P \, V_1^{(Q)}(n,n) \quad \text{ for all }Q\in \bar Q.
\]
Since $|\bar Q|\asymp P/\log P$ we conclude that
\begin{equation} \label{eq:V-1-sum-V-1-Q}
	V_1(n,n) \ll \log P \sum_{Q\in \bar Q} V_1^{(Q)}(n,n).
\end{equation}

In the proof of Theorem 5 of \cite{duke-half-integral}, Duke gives the estimates
\begin{equation} \label{eq:V-3-S-Q-est}
	\sum_{Q\in \bar Q} |V_3^{(Q)}(n,n)| \ll_{N,D} n^{\frac 37+\epsilon} \quad \text{ and } \quad \sum_{Q\in \bar Q} |S_Q| \ll_{N,D} \big( (n/P)^{\frac 12} + (nP)^{\frac 38} \big)n^{\epsilon}
\end{equation}
which follow from work of Iwaniec \cite{iwaniec-fourier-coefficients}.
By \eqref{eq:V-1-avg-level}, \eqref{eq:V-1-sum-V-1-Q}, and \eqref{eq:V-3-S-Q-est} we conclude that
\[
	V_1(n,n) \ll_{N,D} \log P \big( (n/P)^{\frac 12} + (nP)^{\frac 38} + n^{\frac 37} \big)n^\epsilon.
\]
Choosing $P=n^{1/7}$, we find that $V_1(n,n)\ll_{N,D} n^{3/7+\epsilon}$.
By \eqref{eq:duke-phi-hat-asymp} we have
\[
	n\sum_{0\leq r_j\leq x} \frac{|b_j(n)|^2}{\ch\pi r_j} 
	\ll n\sum_{0\leq r_j \leq x} r_j^{5-k} \frac{|b_j(n)|^2}{\ch\pi r_j} \hat\phi(r_j) 
	\ll x^{5-k} \, V_1(n,n),
\]
from which \eqref{eq:avg-duke} follows.
\end{proof}

We turn to the proof of Theorem~\ref{thm:avgduke}.
\begin{proof}[Proof of Theorem~\ref{thm:avgduke}]
Set $\alpha := [\Gamma_0(1):\Gamma_0(24,24)]^{1/2}$.
 With $\{u_j\}$ as in the hypotheses, and recalling \eqref{eq:L-k-norm}, 
define 
\[
\begin{aligned}
v_j(\tau)&:=\tfrac1\alpha u_j(24\tau)=\sum_{n\equiv 1(24)} b_j(n)  W_{\frac 14\sgn n ,ir_j}(4\pi|n|y)e(nx),\\
v_j'(\tau)&:=\tfrac1\alpha \(r_j^2+\tfrac1{16}\)^{-\frac12}\bar{L_\frac12 u_j}(24\tau)=\sum_{n\equiv 23(24)} b_j'(n)  W_{\frac 34\sgn n ,ir_j}(4\pi|n|y)e(nx).
\end{aligned}
\]
By \eqref{eq:576-theta} we have 
\[v_j\in \mathcal{S}_{\frac12}\(576,\nu_\theta\ptfrac{12}{\bullet}\),\qquad
v_j'\in \mathcal{S}_{\frac32}\(576,\bar{\nu_\theta}\ptfrac{12}{\bullet}\).
\]
From  \eqref{eq:inner_prod_def}  and \eqref{eq:L-k-norm} we see that 
 $\{v_j\}$, $\{v_j'\}$ are orthonormal sets which can be extended to orthonormal bases for 
 these spaces. 
Using \eqref{eq:lower_coeff} and \eqref{eq:conj-coeff} and comparing Fourier expansions, we find  for positive $n$ that
\[ 
b_j(n)=\mfrac1\alpha \, \rho_j\pmfrac{n+23}{24},\qquad 
b_j'(n)=\mfrac1\alpha \, \(r_j^2+\tfrac1{16}\)^{-\frac12} \bar{\rho_j\pmfrac{-n+23}{24}}.
 \]
It follows from Proposition~\ref{prop:avg-duke} that for squarefree $n$ we have
\begin{equation} \label{eq:est-fourier-coeffs-sf}
	\sum_{0< r_j\leq x} \frac{\left|\rho_j\pfrac{n+23}{24}\right|^2}{\ch\pi r_j} \ll  
	|n|^{-\frac 47+\epsilon} \, x^{5-\frac{\sgn n}2}.
\end{equation}

To establish \eqref{eq:est-fourier-coeffs-sf} for 
  non-squarefree $n$ we use the assumption that the $\{u_j\}$ are Hecke eigenforms.
Recall the definition of the lift $S_t$ from Theorem~\ref{thm:shimura}.
For each $j$ there exists some squarefree positive $t\equiv 1\pmod{24}$ for which $S_t(u_j)\neq 0$.
Denote by $\lambda_j(p)$ the eigenvalue of $u_j$ under $T_{p^2}$.  From Corollary~\ref{cor:shim_hecke_commute}
it follows that  $\big(\tfrac{12}{p}\big)\lambda_j(p)$ is a Hecke eigenvalue of the weight zero  Maass cusp form $S_t(u_j)$.
For these eigenvalues we have the estimate
\[|\lambda_j(p)|\leq p^{\frac 7{64}} + p^{-\frac 7{64}} \]
due to Kim and Sarnak \cite[Appendix 2]{kim-sarnak}.

The Hecke action \eqref{eq:hecke_action} gives
\begin{equation} \label{eq:hecke-2}
	p\, \rho_j\pmfrac{p^2n+23}{24}=\lambda_j(p)\rho_j\pmfrac{n+23}{24}-p^{-\frac12}\pmfrac{12n}p\rho_j\pmfrac{n+23}{24}-p^{-1}\rho_j\pmfrac{n/p^2+23}{24}.
\end{equation}
Suppose that $p^2 \nmid n$. Then
\[
	\left|\rho_j\pmfrac{p^2n+23}{24}\right| 
	\leq \(p^{-\frac{57}{64}} + p^{-\frac{71}{64}} +p^{-\frac32}\)  \left|\rho_j\pmfrac{n+23}{24}\right|
	\leq 
	\begin{cases}
		1.25 \, p^{-\frac 47} & \text{ if }p=5\text{ or }7, \\
		p^{-\frac 47} & \text{ if }p\geq 11.
	\end{cases}
\]
Thus
\eqref{eq:est-fourier-coeffs-sf} holds whenever $n$ is not divisible by $p^4$ for any $p$.

To treat the remaining cases, assume that $p\geq 11$ and that for $p^2\nmid n$ we have
\[
	\left|\rho_j\pmfrac{p^{2\ell}n+23}{24}\right| \leq p^{-\frac{4\ell}7} \left|\rho_j\pmfrac{n+23}{24}\right|, \qquad \ell\leq m-1.
\]
Then for $m\geq 2$, the relation \eqref{eq:hecke-2} gives
\begin{align*}
	\left|\rho_j\pmfrac{p^{2m}n+23}{24}\right| 
	&\leq \left(p^{-\frac{57}{64}} + p^{-\frac{71}{64}} \right)\left|\rho_j\pmfrac{p^{2m-2}n+23}{24}\right|+p^{-2}\left|\rho_j\pmfrac{p^{2m-4}n+23}{24}\right| \\
	&\leq \left[ \left(p^{-\frac{57}{64}}+ p^{-\frac{71}{64}}\right) p^{\frac 47} + p^{-2+\frac 87} \right] \, p^{-\frac{4m}{7}} \left|\rho_j\pmfrac{n+23}{24}\right|.
\end{align*}
For $p\geq 11$ the quantity in brackets is $<1$.
It follows that \eqref{eq:est-fourier-coeffs-sf} holds for all $n$ which are not divisible by $5^4$ or $7^4$.
\end{proof}


\section{Sums of Kloosterman sums in the case \texorpdfstring{$m=1$}{m=1}} \label{sec:hardproof}

Throughout this section, $n$ will denote a negative integer.
In our application to the error term in Rademacher's formula
we  need a estimate for   sums of the  Kloosterman sums
$S(1, n, c, \chi)$ which   improves the bound of Theorem~\ref{thm:x-1/6} with respect to $n$.
We will also make the assumption throughout that
\begin{equation} \label{eq:n-5-7-11}
	24n-23 \text{ is not divisible by $5^4$ or $7^4$ }
\end{equation}
so that we may apply Theorem \ref{thm:avgduke} (see the remark after that theorem).

\begin{theorem}\label{sumofks}
For $0<\delta<1/2$ and $n<0$ satisfying \eqref{eq:n-5-7-11} we have
\begin{equation}\label{eq:sumofks}
\sum_{c\leq X}\frac{S(1, n, c, \chi)}{c}\ll_{\delta} |n|^{\frac{13}{56}+\ep}X^{\frac 34\delta}+\(|n|^{\frac{41}{168}+\ep}+X^{\frac 12-\delta}\) \log X.
\end{equation}
\end{theorem}

Theorem~\ref{sumofks} will follow from an estimate for  dyadic sums.
\begin{proposition}\label{dyadicprop}
Suppose that $0<\delta<1/2$ and that $n<0$ satisfies \eqref{eq:n-5-7-11}.  Then
\begin{equation} \label{dyadic1}
	\sum_{x\leq c\leq 2x} \frac{S(1,n,c,\chi)}c\ll_{\delta} |n|^{\frac{41}{28}+\ep}x^{-\frac52} +|n|^{\frac{13}{56}+\ep} x^{\frac 34\delta}
	+ x^{\frac 12-\delta}\log x.
\end{equation}
\end{proposition}

\begin{proof}[Deduction of Theorem~\ref{sumofks} from Proposition~\ref{dyadicprop}]
We break the sum \eqref{eq:sumofks} into 
an  initial segment corresponding to $c\leq |n|^\alpha$ and dyadic intervals of the form \eqref{dyadic1} with $x\geq |n|^\alpha$.
Estimating the initial segment using Lehmer's bound \eqref{eq:super_weil} and applying Proposition \ref{dyadicprop} to each of the dyadic intervals, 
we find that
\[\sum_{c\leq X}\frac{S(1, n, c, \chi)}{c}\ll_{\delta, \alpha} |n|^{\frac\alpha2+\epsilon}+|n|^{\frac{13}{56}+\ep}X^{\frac34\delta}+
\(|n|^{\frac {41}{28}-\frac52\alpha+\ep}+X^{\frac 12-\delta}\)\log X.\]
 Theorem~\ref{sumofks} follows upon 
setting $\alpha=41/84$.
\end{proof}

\begin{proof}[Proof of Proposition~\ref{dyadicprop}]
Let $T$ satisfy 
 \begin{equation} \label{Tdef}
 	T \asymp x^{1-\delta} \quad \text{ and } \quad T\leq \frac x3,
 \end{equation}
where $0<\delta<1/2$ is a parameter to be chosen later, and set
 \[
 	a:= 4\pi \sqrt{|\tilde n|}.
 \]
Arguing as in \eqref{dyadic0easy} (using Lehmer's bound \eqref{eq:super_weil} to remove the dependence on $n$) 
we have
\begin{equation} \label{dyadic0}
\left|\sum_{c=1}^\infty \frac{S(1,n,c,\chi)}c\phi\leg ac-\sum_{x\leq c\leq 2x} \frac{S(1,n,c,\chi)}c\right|
\ll\delta x^{\frac12-\delta} \log x.
\end{equation}
Let $\{u_j\}$ be an orthonormal basis for $\mathcal S_\frac12(1,\chi)$ as in \eqref{eq:uj-fix} which consists of eigenforms for the Hecke operators
$T_{p^2}$ with $p\nmid 6$.
Theorem \ref{KTFmixedsign}   gives
\begin{equation}\label{ktf-kloos}\sum_{c=1}^\infty \frac{S(1,n,c,\chi)}c\phi\leg ac=8\sqrt{i}\sqrt{|\tilde n|}\sum_{r_j} \frac{\overline{\rho_j(1)}\rho_j(n)}{\ch \pi r_j}\check\phi(r_j).
\end{equation}
We break the sum on $r_j$ into   ranges corresponding to the three ranges of the $K$-Bessel function:
\begin{enumerate}[\hspace{1.5em}(i)]\setlength\itemsep{.4em}
\item $r_j\leq \mfrac a{8x}$,
\item $\mfrac a{8x}<r_j<\mfrac ax$,
\item $ r_j\geq \mfrac ax$.
\end{enumerate}
We may restrict our attention to $r_j>1$ by  Corollary~\ref{cor:shimura} (the eigenvalue $r_0=i/4$ does not contribute).

For the first range, 
Theorem~\ref{thm:phicheck} gives
gives
$\check\phi(r_j) \ll r_j^{-3/2} e^{-r_j/2}$.
By Theorem~\ref{thm:avgduke} 
we have
\[\frac{\rho_j(1)}{\sqrt{\ch \pi r_j}}\ll r_j^\frac94,\qquad
\frac{\rho_j(n)}{\sqrt{\ch \pi r_j}}\ll |n|^{-\frac27+\epsilon}r_j^{\frac{11}{4}}.\]
Combining these estimates gives
\begin{equation}\label{smallrange}
\sqrt {|\tilde n|}\sum_{1<r_j\leq\frac{a}{8x}} \left|\frac{\overline{\rho_j(1)}\rho_j(n)}{\ch \pi r_j}\check\phi(r_j)\right|
\ll |n|^{\frac 3{14}+\epsilon} \sum_{1<r_j\leq\frac{a}{8x}} r_j^{\frac72}e^{-r_j/2}\ll |n|^{\frac{3}{14}+\epsilon},
\end{equation}
where  we have used \eqref{eq:weyl_law} to conclude that the latter sum over $r_j$ is $O(1)$.

For the other ranges we need the  mean value estimates of Theorem~\ref{thm:mve}:
\begin{gather}
\label{meanvalpos}
\sum_{1<r_j\leq x}\frac{\left|\rho_j(1)\right|^2}{\ch\pi r_j}\ll x^{\frac 32}, \\
\label{meanvalneg}
|\tilde n|\sum_{1<r_j\leq x}\frac{\left|\rho_j(n)\right|^2}{\ch\pi r_j}\ll x^{\frac 52} + |n|^{\frac 12+\ep} x^{\frac 12} ,
\end{gather}
as well as the average version of Duke's estimate (Theorem~\ref{thm:avgduke}):
\begin{equation}\label{averageDuke}
|\tilde n|\sum_{r_j\leq x}\frac{\left|\rho_j(n)\right|^2}{\ch\pi r_j}\ll |n|^{\frac 37+\ep}x^{\frac{11}2}.
\end{equation}
Using \eqref{meanvalpos} and \eqref{meanvalneg} with the Cauchy-Schwarz inequality, we find that
\begin{equation*} 
\sqrt {|\tilde n|}\sum_{r_j\leq x} \frac{|\overline{\rho_j(1)}\rho_j(n)|}{\ch \pi r_j}\ll x^{\frac 34}\(x^{\frac 52}+|n|^{\frac 12+\ep}x^{\frac 12}\)^{\frac 12}\ll x^2+|n|^{\frac 14+\ep}x.
\end{equation*}
Using \eqref{meanvalpos} and \eqref{averageDuke} we obtain
\begin{equation}\label{CS2}
\sqrt {|\tilde n|}\sum_{r_j\leq x} \frac{|\overline{\rho_j(1)}\rho_j(n)|}{\ch \pi r_j}\ll x^{\frac 34}\(|n|^{\frac 37+\ep} x^{\frac {11}2}\)^{\frac 12}\ll |n|^{\frac 3{14}+\ep}x^{\frac 72}.
\end{equation}

In the second range  
Theorem~\ref{thm:phicheck} gives
$\check\phi(r_j)\ll r_j^{-1}$. 
It follows from \eqref{CS2} (assuming as we may that $a/x\geq 1$) that 
\begin{equation}
\label{medrange}
\begin{aligned}
\sqrt {|\tilde n|}\sum_{\frac{a}{8x}< r_j<\frac{a}{x}} \left|\frac{\overline{\rho_j(1)}\rho_j(n)}{\ch \pi r_j}\check\phi(r_j)\right|&\ll
|n|^{\frac 3{14}+\epsilon} \pfrac{a}{x}^{\frac 52} \ll |n|^{\frac{41}{28}+\epsilon}x^{-\frac 52}.
\end{aligned}
\end{equation}

In the third range Theorem~\ref{thm:phicheck} gives
\[
	\check\phi(r_j)\ll \min\(r_j^{-\frac32}, \ r_j^{-\frac52}\frac xT\).
\]
We use  dyadic sums 
corresponding to intervals $A\leq r_j\leq 2A$ with  $A\geq \max\big(\frac ax,1\big)$.
For such a sum we have 
\begin{equation*}
\begin{aligned}
\sqrt {|\tilde n|}\sum_{A\leq r_j\leq 2A} \left|\frac{\overline{\rho_j(1)}\rho_j(n)}{\ch \pi r_j}\check\phi(r_j)\right|&\ll
\min\(A^{-\frac 32}, \ A^{-\frac 52} \, \frac xT\)\min\(A^2+|n|^{\frac 14+\ep}A,\ |n|^{\frac 3{14}+\ep} A^{\frac 72}\)\\
&\ll \min\(A^\frac12,A^{-\frac12}\frac{x}{T}\)\(1+|n|^{\frac{13}{56}+\ep}A^{\frac 14}\)\\
&\ll |n|^{\frac{13}{56}+\ep}\min\(A^{\frac 34}, A^{-\frac14}\frac xT\),
\end{aligned}
\end{equation*}
where we have used the fact that for positive $B$, $C$, and $D$ we have 
\[
	\min(B, C+D)\leq \min(B, C)+\min(C, D) \qquad \text{and} \qquad \min(B, C)\leq \sqrt{BC}.
\]

Combining the dyadic sums, we find that 
\begin{equation}
\label{bigrange}
\sqrt {|\tilde n|}\sum_{ r_j\geq\frac{a}{x}} \left|\frac{\overline{\rho_j(1)}\rho_j(n)}{\ch \pi r_j}\check\phi(r_j)\right|\ll
|n|^{\frac {13}{56}+\ep}\leg x T^{\frac 34} \ll_\delta |n|^{\frac{13}{56}+\ep} x^{\frac 34\delta}.
\end{equation}
Using \eqref{ktf-kloos} with \eqref{smallrange}, \eqref{medrange}, and \eqref{bigrange} we obtain
\[\sum_{c=1}^\infty \frac{S(1,n,c,\chi)}c\check\phi\leg ac\ll_\delta  |n|^{\frac {41}{28}+\ep}x^{-\frac 52}+|n|^{\frac{13}{56}+\ep} x^{\frac 34\delta}.\]
Proposition~\ref{dyadicprop} follows from this estimate together with   \eqref{dyadic0}.
\end{proof}


\section{Application to the remainder term in Rademacher's formula} \label{sec:partitionproof}
 
 In the final section, we use Theorem~\ref{sumofks} to  prove Theorems~\ref{raderror} and \ref{raderror1}.
Recall that we wish to bound $R(n, N)$, where
\begin{equation*}  
	p(n) = \frac{2\pi}{(24n-1)^{\frac34}} \sum_{c=1}^N \frac{A_c(n)}{c} I_{\frac32}\pfrac{\pi\sqrt{24n-1}}{6c}+R(n, N),
\end{equation*}
and
\[A_c(n) = \sqrt{-i}\, S(1,1-n,c,\chi).\]
To match the notation of the previous sections, we  assume that $n< 0$ satisfies \eqref{eq:n-5-7-11} and  we provide a bound for
\begin{equation}\label{eq:Rn_neg}
R(1-n,N)=\frac{2\pi\sqrt{-i}}{24^{\frac34}}|\tilde n|^{-\frac34}\sum_{c>N}\frac{S(1, n, c, \chi)}{c}I_{\frac32}
\(\frac ac\),
\end{equation}
where
\begin{equation} \label{eq:Rn-a-def}
	a:= \sqrt{\mfrac23} \,\pi\sqrt{|\tilde n|}.
\end{equation}

We  will apply partial summation to \eqref{eq:Rn_neg}.
For fixed $\nu,M>0$ and $0\leq z\leq M$, the asymptotic formula \cite[(10.30.1)]{nist} gives
\begin{equation} \label{besselsmallz}
I_\nu(z)\ll_{\nu,M}z^\nu.
\end{equation}
We have a straightforward lemma.
\begin{lemma}  \label{lem:bessel}
Suppose that  $a>0$ and that  $\alpha>0$.  
Then for $\frac ta \geq \alpha$ we have
\[\big(I_{\frac32}(a/t)\big)'\ll_\alpha a^\frac32 t^{-\frac52}.\]
\end{lemma}

\begin{proof}
By \cite[(10.29.1)]{nist} we have
\[I_\frac32'(t)=\tfrac12\big(I_\frac12(t)+I_\frac52(t)\big).\]
It follows that 
\begin{equation} \label{eq:fa-prime}
\big(I_{\frac32}(a/t)\big)'=-\frac{a}{2t^2}\big(I_\frac12(a/t)+I_\frac52(a/t)\big).
\end{equation}
For $\frac ta \geq \alpha$, equations \eqref{besselsmallz} and \eqref{eq:fa-prime} give
\[
	\big(I_{\frac32}(a/t)\big)'\ll_\alpha a^\frac32 t^{-\frac52} + a^\frac72 t^{-\frac92},
\]
and the lemma follows.
\end{proof}
\begin{proof}[Proof of Theorems~\ref{raderror} and \ref{raderror1}]
Let $a$ be as in \eqref{eq:Rn-a-def}.
Then
\begin{equation*}
R(1-n,N)\ll |\tilde n|^{-\frac34}\left|\sum_{c>N}\frac{S(1, n, c, \chi)}{c} I_{\frac32}\(\frac ac\)\right|.
\end{equation*}
Let
\[S(n,X):= \sum_{c\leq X} \frac{S(1, n, c, \chi)}c,\]
so that Theorem~\ref{sumofks} (for $0<\delta<1/2$) gives
\[ 
	S(n, X)\ll_{\delta} |n|^{\frac{13}{56}+\ep}X^{\frac 34\delta}+\(|n|^{\frac{41}{168}+\ep}+X^{\frac 12-\delta}\) \log X.
\]
By partial summation using \eqref{besselsmallz}  we have
\begin{equation*}
\sum_{c>N}\frac{S(1, n, c, \chi)}{c}  I_{\frac32}\(\frac ac\)=-S(n,N)I_{\frac32}\(a/N\)-\int_N^\infty S(n,t)\big(I_{\frac32}(a/t)\big)'\, dt.
\end{equation*}

Let $\alpha>0$ be fixed and take 
\[N=\alpha |n|^{\frac12+\beta}\]
where $\beta\in [0, 1/2]$ is to be chosen.
By \eqref{besselsmallz} we have
\[
	I_{\frac32}\(a/N\)\ll_\alpha |n|^\frac34N^{-\frac32}
\]
and by Lemma \ref{lem:bessel} we have
\[\big(I_{\frac32}(a/t)\big)'\ll_\alpha |n|^{\frac34}t^{-\frac52}\qquad \text{for $t\geq N$.}\]
Thus
\[
	S(n, N)I_{\frac32}\(a/N\)\ll_{\delta, \alpha} |n|^{\frac34}N^{-\frac32}\(|n|^{\frac{13}{56} }N^{\frac 34\delta}+ |n|^{\frac{41}{168} }+N^{\frac 12-\delta}\)|n|^\ep,
\]
and the same bound holds for the integral term. 
 Therefore
\begin{equation}\label{eq:last}
R(1-n,\alpha |n|^{\frac12+\beta})\ll_{\delta, \alpha} 
|n|^{(\frac12+\beta)(\frac34\delta-\frac32)+\frac{13}{56}+\ep}+|n|^{-\frac32(\frac12+\beta)+\frac{41}{168}+\ep}+|n|^{-(\frac12+\beta)(1+\delta)+\ep}.
\end{equation}

When  $\beta=0$, the estimate \eqref{eq:last} becomes
\[R(1-n,\alpha |n|^{\frac12})\ll_{\delta, \alpha} 
|n|^{\frac38\delta-\frac{29}{56}+\ep}+|n|^{-\frac{85}{168}+\ep}+|n|^{-\frac12-\frac12\delta+\ep}.
\]
For any choice of $\delta\in [\frac1{84}, \frac2{63}]$ this gives
\[R(1-n,\alpha |n|^{\frac12})\ll_{\alpha} |n|^{-\frac12-\frac{1}{168}+\ep}.\]
Theorem~\ref{raderror} follows after replacing $n$ by $1-n$.

To optimize, we choose $\beta=\frac{5}{252}$ and $\delta=\frac{4}{131}$, which makes all three exponents in \eqref{eq:last} equal 
to $-\frac12-\frac{1}{28}$.  
This gives Theorem~\ref{raderror1}.
\end{proof}

\nocite{nist-book}

\bibliographystyle{alphanum-2}
\bibliography{kloosterman-bib}

\end{document}